\documentclass[11pt,oneside,reqno]{amsart}

\usepackage{geometry}
\usepackage{amsmath,amssymb,amsfonts,amsthm}
\usepackage{mathtools}
\usepackage{enumerate}
\usepackage{color}
\usepackage{xcolor}
\usepackage{comment}
\usepackage{graphicx} 
\usepackage{url}
\usepackage{hyperref}
\hypersetup{
  colorlinks   = true, 
  urlcolor     = blue, 
  linkcolor    = red, 
  citecolor   = red 
}
\usepackage[normalem]{ulem}
\usepackage{dsfont}
\def\div{\mathrm{div}}
\def\diag{\mathrm{diag}}
\def\tr{\mathrm{tr}}
\def\R{\mathbb{R}}
\def\vep{\varepsilon}
\def\X{\mathbf{X}}
\def\q{\mathbf{q}}
\def\p{\mathbf{p}}
\def\Qb{\mathbf{Q}}
\def\Pb{\mathbf{P}}
\def\z{\mathbf{z}}
\def\x{\mathbf{x}}
\def\v{\mathbf{v}}
\def\div{\mathrm{div}}
\def\D{\mathbf{D}}
\def\W{\mathbf{W}}
\newcommand{\la}{\langle}
\newcommand{\ra}{\rangle}
\newcommand{\E}{\mathbb{E}}

\newcommand{\Norm}[1]{\Big\|  #1 \Big\|}
\newcommand{\Normbig}[1]{\big\|  #1 \big\|}

\renewcommand{\P}{\mathbb{P}}

\newcommand{\inner}[2]{\Big\langle #1 , #2 \Big\rangle}
\newcommand{\innerbig}[2]{\big\langle #1 , #2 \big\rangle}

\newcommand{\lr}[1]{\Big(  #1 \Big)}
\newcommand{\lrbig}[1]{\big(  #1 \big)}

\newcommand{\ud}{\ensuremath{\mathrm{d} }}
\newtheorem{theorem}{Theorem}[section]

\newtheorem{lemma}[theorem]{Lemma}
\newtheorem{assumption}[theorem]{Assumption}
\newtheorem{proposition}[theorem]{Proposition}
\newtheorem{remark}{Remark}[section]
\newtheorem{definition}[theorem]{Definition}

\numberwithin{equation}{section}

\title[Ergodicity and asymptotic limits for Langevin interacting systems]{Ergodicity and asymptotic limits for Langevin interacting systems with singular forces and multiplicative noises}
\author{Manh Hong Duong$^1$, Hung Dang Nguyen$^2$ and Wenxuan Tao$^1$}
\address{$^1$ School of Mathematics, University of Birmingham, Birmingham, UK}

\address{$^2$ Department of Mathematics, University of Tennessee, Knoxville, Tennessee, USA}

\date{}

\begin{document}

\begin{abstract}
In this paper, we study systems of $N$ interacting particles described by the classical and relativistic Langevin dynamics  with singular forces and multiplicative noises. For the classical model, we prove the ergodicity, obtaining an exponential rate of convergence to the invariant Boltzmann-Gibbs distribution, and the small-mass limit, recovering the $N$-particle interacting overdamped Langevin dynamics. For the relativistic model, we establish the ergodicity, obtaining an algebraic mixing rate of any order to the Maxwell-J\"uttner distribution, and the Newtonian limit (that is when the speed of light tends to infinity), approximating a system of underdamped Langevin dynamics. The proofs rely on the construction of Lyapunov functions that account for irregular potentials and multiplicative noises.
\end{abstract}
\maketitle
\tableofcontents
\section{Introduction}
\subsection{Langevin interacting systems with singular forces and multiplicative noise}
In this paper, we are interested in $N$-particle (underdamped) Langevin interacting systems of the form

\begin{subequations}
\label{eq: general model}
    \begin{align} 
\ud Q^\vep_i &= \nabla K^\vep(P^\vep_i)\,\ud t,~~i=1,\ldots, N, \label{position}\\
\ud P^\vep_i & = -\nabla U(Q^\vep_i)\, \ud t-\sum_{j\neq i}\nabla G\big(Q^\vep_i-Q^\vep_j\big) \, \ud t -D^\vep(Q^\vep_i,p^\vep_i)\nabla K^\vep(P^\vep_i)\,\ud t\notag\\
&\qquad+\div_{P_i}D^\vep(Q^\vep_i,P^\vep_i)\,\ud t+\sqrt{2 D^\vep(Q^\vep_i,P^\vep_i)}\,\ud W_i.\label{eq:velocity}
\end{align}
\end{subequations}

This system describes the motion of the $i$-th particle with position $Q_i^\varepsilon\in \mathbb{R}^d$ and momentum/velocity $P_i^\varepsilon \in \mathbb{R}^d$. The first equation \eqref{position} captures the relation between the position and the momentum/velocity, where we allow for a general kinetic energy $K^\vep:\mathbb{R}^d\rightarrow\mathbb{R}$. The second equation \eqref{eq:velocity} 
posits that the particle moves in accordance with four different forces, namely, (i) an external confining force $-\nabla U(Q_i^\varepsilon)$, (ii) an interaction force $-\sum_{j\neq i}\nabla G(Q_i^\varepsilon-Q_j^\varepsilon)$ obtained from its interactions with other particles, (iii) a friction $-D^\vep (Q_i^\varepsilon,P_i^\varepsilon)\nabla K(P_i^\varepsilon)$ and (iv) a stochastic noise $\sqrt{2 D^\vep(Q^\vep_i,P^\vep_i)})\,\ud W_i$. Here $U, G:\mathbb{R}^d\rightarrow \mathbb{R}$ are given functions, representing external potential and interaction potential respectively. In \eqref{eq:velocity}, we consider multiplicative noises, where $\{W_i\}$ are independent $d$-dimensional standard Wiener processes, and $D^\vep:\mathbb{R}^{2d}\rightarrow\mathbb{R}^{d\times d}$ is a given matrix, representing the strength of the noises. We are using It\^o's formulation and the term $\div_{P_i}D^\vep(Q^\vep_i,P^\vep_i)\,\ud t$ is an It\^o-correction term arising from the multiplicative noise, where for $z\in \R^d$ and a matrix $D=D(z)\in \R^{d\times d}$, the divergence of $D$, $\div(D(z))\in \R^d$ is defined as
\[
[\div(D(z))]_i=\sum_{j=1}^d\frac{\partial  D_{ij}}{\partial z_j}(z).
\]
In \eqref{position}-\eqref{eq:velocity}, we use the superscript $\varepsilon$ to indicate that \eqref{eq: general model} depends on a small parameter $\varepsilon$ that will be specified later. The It\^o-correction term, together with the dissipation-fluctuation relations that appear in \eqref{eq: general model} via the scaling of the matrix $D$ in the external force, the friction and the noise terms, guarantees that, at least formally, the Gibbs distribution $\pi^\varepsilon$ is an invariant probability measure of \eqref{eq: general model}. In particular, thanks to the Hamiltonian structure of the system, $\pi^\varepsilon$ is explicitly given by
\begin{align} \label{form:pi^epsilon}
\pi^\varepsilon(\ud\mathbf{Q},\ud\mathbf{P})=\frac{1}{Z_N^\varepsilon}e^{-H^\varepsilon(\mathbf{Q},\mathbf{P})}\ud\mathbf{Q} \ud\mathbf{P},  
\end{align}
where $Z_N^\varepsilon$ is the normalization constant and the Hamiltonian $H^{\varepsilon}$ is defined as
\[
H^\vep(\mathbf{Q},\mathbf{P})=\sum_{i=1}^N\big(K^\vep(P_i)+ U(Q_i)\big)+\frac{1}{2}\sum_{1\leq i\neq j\leq N} G(Q_i-Q_j).
\]
  This can also be seen by writing \eqref{eq: general model} in a more compact and familiar form of a general underdamped Langevin equation using the notations $\mathbf{Q}^\varepsilon=(Q_1^\varepsilon,\ldots, Q_N^\varepsilon)$, $\mathbf{P}^\varepsilon=(P_1^\varepsilon,\ldots, P_N^\varepsilon)$ and $\D^\varepsilon(\mathbf{Q}^\varepsilon,\mathbf{P}^\varepsilon)=\diag(D^\varepsilon(Q_1^\varepsilon,P_1^\varepsilon),\ldots,D^\varepsilon(Q_N^\varepsilon,P_N^\varepsilon))$:
\begin{align*}  
    \notag \ud \mathbf{Q}^\vep=& \nabla_{\mathbf{P}} H^\vep(\mathbf{Q}^\vep,\mathbf{P}^\vep) \, \ud t,\\
 \notag \ud \mathbf{P}^\vep=& - \nabla_{\mathbf{Q}}H^\vep(\mathbf{Q}^\vep,\mathbf{P}^\vep)\, \ud t- \D^\varepsilon(\mathbf{Q}^\vep,\mathbf{P}^\vep)\nabla_{\mathbf{P}}H^\vep(\mathbf{Q}^\vep,\mathbf{P}^\vep) \,\ud t 
 \\
 &\quad +\div_{\mathbf{P}}\D^\vep(\mathbf{Q}^\vep,\mathbf{P}^\vep)\ud t+ \sqrt{2 \D^\vep(\mathbf{Q}^\vep,\mathbf{P}^\vep)}\ud \W_t.
\end{align*}
The purpose of this paper is to rigorously establish the unique ergodicity of \eqref{eq: general model} together with a qualitative convergence rate toward the equilibrium $\pi^\varepsilon$, and derive the asymptotic limit of \eqref{eq: general model} as $\varepsilon\rightarrow 0$.    
More specifically, we focus on the following two important classes of models:

(1) \textbf{$N$-particle interacting systems of classical Langevin dynamics with state-dependent friction}:
\begin{equation}\label{E:Classical_Eq}
    \begin{aligned}
        \ud x_i &= v_i\,\ud t,~~i=1,\ldots, N,\\
m \ud v_i & = -\nabla U(x_i)\, \ud t-\sum_{j\neq i}\nabla G\big(x_i-x_j\big) \, \ud t -D(x_i)v_i\,\ud t+\sqrt{2 D(x_i)}\,\ud W_i.
    \end{aligned}
\end{equation}
In the above, $m$ is a positive constant representing the particle's mass. We note that \eqref{E:Classical_Eq} is a special case of \eqref{eq: general model} with
\[(Q,P)=(x,mv),\quad
K(P)=\frac{1}{2m}|P|^2=\frac{1}{2}m|v|^2,\quad D(Q,P)=D(x),\quad \vep=m.
\]
Here, while $K$ is the classical kinetic energy, the noise and friction coefficients are both state-dependent through the matrix $D$ that satisfies Assumption \ref{Assumption D} below. Langevin dynamics with state-dependent friction and diffusion of the form \eqref{E:Classical_Eq} arise in many applications, for instance as a coarse-grain model from quantum classical molecular
dynamics~\cite{szepessy2011langevin}, in dissipative particle dynamics \cite{leimkuhler2016pairwise}, and in sampling techniques~\cite{lim2025appropriate}, just to name a few.

Note that we write equation \eqref{E:Classical_Eq} in terms of the position and velocity rather than the position and momentum. This will be more convenient for the purpose of studying the behaviors of \eqref{E:Classical_Eq} in the regime $\varepsilon=m\rightarrow 0$, which is the celebrated small mass (also known as the overdamped or Smoluchowski-Kramers) limit. Another asymptotic limit of interest is the convergence of \eqref{E:Classical_Eq} toward the invariant probability measure, denoted by $\pi^\varepsilon=\pi^m_{\textup{GB}}$, that corresponds to the Gibbs-Boltzmann distribution. More specifically, in view of \eqref{form:pi^epsilon}, $\pi^m_{\textup{GB}}$ is given by
\begin{align} \label{form:pi^m_GB}
   \pi^m_{\textup{GB}}(\ud \x,\ud\v) = \frac{1}{Z_N^m}\exp\Big\{-\sum_{i=1}^N \Big(\frac{1}{2}m|v_i|^2+U(x_i)\Big)-\frac{1}{2}\sum_{1\le i\neq j\le N}G(x_i-x_j)   \Big\} \ud \x\ud\v.
\end{align}
The main results established for \eqref{E:Classical_Eq} in the aforementioned topics will be precisely stated in Theorems \ref{thm: ergodicityLE} and \ref{thm: smallLE} in Section \ref{sec:intro:result}.

(2) \textbf{$N$-particle interacting systems of relativistic Langevin dynamics}:
\begin{align}\label{E:RelativisticSys}
\ud q_i(t) &= \frac{\mathbf{c}p_i}{\sqrt{m^2\mathbf{c}^2+|p_i|^2}}\, \ud t, ~~i=1,\ldots, N,\notag \\
\ud  p_i(t) & = - \gamma D(p_i) \frac{\mathbf{c}p_i}{\sqrt{m^2\mathbf{c}^2+|p_i|^2}} \, \ud  t + \gamma \div(D(p_i))\ud  t  +\sqrt{2 \gamma D(p_i)} \ud  W_{i}(t)\notag \\
& \qquad -\nabla U(q_i)\, \ud  t- \sum_{j\neq i}\nabla G\big(q_i-q_j\big) \, \ud  t,
\end{align}
where $\mathbf{c}$ denotes the speed of light, $\gamma$ represents the friction coefficient, and the friction matrix $D$ is given by
\begin{align} \label{E:DefD:c}
     D(p)=\frac{m\mathbf{c}}{\sqrt{m^2\mathbf{c}^2+|p|^2}}\Big(I+\frac{p\otimes p}{m^2\mathbf{c}^2}\Big) .
\end{align}
Here $I\in \R^{d\times d}$ denotes the identity matrix. We note that system \eqref{E:RelativisticSys} is a special case of \eqref{eq: general model} with
\begin{align}
    \label{E:RelativisiicKD}
(Q,P)=(q,p),\quad K(p)=\mathbf{c}\sqrt{m^2\mathbf{c}^2+|p|^2},\quad D(q,p)=D(p),\quad \vep=\frac{1}{\mathbf{c}^2}.
\end{align}
Historically, the single-particle system ($N=1$, $G=0$) of \eqref{E:RelativisticSys} was originally introduced in \cite{dunkel2005atheory,dunkel2005theory} as an extension of the classical Langevin dynamics to the special relativity setting, that is to comply with Einstein's special relativity. The specific momentum-dependent form of the diffusion matrix $D$ in \eqref{E:DefD} guarantees, among other properties, that system \eqref{E:RelativisticSys}, in the absence of friction, is Lorentz invariant \cite{alcantara2011relativistic,dunkel2005atheory,dunkel2005theory}. We note that while this type of invariance is a desirable feature when extending the classical Langevin dynamics to the special relativity, it does not hold true in the presence of additive noise, hence the choice of the diffusion matrix $D$ as in \eqref{E:DefD:c}. In the past decades, many relativistic stochastic models similar to \eqref{E:RelativisticSys} have been proposed, see for instance the work of \cite{alcantara2011relativistic,chevalier2008relativistic,
debbasch2004diffusion,debbasch1998diffusion,
duong2015formulation,felix2013newtonian,
haba2009relativistic,
haba2009relativisticII,haba2010energy,
pal2020stochastic} and the survey articles \cite{debbasch2007relativistic,dunkel2009relativistic} for a detailed exposition of the topic.

In this model, since the mass parameter $m$ and the friction coefficient $\gamma$ are fixed and do not affect the analysis, we will set $m=\gamma=1$ for the sake of simplicity. So, in term of $\varepsilon=1/\mathbf{c}^2$, equation \eqref{E:RelativisticSys} can be recast as 
\begin{align}\label{E:RelativisticSys:epsilon}
\ud q^\varepsilon_i(t) &= \frac{p^\vep_i}{\sqrt{1+\varepsilon|p^\varepsilon_i|^2}}\, \ud t, \notag \\
\ud p^\varepsilon_i(t) & = - D(p^\varepsilon_i) \frac{p^\varepsilon_i}{\sqrt{1+\varepsilon|p^\varepsilon_i|^2}} \, \ud t + \div(D(p^\varepsilon_i))\ud t  +\sqrt{2 D(p^\varepsilon_i)} \ud W_{i}(t) \\
& \qquad -\nabla U(q^\varepsilon_i)\, \ud t- \sum_{j\neq i}\nabla G\big(q^\varepsilon_i-q^\varepsilon_j\big) \, \ud t,\notag
\end{align}
where the friction matrix $D$ is expressed as
\begin{equation}\label{E:DefD}
    D(p) =\frac{1}{\sqrt{1+\varepsilon|p|^2}}\lrbig{I+\varepsilon p\otimes p}.
\end{equation}
Note that $D$ also admits the following representation
\begin{align*}
    D(p)=\frac{1}{\sqrt{1+\varepsilon|p|^2}} P^\perp+\sqrt{1+\varepsilon|p|^2}P,
\end{align*}
where $P=p \otimes p/|p|^2$ denotes the projection map from $\R^d$ to $\text{span}\{p^\varepsilon_i\}$ and $P^\perp=I-P$. Since $D$ is positive definite, we can take the square root to get 
\begin{align}\label{E:sqrtD}
    \sqrt{D(p)}=\frac{1}{\big(1+\varepsilon|p|^2\big)^{1/4}}P^{\perp}+\lrbig{1+\varepsilon|p|^2}^{1/4}P.
\end{align}
In the case $N=1$, \eqref{E:RelativisticSys:epsilon} is reduced to
\begin{align} \label{E:RelativisticSysN=1}
\ud q(t) &= \frac{p}{\sqrt{1+\varepsilon|p|^2}}\, \ud t, \notag \\
\ud  p(t) & = - D(p) \frac{p}{\sqrt{1+\varepsilon|p|^2}} \, \ud  t + \div(D(p))\ud  t  +\sqrt{2 D(p)} \ud  W(t)\notag \\
& \qquad -\nabla U(q)\, \ud  t- \nabla G\big(q\big) \, \ud  t.
\end{align}
It is important to point out that we still include the potential $G$ in \eqref{E:RelativisticSysN=1} in order to understand the role of singularity. 

Two asymptotic topics of interest for \eqref{E:RelativisticSys:epsilon} are the large-time stability and the approximation of \eqref{E:RelativisticSys:epsilon} in the regime of $\varepsilon\rightarrow 0$ (i.e., $\mathbf{c}\rightarrow +\infty$ in \eqref{E:RelativisticSys}). Particularly, the former explores the convergence of the dynamics toward the statistically steady state, which is the Maxwell-J\"uttner distribution and is denoted by $\pi^\varepsilon=\pi^{\varepsilon}_{\textup{MJ}}$, whereas the latter corresponds to the so-called Newtonian (non-relativistic) limit. Here, in comparison with \eqref{form:pi^epsilon}, $\pi^\varepsilon_{\text{MJ}}$ is given by
\begin{align} \label{form:pi^epsilon_MJ}
    \pi^\vep_{\text{MJ}}(\ud \q,\ud \p) =  \frac{1}{Z_N^\vep}\exp\Big\{-\sum_{i=1}^N\Big( \frac{1}{\varepsilon}\sqrt{1+\varepsilon  |p_i|^2}+U(q_i)
    \Big)-\frac{1}{2}\sum_{1\le i\neq j\le N}G(q_i-q_j)   \Big\} \ud \q\ud\p.
\end{align}

We refer the reader to Theorems \ref{thm: ergodicityRLE} and \ref{thm:NewtonianLimit} for the precise statements of the main results established for \eqref{E:RelativisticSys:epsilon}.

\subsection{Main results} \label{sec:intro:result}
In this paper, we will seek sufficient conditions on the nonlinearities $U, G$ and the diffusion matrix $D$ in order to establish the ergodicity and small-mass limit for the classical model \eqref{E:Classical_Eq}, as well as the ergodicity and Newtonian limit for the relativistic equation \eqref{E:RelativisticSys:epsilon}. Although \eqref{E:Classical_Eq} and \eqref{E:RelativisticSys:epsilon} appear unrelated, they can actually be analyzed using the same methodology and techniques.

Owing to the presence of the repulsive force $G$, it is crucial to ensure that no collision occurs in finite time. For this purpose, we introduce the domain $\mathcal{D}$ where the process $\Qb^\varepsilon(t)$ evolves in
\begin{align} \label{form:D}
\mathcal{D}=\begin{cases}
    \{\Qb=(Q_1,\dots,Q_N) \in (\R^d)^N: Q_i\neq Q_j \text{ if } i\neq j\},& d\ge 2,\\
    \{\Qb=(Q_1,\dots,Q_N) \in (\R)^N: Q_1<Q_2<\dots<Q_N\}, & d=1.
\end{cases} 
\end{align}
As mentioned elsewhere in \cite{duong2024trend,herzog2017ergodicity,lu2019geometric}, since the set $\{\q=(q_1,\dots,q_N) \in (\R^d)^N: q_i\neq q_j \text{ if } i\neq j\}$ is not path connected when $d=1$, we have to restrict the dynamics of \eqref{eq: general model} to a connected component of $(\R)^N$, hence the choice of $\mathcal{D}$ when $d=1$ for simplicity. Then, we define the phase space for the general equation \eqref{eq: general model} as follow:
\begin{align} \label{form:X}
\X=\mathcal{D}\times (\R^d)^N .
\end{align}
Under the assumptions of the present paper, cf. Section \ref{sec: prim:assumption}, the well-posedness of the two particular systems \eqref{E:Classical_Eq} and \eqref{E:RelativisticSys:epsilon} is guaranteed, see Theorem \ref{T:ExistandUnique} in Section \ref{sec: wellposedness} below. In other words, there always exists a unique global strong solution evolving in the phase space $\X$. 

As an immediate consequence of the well-posedness result stated in Theorem \ref{T:ExistandUnique}, we can introduce a Markov transition probability given by the solution $X^\varepsilon=(\Qb^\varepsilon,\Pb^\varepsilon)$. That is for any Borel set $A\subset \X$,
\begin{align*}
P_t^\varepsilon(X_0,A):=\P(X^\varepsilon(t;X_0)\in A),
\end{align*}
is well-defined for all $t\ge 0$ where $X_0$ is the initial condition of~\eqref{eq: general model}. Moreover, $P^\varepsilon_t$ defines a Markov semigroup on the space of bounded Borel measurable functions $\mathcal{B}_b(\X)$. That is for any $f\in \mathcal{B}_b(\X)$,
\begin{align*}
    P_t^\varepsilon f(X_0)=\E[f(X^\varepsilon(t;X_0))], \,\, f\in \mathcal{B}_b(\X).
\end{align*}
Let $\mathcal{P}r(\X)$ denote the space of probability measures on the Borel sets of $\X$. A distribution is called invariant if for any $f\in \mathcal{B}_b(\X)$, it holds that
\begin{align*}
\int_{\X} f(X) (P_t^\varepsilon)^*\mu(\ud X)=\int_{\X} f(X)\mu(\ud X),
\end{align*}
where $(P_t^\varepsilon)^*\mu(\ud X)$ is the probability measure obtained by applying $P^\varepsilon_t$ to the measure $\mu$,
\begin{align*}
(P_t^\varepsilon)^*\mu(A) = \int_{\X} P_t^\varepsilon(X,A)\mu (\ud X).
\end{align*}

To measure the convergence rate of solutions of \eqref{E:Classical_Eq} and \eqref{E:RelativisticSys:epsilon} toward the corresponding equilibrium measures, we need to recall the definition of a weighted total variance distance developed in \cite{hairer2011yet}. For a positive function $V$, we introduce the weighted norm on the measurable functions defined as
\begin{align}
    \Normbig{f}_V=\sup_{X\in\X}\frac{|f(X)|}{1+V(X)},
\end{align}
which induces a Banach space denoted by $\mathcal{B}(\X;V)$. The corresponding weighted total variance distance is given by
\begin{align}
    \mathcal{W}_V\lrbig{\mu_1,\mu_2}=\sup_{\|f\|_V\le 1}\Big|\int_{\X} f(X)\mu_1(\ud X)-\int_{\X} f(X)\mu_2(\ud X)\Big|.
\end{align}

We are now in the position to state the main results of the present paper, namely, Theorems \ref{thm: ergodicityLE} - \ref{thm: smallLE} and Theorems \ref{thm: ergodicityRLE} - \ref{thm:NewtonianLimit} respectively concerning the asymptotic analysis for equation \eqref{E:Classical_Eq} and equation \eqref{E:RelativisticSys:epsilon}.

We start with Theorem \ref{thm: ergodicityLE} establishing that the distributions of solutions to the classical Langevin model \eqref{E:Classical_Eq} exponentially converges, in a weighted total variation distance, to the unique Gibbs-Boltzmann invariant measure.
\begin{theorem}[Ergodicity of the classical model]
\label{thm: ergodicityLE}
    Suppose Assumptions \ref{Assumption U}, \ref{A:G} and \ref{Assumption D} respectively on the external potential $U$, the interaction potential $G$, and the diffusion matrix $D$ hold. Then the Gibbs-Boltzmann distribution $\pi^m_{\textup{GB}}$ given by \eqref{form:pi^m_GB} is the unique invariant measure of the classical Langevin system~\eqref{E:Classical_Eq}. Moreover, for each $m>0$, there exist a function $V_m:\X\to[1,\infty)$, positive constants $\alpha_m$ and $C_M$ such that the following holds
    \begin{equation}
  \label{eq:LErate}      \mathcal{W}_{V_m}\lrbig{\lrbig{P_t^m}^*\mu,\pi^{m}_{\textup{GB}} }\le C_m e^{-\alpha_mt}\mathcal{W}_{V_m}\lrbig{\mu,\pi^{m}_{\textup{GB}}},\quad t\ge 0,
    \end{equation}
    for all probability measure $\mu$ such that
    \begin{align*}
        \int_\X V_m(X)\mu(\ud X)<\infty.
    \end{align*}
\end{theorem}

In the second result stated for the classical Langevin model \eqref{E:Classical_Eq}, we rigorously justify that the overdamped Langevin dynamics of $N$ interacting particles is the small-mass limit of \eqref{E:Classical_Eq}.

\begin{theorem}[Small-mass limit of the classical model]
\label{thm: smallLE}  Suppose Assumptions \ref{Assumption U}, \ref{A:G} and \ref{Assumption D} respectively on the external potential $U$, the interaction potential $G$, and the diffusion matrix $D$ hold.
For all initial value $(\x_0,\v_0)\in \X$, let $(\x^m,\v^m)$ be the solution to~\eqref{E:Classical_Eq} and $\q$ solve the following overdamped Langevin dynamics:
    \begin{align}\label{EC:Limiting}
    \ud q_i=[-D^{-1}(q_i)(\nabla U(q_i)+\sum_{j\neq i}\nabla G(q_j-q_i))-\div D^{-1}(q_i)]\ud t+\sqrt{2}\sqrt{D}^{-1}\ud W_i.
\end{align}
with initial condition $\q(0)=\x_0$.
Then, $\x^m$ converges to $\q$ in probability in the small mass regime, i.e., for all $T>0$ and $\xi>0$,
\begin{align*}
    \P\Big(\sup_{t\in[0,T]}|\x^m(t)-\q(t)|\ge \xi\Big)\to0,\quad m\to 0.
\end{align*}
\end{theorem}
It is worth noting the appearance of the diffusion matrix $D$ in the limiting system, in particular it gives rise to an additional drift term $\div{D}^{-1}(q_i)$ in the limiting system. This is known as noise-induced phenomena which has been revealed in many other systems, see for instance \cite{volpe2016effective}. The proofs of Theorems \ref{thm: ergodicityLE} - \ref{thm: smallLE} will be supplied in details in Section \ref{sec: classical LE}.

In the next two theorems stated for the relativistic model \eqref{E:RelativisticSys:epsilon}, we establish the ergodicity for large-time stability and the Newtonian limit $\varepsilon\to 0$, i.e., as the speed of light $\mathbf{c}\to\infty$ in \eqref{E:RelativisticSys}. Concerning the former limit, in contrast to the classical model, due to the interplay between the nonlinearity of the relativistic kinetic energy and the irregularity of the interaction forces, provided the speed of light $\mathbf{c}$ is large enough, we are only able to obtain an algebraic convergence rate, of any order, toward the unique Maxwell-J\"uttner invariant distribution.
\begin{theorem}[Ergodicity of the relativistic model]
\label{thm: ergodicityRLE}
     Suppose that Assumptions \ref{Assumption U} and \ref{A:G} respectively on the external potential $U$ and the interaction potential $G$ hold and that the diffusion matrix $D$ satisfies \eqref{E:DefD}. Suppose further that when $N\ge2$, $G$ also satisfies Assumption \ref{A:G:relativistic:N>1}. 
     
     \textup{(1)} For all $\varepsilon>0$ sufficiently small, $\pi^\varepsilon_{\textup{MJ}}$ defined in \eqref{form:pi^epsilon_MJ} is the unique invariant measure of the relativistic Langevin system \eqref{E:RelativisticSys:epsilon}.

         \textup{(2)} For all $r>0$, there exists a positive constant $\varepsilon^*=\varepsilon^*(r)$ sufficiently small such that for all $\varepsilon<\varepsilon^*$, there exist a function $V_{\varepsilon,r}:\X\to[1,\infty)$ and a positive constant $C_{\varepsilon,r}$ such that the following holds
\begin{align} \label{eq: RLErate}
\mathcal{W}_{V_{\varepsilon,r}}\lrbig{\lrbig{P_t^\varepsilon}^*\mu,\pi^{\varepsilon}_{\textup{MJ}} }\le \frac{C_{\varepsilon,r}}{(1+t)^r}\mathcal{W}_{V_{\varepsilon,r}}\lrbig{\mu,\pi^{\varepsilon}_{\textup{MJ}}},\quad t\ge 0,
    \end{align}
    for all probability measure $\mu$ such that
    \begin{align*}
        \int_\X V_{\varepsilon,r}(X)\mu(\ud X)<\infty.
    \end{align*}

\end{theorem}
In the last result concerning the Newtonian limit for the relativistic Langevin equation \eqref{E:RelativisticSys:epsilon}, we send the speed of light $\mathbf{c}$ to infinity and rigorously justify that \eqref{E:RelativisticSys:epsilon} is well approximated by the classical Langevin system studied in \cite{herzog2017ergodicity,lu2019geometric}.

\begin{theorem}[Newtonian limit]
\label{thm:NewtonianLimit}
     Suppose that Assumptions \ref{Assumption U} and \ref{A:G} respectively on the external potential $U$ and the interaction potential $G$ hold and that the diffusion matrix $D$ satisfies \eqref{E:DefD}.
     
     \textup{(1)} Multi-particle case \textup{($N\ge 2$)}: Let $(\q^\varepsilon,\p^\varepsilon)$ be the solution of~\eqref{E:RelativisticSys:epsilon} and $(\q,\p)$ solve the equation
\begin{align}\label{E:LimitingRelativisticSys}
\ud q_i(t) &= p_i\, \ud t,\quad i=1,\dots,N, \notag \\
\ud p_i(t) & = - p_i \, \ud t  -\nabla U(q_i)\, \ud t- \sum_{j\neq i}\nabla G\big(q_i-q_j\big) \, \ud t+\sqrt{2} \ud W_{i}(t),
\end{align}
with $(\q^\varepsilon(0),\p^\varepsilon(0))=(\q(0),\p(0))=(\q_0,\p_0)$. Then, for all $T>0$ and $\xi>0$.
        \begin{equation} \label{lim:Newtonian:N>1}
    \mathbb{P} \Big( \sup_{t \in [0, T]} \Big[ |\q^\varepsilon(t) - \mathbf{q}(t)| + |\p^\varepsilon(t) - \mathbf{p}(t)| \Big] > \xi \Big) \to 0, \quad \varepsilon \to 0;
\end{equation}

        \textup{(2)} Single-particle case $(N=1)$:   Let $(q^\varepsilon,p^\varepsilon)$ be the solution of~\eqref{E:RelativisticSysN=1} and $(q,p)$ solve the equation
\begin{align}\label{E:LimitingRelativisticSys:N=1}
\ud q(t) &= p\, \ud t,\notag \\
\ud p(t) & = - p \, \ud t  -\nabla U(q)\, \ud t- \nabla G(q) \, \ud t+\sqrt{2} \ud W(t),
\end{align}
with $(q^\varepsilon(0),p^\varepsilon(0))=(q(0),p(0))=(q_0,p_0)$. Then, for all $T>0$ and $n\ge 1$, 
        \begin{equation} \label{lim:Newtonian:N=1:lambda>1}
            \E\sup_{t\in [0,T]}\Big[|q^\varepsilon(t)-q(t)|^n+|p^\varepsilon(t)-p(t)|^n\Big] \to 0,\quad \varepsilon \to 0.
        \end{equation}
\end{theorem}

\subsection{Methodology of the proofs} \label{sec:intro:proof}
\textit{Idea of the proof of convergence rate to equilibrium.} With regards to the mixing results in Theorems \ref{thm: ergodicityLE} and \ref{thm: ergodicityRLE}, we adopt the framework of \cite{hairer2009hot}, which in turn was built upon the techniques developed in \cite{ bakry2008rate, douc2009subgeometric, fort2005subgeometric}. The proof of estimates \eqref{eq:LErate} and \eqref{eq: RLErate} consists of three main ingredients, namely, the H\"ormander's condition ensuring the smoothness of the transition probabilities, the solvability of the associated control problem giving the possibility of returning to the center of the phase space, and a suitable Lyapunov function quantifying the convergent rate. 

The first criterion is a consequence of the classical H\"ormander's Theorem \cite{hormander1967hypoelliptic} asserting that the phase space may be generated by the collection of vector fields jointly induced by the diffusion and the drifts. Since we are dealing with finite-dimensional settings, this is a relatively short computation on Lie brackets. The second condition on the associated control problem can be established by the Support Theorem \cite{stroock1972degenerate} showing that one can always find appropriate controls allowing for driving the dynamics to any bounded ball. The proof of H\"ormander theorem and solvability condition are rather standard and will be discussed in Section \ref{sec: prim}.
The last ingredient requires the construction of a Lyapunov function, which is an energy-like function $V$ satisfying an inequality of the form
\begin{align} \label{ineq:Lyapunov:polynomial-mixing}
\frac{\ud}{\ud t}\E\big[ V(\mathbf{Q}^\vep(t),\mathbf{P}^\vep(t)) \big]\le  -c_1\E\big[ V(\mathbf{Q}^\vep(t),\mathbf{P}^\vep(t))^\alpha\big]+c_2,\quad t\ge 0,
\end{align}
for a suitable constant $\alpha\in(0,1]$. 

The construction of such a function $V$ is highly nontrivial due to the singularity of the interaction forces and the degeneracy of the noises. It needs to be done on a case-by-case basis. For the classical model \eqref{E:Classical_Eq}, we will show that \eqref{ineq:Lyapunov:polynomial-mixing} holds for $\alpha=1$, giving rise to the exponential rate of convergence stated in Theorem \ref{thm: ergodicityLE}. However, for the relativistic model \eqref{E:RelativisticSys:epsilon}, due to the lack of strong dissipation in the $p$-direction as well as the impact of the nonlinearity of the relativistic kinetic energy and the singularities of the interaction potentials, we are only able to prove \eqref{ineq:Lyapunov:polynomial-mixing} for some $\alpha\in(0,1)$, thus yielding only an algebraic mixing rate as stated in Theorem \ref{thm: ergodicityRLE}.

\textit{Idea of the proofs of the singular parameter limits.} The argument of Theorems \ref{thm: smallLE} and \ref{thm:NewtonianLimit} is motivated by the strategy of \cite{ duong2024asymptotic, duong2024trend,herzog2016small}. The general approach essentially consists of two main steps: firstly, we truncate the nonlinearities in both the original and the limiting systems. This results in Lipschitz dynamics, whose convergence in expectations can be handled directly. Then, we remove the Lipschitz constraint by making use of suitable moment bounds on the solutions. In particular, for the relativistic model \eqref{E:RelativisticSys:epsilon}, due to the unboundedness of the diffusion matrix, it is also crucial to truncate the noise term, making the analysis more intricate than the classical model \eqref{E:Classical_Eq}.
\subsection{Novelties}The key challenge and novelty of the present paper is that we deal with $N$-particle interacting systems with physically relevant singular interaction forces, including Coulomb and Lennard-Jones forces, and multiplicative noises, and with non-quadratic kinetic energy in the relativistic model. The combination of the nonlinearity, the singularity and the multiplicative noises makes the analysis of the systems highly nontrivial. Existing works in the literature lack one or more of these features. In what follows, we review the historical background of equations \eqref{E:Classical_Eq} and \eqref{E:RelativisticSys:epsilon}.
\subsection{Relevant literature}
There is a vast literature on the ergodicity and asymptotic limits for the Langevin dynamics and related models. Below, we only discuss most relevant papers that consider either singular forces or multiplicative noises. 

\textit{On ergodicity of the classical Langevin dynamics.} Under different assumptions on the nonlinearities, the geometric ergodicity of Langevin dynamics with repulsive interactions was proven using the PDE and hypocoercivity approach \cite{villani2009hypocoercivity} in~\cite{conrad2010construction,grothaus2015hypocoercivity}, and through the construction of suitable Lyapunov functions in \cite{baudoin2021gamma,camrud2021weighted,cooke2017geometric,herzog2017ergodicity,lu2019geometric,song2020well}. The latter method has been extended to the generalized Langevin dynamics \cite{duong2024asymptotic}, Langevin dynamics driven by jump processes \cite{bao2024exponential} and the setting of quasi-stationary distributions \cite{guillin2023quasi}. However, the common feature of these aforementioned papers is the presence of additive noise. While there are several work exploring ergodicity of the Langevin dynamics with multiplicative noises \cite{bertram2022essential, eisenhuth2024hypocoercivity}, they only consider smooth confining potentials ($G\equiv 0$). 

\textit{On small mass limit of the classical Langevin model}. The small-mass limit of the classical Langevin dynamics, with smooth coefficients, has been studied intensively since the seminal paper \cite{kramers1940brownian}, see the recent paper \cite{son2024rate} and references therein for more information. Regarding the setting with singular forces, in \cite{el2010diffusion, goudon2005hydrodynamic,poupaud2000parabolic}, the qualitative analysis of overdamped limit from the Vlasov–Poisson–Fokker–Planck system (which is the mean-field limit of the multi-particle Langevin dynamics with Coulomb potentials) towards the drift-diffusion equation is studied. A quantitative rate of convergence in the Wasserstein distance for this limit is provided in~\cite{choi2022quantified}. We also mention the article \cite{grothaus2020overdamped}, which establishes the small mass limit (convergence in laws) for the single-particle (no pair-wise interaction forces) Langevin dynamics with singular external forces. On the one hand, the results from these papers only cover the case of singular forces under the impact of additive noise. On the other hand, the small mass limit for the Langevin dynamics in the presence of multiplicative noise but regular forces is proved rigorously in~\cite{herzog2016small,hottovy2015smoluchowski}. See also \cite{volpe2016effective} for a survey on the topic. More recently, the work of \cite{xie2022smoluchowski} investigated the small-mass limit of a single-particle Langevin dynamics with multiplicative noise and H\"older continuous forces, which is a different setting from those considered in the present article.

For the classical Langevin dynamics \eqref{E:Classical_Eq}, compared to the existing literature on both topics, the novelty in this paper is that we incorporate both singular forces and multiplicative noises for $N$-particle interacting systems.

\textit{On ergodicity of the relativistic model}. Concerning the issue of large-time stability, a rigorous proof of the exponential convergence toward equilibrium for the spatially homogeneous relativistic Fokker-Planck equation is established in \cite{angst2011trends}. The well-posedness of the spatially inhomogeneous Fokker-Planck equation is established in \cite{alcantara2011relativistic} whereas in the absence of potentials $U$, the momentum processes are shown to exhibit a geometric mixing rate \cite{angst2011trends,
felix2013newtonian}. Similar asymptotic behavior is also obtained when the position variable is posed on a torus \cite{calogero2012exponential} making use of the hypocoercive method \cite{villani2009hypocoercivity}. More recently, the technique has been extended by means of Lyapunov functionals to successfully treat the case of smooth potential $U$ in~\cite{arnold2025trend}. We also mention papers that study convergence to equilibrium for Langevin dynamics with general kinetic energy \cite{brigati2023construct,brigati2024explicit,stoltz2018langevin}.

\textit{On the Newtonian limit of the relativistic Langevin dynamics}. The Newtonian limit for the single-particle relativistic Langevin dynamics, with regular forces, has been formally derived in \cite{debbasch1997relativistic,dunkel2005theory} and rigorously justified later in \cite{felix2013newtonian}. See also \cite{han2018long,schaeffer1986classical} for the related Maxwell-Vlasov systems. 

To the best of our knowledge, the ergodicity and the Newtonian limit for the $N$-particle relativistic Langevin dynamics with singular forces is investigated only recently by the first two named authors in \cite{duong2024trend}. However, this paper deals with an additive constant diffusion matrix, that is \eqref{E:RelativisticSys:epsilon} with $D=I$, which is the relativistic model proposed in \cite{debbasch.mallick.ea;97;relativistic}. In the present article, we extend the technique employed in \cite{duong2024trend} to successfully handle the case of multiplicative noise in \eqref{E:RelativisticSys:epsilon} with $D$ given by \eqref{E:DefD}.


\subsection{Organization of the paper}
The rest of the paper is organized as follows. In Section \ref{sec: prim}, we present precise assumptions on the external potential $U$, the interaction potential $G$, and the diffusion matrix $D$, as well as a preliminarily result on the well-posedness, mininorization condition and controllability of the general system \ref{eq: general model}.
In Section \ref{sec: classical LE}, we prove Theorems \ref{thm: ergodicityLE}  and \ref{thm: smallLE} respectively establishing the ergodicity and the small-mass limit for the classical model \eqref{E:Classical_Eq}. In Section \ref{sec: RLE}, we prove Theorems \ref{thm: ergodicityRLE} and \ref{thm:NewtonianLimit} respectively for the ergodicity and the Newtonian limit of the relativistic model \eqref{E:RelativisticSys:epsilon}. The paper concludes with Appendix \ref{sec:appendix} where we collect auxiliary estimates that are invoked to prove the main theorems.

\section{Preliminary}
\label{sec: prim}
\subsection{Assumptions} \label{sec: prim:assumption}
In this section, we introduce the assumptions we make on the the external potential $U$, the interaction potential $G$ and the diffusion matrix $D$. We also present preliminary results that hold true for both systems \eqref{E:Classical_Eq} and \eqref{E:RelativisticSys:epsilon}, using the general and common form \eqref{eq: general model}.
We start with Assumption \ref{Assumption U} giving sufficient conditions on the smooth potential $U$. For $a,b\in \mathbb{R}^d$, we use both notations $a\cdot b$ or $\langle a, b\rangle$ to denote their standard inner product.
\begin{assumption}[U]\label{Assumption U}
    The function $U\in C^\infty\lrbig{\R^d;[1,\infty)}$ satisfies the followings for all $ Q\in\R^d $:

    \textup{(i)} There exist constants $a_1,a_2,a_3>0$ and $\lambda\ge 1$ such that
    \begin{align}
        &\begin{aligned}\label{cond:U:U(x)=O(x^lambda+1)}
            |U(Q)|\le a_1\lrbig{1+|Q|^{\lambda+1}},
        \end{aligned}\\
        &\begin{aligned}\label{cond:U:U'(x)=O(x^lambda)} 
            |\nabla U(Q)|\leq a_1\lrbig{1+|Q|^\lambda},
        \end{aligned} \\
        &\begin{aligned}\label{cond:U:x.U'(x)>-x^(lambda+1)}
            \langle\nabla U(Q),Q\rangle\geq a_2|Q|^{\lambda+1}-a_3.
        \end{aligned}
    \end{align}

    \textup{(ii)} In the classical model \eqref{E:Classical_Eq}, we further assume that
    \begin{align}\label{E:U_3}
         \|\nabla^2U(Q)\|\leq a_1(1+|Q|^{\lambda-1}), 
    \end{align}
    where $\|A\|$ denotes the spectral norm of the matrix $A$.
\end{assumption}
\begin{remark} We note that as a consequence of \eqref{cond:U:x.U'(x)>-x^(lambda+1)}, $U(Q)$ is bounded from below by $\frac{1}{a_1}|Q|^{\lambda+1}-a_1$. To see this, for $Q$ satisfying $|Q|\ge 1$, let $\phi(r)=U(r\xi)=U(Q)$ where $\xi=\frac{Q}{|Q|}$ and $r=|Q|$. Then, we have
    \begin{align*}
        \phi'(r)=\innerbig{\nabla U(r\xi)}{\xi}\ge \frac{1}{r}\lrbig{a_2|r|^{\lambda+1}-a_3},
    \end{align*}
    whence
    \begin{align*}
        \phi(r)-\phi(1)\ge \int_1^r\frac{1}{s}\lrbig{a_2|s|^{\lambda+1}-a_3}\ud s=\frac{a_2}{\lambda+1}(r^{\lambda+1}-1)-a_3\log r.
    \end{align*}
    Since $U$ is continuous, $\phi(1)=U(\xi)\ge \inf_{|Q|\le 1}U(Q)$. Then, we may infer the existence of a positive constant $C$ such that for all $Q$ satisfying $ |Q|\ge 1$
    \begin{align*}
        U(Q)\ge \frac{a_2}{2(\lambda+1)}|Q|^{\lambda+1}-C.
    \end{align*}
    By taking $a_1$ sufficiently large if necessary, we conclude 
    \begin{align}\label{cond:U:U(x)>=O(x^lambda+1)}
         |U(Q)|\ge \frac{1}{a_1}|Q|^{\lambda+1}-a_1,\quad Q\in\R^d,
    \end{align}
    as claimed.
\end{remark}
Next, we impose the following assumption on the singular potential $G$:
\begin{assumption}[G]\label{A:G}
The even function $G\in C^\infty(\R^d\setminus\{0\};\R) $ satisfies the followings for all $Q\in \R^d$:

\textup{(i)} $G(Q)$ converges to infinity as $|Q|$ goes to $0$. Furthermore, there exists $\beta_1\ge 1$ such that for all $Q\in \R^d\setminus\{0\}$,
    \begin{align}
        &\begin{aligned}\label{E:G_1}
            |G(Q)|\leq a_1\Big(1+|Q|+\frac{1}{|Q|^{\beta_1}}\Big),
        \end{aligned}\\
        &\begin{aligned}\label{E:G_2}
            |\nabla G(Q)|\leq a_1\Big    (1+\frac{1}{|Q|^{\beta_1}}\Big),
        \end{aligned}\\
        &\begin{aligned}\label{E:G_3}
            |\nabla^2G(Q)|\leq a_1\Big(1+\frac{1}{|Q|^{\beta_1+1}}\Big),
        \end{aligned}
    \end{align}
    where $a_1$ is given by Assumption~\ref{Assumption U}.

\textup{(ii)} In the classical case \eqref{E:Classical_Eq}, we assume that there exist constants $\beta_2\in[0,\beta_1)$, $a_4>0$, $a_5\in\R$ and $a_6>0$ such that 
    \begin{equation} \label{cond:G:|grad.G(x)+q/|x|^beta_1|<1/|x|^beta_2:Classic}
        \Big|\nabla G(Q) +a_4\frac{Q}{|Q|^{\beta_1+1}}\Big| \le a_5\frac{1}{|Q|^{\beta_2}} +a_6, \quad Q\in\R^d\setminus\{0\}.
    \end{equation}

\textup{(ii')} In the relativistic case \eqref{E:RelativisticSys:epsilon}, inequality \eqref{cond:G:|grad.G(x)+q/|x|^beta_1|<1/|x|^beta_2:Classic} is replaced by a stronger condition
    \begin{equation} \label{cond:G:|grad.G(x)+q/|x|^beta_1|<1/|x|^beta_2:Relativistic}
\Big|\nabla G(Q) +a_4\frac{Q}{|Q|^{\beta_1+1}}+a_5\frac{Q}{|Q|^{\beta_2+1}}\Big| \le a_6, \quad Q\in\R^d\setminus\{0\}.
\end{equation}
\end{assumption}
While Assumption \ref{A:G} (G) will be employed throughout of the analysis, in the ergodic and mixing rate result for the relativistic model \eqref{E:RelativisticSys:epsilon}, we will have to further restrict the range of $\beta_1$. 

\begin{assumption} \label{A:G:relativistic:N>1}
    Let $\beta_1$ be the positive constant from Assumption \ref{A:G}, we further restrict $\beta_1$ by assuming that $\beta_1\in (1,2]$.
\end{assumption}

By adding a suitable constant to $U$, we can assume without loss of generality that
\begin{align*}
    \sum_{i=1}^{N}U(Q_i)+\sum_{1\le i<j\le N}G(Q_i-Q_j)\ge 1
\end{align*}
\begin{remark} The assumptions on the external and interaction potentials are essentially the same as in \cite{duong2024asymptotic,duong2024trend}. As mentioned in \cite{duong2024asymptotic,duong2024trend}, a routine calculation shows that both the Coulomb and the Lennard-Jones functions satisfy \eqref{cond:G:|grad.G(x)+q/|x|^beta_1|<1/|x|^beta_2:Classic}-\eqref{cond:G:|grad.G(x)+q/|x|^beta_1|<1/|x|^beta_2:Relativistic}. In particular, a well-known example for the case $\beta_1=1$ is the log function $G(x)=-\log|x|$ whereas the case $\beta_1>1$ includes the instance $G(x)=|q|^{-\beta_1+1}$.
\end{remark}
In the classical model \eqref{E:Classical_Eq}, we also need the following assumptions on the diffusion matrix $D$.
\begin{assumption}[D]\label{Assumption D}
The matrix valued function $D  \in C^{1,1}(\R^d,\R^{d\times d})$ satisfies the following:

\textup{(i)} $D(x)$ is Lipschitz with Lipschitz derivative. Furthermore, $D(x)$ is uniformly positive definite and bounded. That is, there exist positive constants $0<\underline{\gamma}\leq \overline{\gamma}<\infty$ such that for all $\xi\in \R^d$ and all $x\in \R^d$,
    \begin{align}\label{E:UniformElliptic}
        \underline{\gamma}|\xi|^2\le \innerbig{D(x)\xi}{\xi}\le \overline{\gamma}|\xi|^2.
    \end{align}

\textup{(ii)} If $\lambda=1$, where $\lambda$ is given by Assumption~$(U)$, we further assume that the derivative of $D$ has sublinear growth rate, which means that
    \begin{align*}
        \Big\|\frac{\partial}{\partial x_j}D(x)\Big\|=o(|x|) 
    \end{align*}
    as $|x|$ goes to infinity, for any $j= 1,\dots ,d$.

\textup{(iii)} In the classical model \eqref{E:Classical_Eq}, when $\beta_1=1$, we further assume that constant $a_4$ as in Assumption~\ref{A:G} (G) (ii) satisfies
    \begin{align}\label{E:ConditionA4}
        a_4+d-2\bar{\gamma}\underline{\gamma}^{-1}>0.
    \end{align}
\end{assumption}
\begin{remark}
    The condition~\eqref{E:ConditionA4} is technical and we use it in the proof of Lemma~\ref{L_:Step2Goal} in the special case when $\beta_1=1$. In turn, the result of Lemma~\ref{L_:Step2Goal} will be employed to establish the small mass limit.
\end{remark}
Now, we turn to the topic of large-time stability and denote by $\mathcal{L}_N$ the generator of the general Langevin dynamics \eqref{eq: general model}. As mentioned in the introduction, to prove the ergodicity and the convergence rate to equilibrium for the two specific Langevin dynamics \eqref{E:Classical_Eq} and \eqref{E:RelativisticSys:epsilon}, we will employ the framework of \cite{hairer2009hot} which consists of three main ingredients: a minorization condition, the solvability of an associated control problem and the construction of Lyapunov functions. For the completeness, we provide the definitions of these concepts adapted to the system \eqref{eq: general model} below.
\begin{definition}[Lyapunov function]\label{D:Lyapunov function}
    A function $V\in C^2\lrbig{\X;[1,\infty)}$ is called a Lyapunov function for~\eqref{eq: general model} if both of the following conditions are satisfied:
    \begin{enumerate}
        \item $V(X)\to \infty$ whenever $|X|+\sum_{1\le i<j\le N}|Q_i-Q_j|^{-1}\to \infty$ in $\X$;
        \item for all $X\in \X$, there exist constants $\alpha\in(0,1]$, $c_1>0$  and $c_2\ge 0$ independent of $X$ such that
        \begin{align}\label{E:Lyapunov function}
            \mathcal{L}_N V(X)\le -c_1V(X)^\alpha+c_2.
        \end{align}
   
\end{enumerate}
   \end{definition}
The value of $\alpha$ will determine the speed of convergence of the process to equilibrium, $\alpha=1$ gives rise to an exponential rate whereas $\alpha\in(0,1)$ yields an algebraic rate. In particular, we will show that the classical model \eqref{E:Classical_Eq} satisfies the former case, while the relativistic model \eqref{E:RelativisticSys:epsilon} corresponds to the latter.
\begin{definition}[Minorization condition]\label{D:MinorizationCondition}
    Let $V$ be a Lyapunov function as in Definition~\ref{D:Lyapunov function}. Denote
    \begin{align*}
        \X_R=\big\{X\in \X:V(X)\le R\big\}.
    \end{align*}
    The system~\eqref{eq: general model} is said to satisfy a minorization condition if for all $R$ sufficiently large,  there exist positive constants $t_R$ and $c_R$, a probability measure $\nu_R$ such that $\nu_R(\X_R)=1$ and for every $X\in \X_R$ and any Borel set $A\subset X$,
    \begin{align*}
        P_{t_R}(X,A)\ge c_R\nu_R(A).
    \end{align*}
\end{definition}
Suppose that the generator $\mathcal{L}_N$ of the system~\eqref{eq: general model} can be written of the form
\begin{align}\label{E:HormanderVector}
    \mathcal{L}_N=Y^0+\sum_{i=1}^N(Y^{i})^2,
\end{align}
where $Y^0, Y^{i,k}$ ($i=1,\ldots, N$ and $k=1,\ldots, d$) are some vector fields. H\"ormander's condition for \eqref{eq: general model} is defined as follows.
\begin{definition}[Hormander's condition]\label{D:Hormander's condition}
    The collection of vector fields $\{Y^i\}_{i=1}^N$ given by~\eqref{E:HormanderVector} is called to satisfy H\"ormander's condition if the Lie algebra generated by 
    \begin{align*}
        \{Y^i\}_{i=1}^N,\{[Y^i,Y^j]\}_{i,j=0}^N,\{[[Y^i,Y^j],Y^k]\}_{i,j,k=0}^N\},\dots,
    \end{align*}
    has maximal rank at every $X\in\X$. Here, we recall $[X,Y](f)=XY(f)-YX(f)$ is the Lie bracket of $X$ and $Y$.
\end{definition}
The last definition describes the solvability of the following system
\begin{align}\label{E:ControlledSys}
\ud q_i(t) &= \nabla K(p_i)\, \ud t, \notag \\
\ud p_i(t) & = - D(q_i,p_i)\nabla K(p_i)\, \ud t + \div_{p_i}(D(q_i,p_i))\ud t  +\sqrt{2 D(q_i,p_i)} \ud V_{i}(t)\notag \\
& \qquad -\nabla U(q_i)\, \ud t- \sum_{j\neq i}\nabla G\big(q_i-q_j\big) \, \ud t,
\end{align}
where $V_i(t)$ is a deterministic controlled path.
\begin{definition}[Solvability]\label{D:Solvable}
    The controlled system  \eqref{E:ControlledSys} is called solvable if for all $X_0,X_1\in\X$, there exists $T>0$ and paths $\{V_i(t)\}_{i=1}^N$ such that the solution $Z(t)$ of the ODE system~\eqref{E:ControlledSys} with initial value $Z(0)=X_0$ satisfies $Z(T)=X_1$.
\end{definition}
 \subsection{Preliminary results}
\label{sec: wellposedness} Having introduced the terminologies related to the ergodicity, we assert the following result establishing the well-posedness, the minorization condition and the controllability of the general system \eqref{eq: general model}.
\begin{theorem}\label{T:ExistandUnique}
    Suppose that Assumptions \ref{Assumption U} and \ref{A:G} hold and that $D$ either satisfies Assumption \ref{Assumption D} in the classical equation \eqref{E:Classical_Eq} or is given by~\eqref{E:DefD} in the relativistic model \eqref{E:RelativisticSys:epsilon}. Then, for every $X_0\in\X$, there exists unique strong solutions to both equations \eqref{E:Classical_Eq} and \eqref{E:RelativisticSys:epsilon}. Furthermore, the minorization condition holds.
\end{theorem}

The proof of the well-posedness is relatively standard, following immediately by using the Hamiltonian structures presented in Proposition~\ref{P:CLyapunov}, Proposition~\ref{P:RLyapunovSing} and Proposition~\ref{P:RLyapunovMulti}. Thus, we will omit the explicit argument and refer the reader to \cite{veretennikov2024lyapunov} for a more detailed discussion. Moreover, the minorization condition is a consequence of H\"ormander's condition proven in Lemma~\ref{L:HormanderCondition} and the solvability of the controlled system verified through Lemma~\ref{L:Solvability} (see e.g.,~\cite{mattingly2002ergodicity}).

\begin{lemma}[Hormander's condition]\label{L:HormanderCondition}
    Under the same hypothesis of Theorem \ref{T:ExistandUnique}, both \eqref{E:Classical_Eq} and \eqref{E:RelativisticSys:epsilon} satisfy H\"ormander's condition as in Definition \ref{D:Hormander's condition}.
\end{lemma}
\begin{proof} With regard to the classical equation~\eqref{E:Classical_Eq}, we define the vector fields $Y^0, Y^{i,k}$ for $1\le i\le N$ and $1\le k\le d$ as
\begin{align*}
    Y^0=&\sum_iv_i\nabla_{x_i}-\frac{1}{m}\sum_{i=1}^N D(x_i)v_i\nabla_{v_i}-\frac{1}{m}\sum_{i=1}^N\nabla U(x_i)\nabla_{v_i}\\
    &-\frac{1}{m}\sum_{1\le i<j\le N}\nabla G(x_i-x_j)[\nabla_{v_i}-\nabla_{v_j}],\\
Y^{i,k}=&\sum_{l=1}^d\Big[\sqrt{2D(x_i)}\Big]_{kl}\nabla_{v_i^l}.
\end{align*}
Denoting
the row vector $Y^i=\lrbig{Y^{i,1},\dots,Y^{i,d}}^t$, we have
\begin{align*}
    \lrbig{Y^1,\dots,Y^N}^t=\sqrt{2\D(\x)}\mathbf{\nabla}_{\v},
\end{align*}
where $\sqrt{2\D(\x)}$ is the block diagonal matrix $\sqrt{2\D(\x)}=\diag(\sqrt{2D(x_1)},\dots,\sqrt{2D(x_N)})$ which is full-ranked and $\nabla_{\v}=\lrbig{\nabla_{v_1},\dots,\nabla_{v_N}}$ forms a vector basis in $\R^{Nd}$. Moreover,
\begin{align*}
    Z^{i,k}\coloneqq\Big[Y^{i,k},Y^0\Big]=\sum_{l=1}^d\Big[\sqrt{2D(x_i)}\Big]_{kl}\nabla_{x_i^l}-\sum_{l=1}^d\Big[{R}(x_i,v_i)\Big]_{kl}\nabla_{v_i^l}
\end{align*}
where $\big[{R}(x_i,v_i)\big]_{kl}=\big[\sqrt{2D(x_i)}D(x_i)\big]_{kl}-v_i\nabla_{x_i}\big[\sqrt{2D(x_i)}\big]_{kl}$. Similarly, by writing $Z^i=\lr{Z^{i,1},\dots,Z^{i,d}}^t$, we get
\begin{align*}
    \lrbig{Z^1,\dots,Z^N,Y^1,\dots,Y^N}^t=\begin{pmatrix}
        \sqrt{2\D(\x)}&\mathbf{R(\x,\v)}\\
        0&\sqrt{2\D(\x)}
    \end{pmatrix}\begin{pmatrix}
        \nabla_{\x}\\
        \nabla_{\v}
    \end{pmatrix},
\end{align*}
where $\mathbf{R(\x,\v)}=\diag\lrbig{R(x_1,v_1),\dots,R(x_N,v_N)}$. The H\"ormander's condition follows from the positivity of $\sqrt{2\D(\x)}$ at any point $\x\in \mathcal{D}$.

Concerning the relativistic model \eqref{E:RelativisticSys:epsilon}, define $Y^0, Y^{i,k}$ for $1\le i\le N$ and $1\le k\le d$
\begin{align*}
    Y^0=&\sum_{i=1}^N\frac{p_i}{\sqrt{1+\varepsilon|p_i|^2}}\nabla_{q_i}-\sum_{i=1}^N D(p_i)\frac{p_i}{\sqrt{1+\varepsilon|p_i|^2}}\nabla_{p_i}+\sum_{i=1}^N\div_{p_i}D(p_i)\nabla_{p_i}\\
    &-\sum_{i}\nabla U(q_i)\nabla_{p_i}-\sum_{1\le i<j\le N}\nabla G(q_i-q_j)[\nabla_{p_i}-\nabla_{p_j}],\\
    Y^{i,k}=&\sum_{l=1}^d\big[\sqrt{2D(p_i)}\big]_{kl}\nabla_{p_i^l}.
\end{align*}
Since the proof in this case employs a similar argument as in the classical case, we will omit the redundant details. In particular, the treatment of $Y^{i,k}$ is the same as in the previous case giving the identity
\begin{align*}
    \lrbig{Y^1,\dots,Y^N}^t=\sqrt{2\D(\p)}\mathbf{\nabla}_{\p},
\end{align*}
where $Y^i=\lrbig{Y^{i,1},\dots,Y^{i,d}}^t$. Also, we have the identity  
\begin{align*}
Z^{i,k}&\coloneqq\Big[Y^{i,k},Y^0\Big]\\
&=\sum_{l=1}^d\big[\sqrt{2D(p_i)}\big]_{kl}\lr{\frac{1}{\sqrt{1+\varepsilon|p_i|^2}}\nabla q_i^l-\frac{\varepsilon p_i^l }{\lrbig{1+\varepsilon|p_i|^2}^{\frac{3}{2}}}p_i\nabla q_i}+ \big[R_{}(q_i,p_i)\nabla _{p_i}\big]_{k},
\end{align*}
for certain matrices $R(q_i,p_i)$. By writing $Z^i=\lr{Z^{i,1},\dots,Z^{i,d}}^t$, we get
\begin{align}\label{E_:DefZ}
    \notag Z^i=&\frac{\sqrt{2D(p_i)}\lrbig{\lrbig{1+\varepsilon|p_i|^2}I-\varepsilon p_i\otimes p_i}}{\lrbig{1+\varepsilon|p_i|^2}^{\frac{3}{2}}}\nabla_{q_i}+R(q_i,p_i)\nabla_{p_i}\\
    \eqqcolon& A(p_i)\nabla_{q_i}+R(q_i,p_i)\nabla_{p_i}.
\end{align}
Alternatively, we may recast the above equations as
\begin{align*}
    \lrbig{Z^1,\dots,Z^N,Y^1,\dots,Y^N}^t=\begin{pmatrix}
        \mathbf{A}(\p)&\mathbf{R(\p,\q)}\\
        0&\sqrt{2\D(\q)}
    \end{pmatrix}\begin{pmatrix}
        \nabla_{\q}\\
        \nabla_{\p}
    \end{pmatrix},
\end{align*}
where $\mathbf{A}(\p)=\diag\lr{A(p_1),\dots,A(p_N)}$.
Given that $\sqrt{\D(\p)}$ is of full rank, it suffices to show that $A(p_i)$ is positive definite for all $i=1,\dots,N$. To see this, recall $A(p_i)$ is given in~\eqref{E_:DefZ},
\begin{align*}
    A(p_i)=\frac{\sqrt{2D(p_i)}}{\lrbig{1+\varepsilon|p_i|^2}^{\frac{3}{2}}}\lrbig{\lrbig{1+\varepsilon|p_i|^2}I-\varepsilon p_i\otimes p_i}\eqqcolon B(p_i)C(p_i).
\end{align*}
On the one hand, it is clear that $B(p_i)$ is positive definite. On the other hand, since $C(p_i)$ is symmetric with spectrum $\{1,1+\varepsilon|p_i|^2\}$, $C(p_i)$ is also positive definite. Moreover, recalling the matrix $D$ from~\eqref{E:DefD} together with the square root matrix $\sqrt{D}$ from~\eqref{E:sqrtD}, we see that both $B(p_i)$ and $C(p_i)$ are of the form
\begin{align*}
    \alpha I+\beta p_i\otimes p_i,
\end{align*}
for some $\alpha,\beta\in \R$ depending on $p_i$. As a consequence, we deduce
\begin{align*}
    B(p_i)C(p_i)=C(p_i)B(p_i)=A(p_i).
\end{align*}
In turn, this implies that $A(p_i)$ is positive definite by combining the positivity of $B(p_i)$ and $C(p_i)$. The proof is thus finished.
\end{proof}

\begin{lemma}[Solvability]\label{L:Solvability}
    Under the same hypothesis of Theorem \ref{T:ExistandUnique}, for both \eqref{E:Classical_Eq} and \eqref{E:RelativisticSys:epsilon}, the controlled system~\eqref{E:ControlledSys} satisfies the solvability condition as in Definition~\ref{D:Solvable}.
\end{lemma}
\begin{proof}
    We present the proof only for the relativistic case \eqref{E:RelativisticSys:epsilon}, as the analogous proof for the classical model \eqref{E:Classical_Eq} is comparatively simpler. Let $K(p)=\frac{1}{\varepsilon}\sqrt{1+\varepsilon|p|^2}$ and $D(p)$ be given as~\eqref{E:DefD}. For any $X_0=(\q_0,\p_0)\in \X$, in view of~\cite[Equation~(3.29)]{duong2024trend}, we see that there exists $T>0$ and a well defined path $\big\{\q(t),\p(t)\big\}_{t\in[0,T]}\subset\X$  such that 
    \begin{align*}
        q_i(0)=q_{0,i},\quad p_i(0)=p_{0,i},\quad q_i(T)=q_{1,i},\quad p_{i}(T)=p_{1,i},
    \end{align*}
    and
    \begin{align*}
        \frac{\ud q_i(t)}{\ud t} =\nabla K(p_i(t)).
    \end{align*}
    We define $V_i(t)$ as
    \begin{align*}
        V_i(t)=\frac{1}{\sqrt{2}}\int_0^t &\lr{\sqrt{D(p_i(s))}}^{-1}\bigg[D(p_i(s))\frac{p_i(s)}{\sqrt{1+\varepsilon|p_i(s)|^2}}-\div_{p_i}\lrbig{D(p_i(s))}\\
        &+\nabla U(q_i(s))+\sum_{j\ne i}\nabla G\lrbig{q_i(s)-q_j(s)}\bigg]\ud s 
    \end{align*}
    By construction, the path $
        \big\{\q(t),\p(t),\mathbf{V}(t)\big\}_{t\in [0,T]}
    $ solves the controlled system~\eqref{E:ControlledSys} where $\mathbf{V(t)}=\lrbig{V_1(t),\dots,V_N(t)}$. The proof is thus completed.
\end{proof}

\section{Langevin dynamics with state-dependent friction}
\label{sec: classical LE}
In this section, we  establish the ergodicity and the small mass limit for the classical equation \eqref{E:Classical_Eq}. More specifically, in Section \ref{sec: classical LE:Ergodicity}, we detail the construction of Lyapunov functions and discuss the proof of Theorem \ref{thm: ergodicityLE} whereas in Section \ref{sec: classical LE:Small-mass}, we present the argument for the small mass limit and prove Theorem \ref{thm: smallLE}. To avoid confusion, we denote the infinitesimal generator of the classical system~\eqref{E:Classical_Eq} as $\mathcal{L}^m_N$, which is given by
\begin{equation}\label{classical N particle generator}
    \begin{aligned}
        \mathcal{L}^m_N f=&\sum_{i=1}^N v_i\cdot\nabla_{x_i} f-\frac{1}{m}\sum_{i=1}^N \Big(\nabla U(x_i)+\sum_{j\neq i}\nabla G\big(x_i-x_j\big) \Big)\cdot\nabla_{v_i} f\\
    &-\frac{1}{m}\sum_{i=1}^N D(x_i)v_i\cdot \nabla_{v_i}f+\frac{1}{m^2}\sum_{i=1}^N\mathrm{div}_{v_i}(D(x_i)\nabla_{v_i} f),
    \end{aligned}
\end{equation}
for $f\in C^2(\X)$. Recall that \eqref{E:Classical_Eq} possesses the following Hamiltonian structure
\begin{equation}\label{classical N-particle Hamiltonian}
    H_{N,m}(\x,\v)=\frac{1}{2}m|\v|^2+ \sum_{i=1}^N U(x_i)+\frac{1}{2}\sum_{1\leq i\neq j\leq N} G(x_i-x_j).
\end{equation}
Also, in order to distinguish $\varepsilon$ given in~\eqref{E:RelativisiicKD}, we use $\varepsilon_1$ to denote the auxiliary constant in the construction of Lyapunov function which would be chosen sufficiently small. The choice of $\varepsilon_1$ may be different between different lemmas/propositions.

\subsection{Ergodicity} \label{sec: classical LE:Ergodicity}
In the following proposition, we construct a Lyapunov function for the classical Langevin dynamics \eqref{E:Classical_Eq}. This together with the minorization condition detailed in Section \ref{sec: prim} ultimately concludes the uniqueness of the invariant probability measure $\pi^m_{\textup{BG}}$, as well as the geometric mixing rate stated in Theorem \ref{thm: ergodicityLE}.

\begin{proposition}\label{P:CLyapunov}
Under the same hypothesis of Theorem \ref{thm: ergodicityLE}, let $H_{N,m}$ be given in \eqref{classical N-particle Hamiltonian} and $V_{N,m}$ be defined as
\begin{equation} \label{form:V_{N,m}:Classic-eq}
    V_{N,m}=H_{N,m}+\sum_{i=1}^N\varepsilon_1 m\langle x_i,v_i\rangle-\sum_{i=1}^{N}\varepsilon_1 m \Big\langle v_i, \sum_{j\neq i}\frac{x_i-x_j}{|x_i-x_j|}\Big\rangle.
\end{equation}
Then, for all $\varepsilon_1=\varepsilon_1(m)>0$ sufficiently small, $V_{N,m}$ is a Lyapunov function to~\eqref{E:Classical_Eq}, i.e.,
\begin{align} \label{ineq:L_N.V_{N,m}:Classic-eq}
    \mathcal{L}_N^m V_{N,m}(X)\le -c_1 V_{N,m}(X)+C_1,\quad X\in\X,
\end{align}
for some positive constants $c_1=c_1(\varepsilon_1,m)$, $C_1=C_1(\varepsilon_1,m)$ independent of $X$ .
\end{proposition}

\begin{proof}
    We first notice that since $\lambda\ge 1$, for $\varepsilon_1$ sufficiently small, it holds that
    \begin{align*}
        \sum_{i=1}^N\varepsilon_1 m\la x_i,v_i\ra-\sum_{i=1}^{N}\varepsilon_1 m \Big\langle v_i, \sum_{j\neq i}\frac{x_i-x_j}{|x_i-x_j|}\Big\rangle &\ge -\frac{\varepsilon_1 m}{2}\lr{\sum_{i=1}^N\lr{|x_i|^2+|v_i|^2+2|v_i|} } \\
        &\ge -\frac{1}{2} H_{N,m}-C,
    \end{align*}
    whence
    \begin{align}\label{EC:Vn>=Hn}
        V_{N,m}\ge \frac{1}{2}H_{N,m}-C
    \end{align}
    which shows that $V_{N,m}\to \infty$ whenever $|X|+\sum_{1\leq i<j\leq N}|x_i-x_j|^{-1}\to \infty$. Now we proceed to estimate $\mathcal{L}_N^m V_{N,m}$ in order to produce the Lyapunov bound \eqref{ineq:L_N.V_{N,m}:Classic-eq}.
    
    With regard to $H_{N,m}$, we have
    \begin{align} \label{eqn:L_N.H_{N,m}:Classic-eq}
        \mathcal{L}_N^m H_{N,m}&=-\sum_{i=1}^N \langle D(x_i)v_i,v_i\rangle+\frac{1}{m}\sum_{i=1}^N \tr(D(x_i))=-\langle\D(\x)\v,\v \rangle+\frac{1}{m}\tr(\D(\x)).
    \end{align}
    We obtain from Assumption~\ref{Assumption D}, c.f.~\eqref{E:UniformElliptic} that
    \begin{align*}
        \langle -D(x_i)v_i,v_i\rangle\leq -\underline{\gamma}|v_i|^2.
    \end{align*}
    For any $x\in\mathcal{D}$, let $\{\lambda_{j}(x)\}_{j=1}^d$ be the eigenvalues of $D(x)$ such that $\lambda_{1}(x)\leq \lambda_{2}(x)\leq \dots \le \lambda_{d}(x)$. From~\eqref{E:UniformElliptic}, we have
    \begin{align*}
        \underline{\gamma}\le \lambda_j(x)\le \bar{\gamma} \quad\text{for all $j=1,\dots,d$ and $x\in \mathcal{D}$.}
    \end{align*}
    As a result, $\tr(D(x))=\sum_{j=1}^d \lambda_j(x)\leq \overline{\gamma}d$ and we obtain the bound
    \begin{align}\label{EC:LnHResult}
        \mathcal{L}_N^m H_{N,m}\leq -\underline{\gamma}|\v|^2+\frac{\overline{\gamma}d}{m}.
    \end{align}

    Next, applying the generator $\mathcal{L}_N^m$ to the second term of $V_{N,m}$ from \eqref{form:V_{N,m}:Classic-eq}, we get
    \begin{align*}
        \mathcal{L}_N^m\Big(\sum_{i=1}^N\varepsilon_1 m \langle x_i, v_i\rangle\Big)&=m\varepsilon_1|\v|^2-\varepsilon_1\langle\nabla U(\x),\x\rangle-\varepsilon_1\sum_{i=1}^d\langle D(x_i)v_i,x_i\rangle\\
        &\qquad-\varepsilon_1\sum_{1\leq i<j\leq N}\langle \nabla G(x_i-x_j),x_i-x_j \rangle.
    \end{align*}
    We apply Cauchy-Schwarz inequality together with~\eqref{E:G_2} to get
    \begin{align*}
        \varepsilon_1\sum_{1\leq i<j\leq N}\langle \nabla G(x_i-x_j),x_i-x_j \rangle\le \varepsilon_1\sum_{1\le i<j\le N}\lr{|x_i-x_j|+\frac{1}{|x_i-x_j|^{\beta_1-1}}}.
    \end{align*}
    Again by Cauchy-Schwarz inequality, for any $\varepsilon_1>0$, we have
    \begin{align*}
        -\langle D(x_i)v_i,x_i\rangle\leq |D(x_i)v_i||x_i|\leq \frac{1}{2}(\frac{1}{\sqrt{\varepsilon_1}}|D(x_i)v_i|^2+\sqrt{\varepsilon_1}|x_i|^2).
    \end{align*}
    Since $D(x_i)$ is symmetric, the spectral norm is the largest eigenvalue $\lambda_d(x_i)$ which is bounded by $\bar{\gamma}$,
    \begin{align*}
        |D(x_i)v_i|^2\le \|D(x_i)\|^2|v_i|^2\le \bar{\gamma}^2|v_i|^2.
    \end{align*}
    From the above estimates and Assumption~\ref{Assumption U} and Assumption~\ref{A:G} respectively on the potentials $U$ and $G$, we have
    \begin{align}\label{EC:LnCrossResult}
                    &\mathcal{L}_N^m\Big(\sum_{i=1}^N\varepsilon_1 m \langle x_i, v_i\rangle\Big) 
                    \\
                    &\leq m\varepsilon_1|\v|^2-\varepsilon_1\sum_{i=1}^N a_2|x_i|^{\lambda+1}+\varepsilon_1  a_3+\varepsilon_1\sum_{1\leq i<j\leq N}\Big(\frac{a_1}{|x_i-x_j|^{\beta_1-1}}+a_1|x_i-x_j|\Big) \notag \\
         &\qquad
        +\varepsilon_1 C\sum_{i=1}^N\Big(\frac{1}{\sqrt{\varepsilon_1}}|v_i|^2+\sqrt{\varepsilon_1}|x_i|^2\Big) \notag\\
        &\leq m\varepsilon_1\sum_{i=1}^N|v_i|^2-\varepsilon_1\sum_{i=1}^N a_2|x_i|^{\lambda+1}+\varepsilon_1  a_3+\varepsilon_1\sum_{1\leq i<j\leq N}\Big(\frac{a_1}{|x_i-x_j|^{\beta_1-1}}\Big)  \notag \\
         &\qquad
        +\varepsilon_1 C\sum_{i=1}^N\Big(\frac{1}{\sqrt{\varepsilon_1}}|v_i|^2+\sqrt{\varepsilon_1}|x_i|^2+|x_i|\Big).
    \end{align}
    
    Concerning the third term of $V_{N,m}$ in \eqref{form:V_{N,m}:Classic-eq}, we have
    \begin{align}\label{E:LnSingu}
        \notag &\mathcal{L}_N^m\Big(-\varepsilon_1 m \sum_{i=1}^N\Big\langle v_i,\sum_{j\neq i}\frac{x_i-x_j}{|x_i-x_j|}\Big\rangle\Big)\\
        \notag&=-\varepsilon_1 m\sum_{1\le i<j\le N}\frac{|v_i-v_j|^2}{|x_i-x_j|}+\varepsilon_1 m \sum_{1\le i<j\le N}\frac{|\langle v_i-v_j, x_i-x_j\rangle|^2}{|x_i-x_j|^3}\\
        &\notag \quad +\varepsilon_1\sum_{i=1}^N \langle D(x_i)v_i,\sum_{j\neq i}\frac{x_i-x_j}{|x_i-x_j|}\rangle +\sum_{1\leq i<j\leq N}\frac{\langle \nabla U(x_i)-\nabla U(x_j),x_i-x_j\rangle}{|x_i-x_j|}\\
        &\quad +\varepsilon_1 \sum_{i=1}^N \Big\langle\sum_{j\neq i}\nabla G(x_j-x_i),\sum_{l\neq i}\frac{x_i-x_l}{|x_i-x_l|}\Big\rangle.
    \end{align}
    By Cauchy-Schwarz inequality, we first notice that
    \begin{equation*}
    -\varepsilon_1 m \sum_{1 \leq i < j \leq N} \frac{|v_i - v_j|^2}{|x_i - x_j|} + \varepsilon_1 m \sum_{1 \leq i < j \leq N} \frac{\langle v_i - v_j, x_i - x_j \rangle^2}{|x_i - x_j|^3} \leq 0,
\end{equation*} 
By the boundedness of $D$,
    \begin{align*}
        \varepsilon_1\sum_{i=1}^N \langle D(x_i)v_i,\sum_{j\neq i}\frac{x_i-x_j}{|x_i-x_j|}\rangle\leq \sum_{i=1}^N |D(x_i)v_i|\cdot\Big|\sum_{j\neq i}\frac{x_i-x_j}{|x_i-x_j|}\Big|\leq \varepsilon_1 \overline{\gamma}(N-1)\sum_{i=1}^N |v_i|.
     \end{align*}
     Turning to the singular term involving $\nabla U $, we apply estimate~\eqref{cond:U:U'(x)=O(x^lambda)} to get 
     \begin{equation*}
    \sum_{1 \leq i < j \leq N} \frac{\langle \nabla U(x_i) - \nabla U(x_j), x_i - x_j \rangle}{|x_i - x_j|} \leq (N - 1)a_1 \Big( N + \sum_{i=1}^{N} |x_i|^\lambda \Big).
    \end{equation*}
     Concerning the crossing term involving $\nabla G$, we decompose this term as follows
\begin{equation*}
\begin{aligned}
    &\sum_{i=1}^{N} \Big( \sum_{j \neq i}\frac{x_i - x_j}{|x_i - x_j|}, \sum_{l \neq i}\nabla G(x_i - x_l) \Big) \\
    &= -a_4 \sum_{i=1}^{N} \Big\langle \sum_{j \neq i} \frac{x_i - x_j}{|x_i - x_j|},\sum_{l \neq i} \frac{x_i - x_l}{|x_i - x_l|^{\beta + 1}} \Big\rangle \\
    &\quad + \sum_{i=1}^{N} \Big\langle \sum_{j \neq i} \frac{x_i - x_j}{|x_i - x_j|}, \sum_{l \neq i} \nabla G(x_i - x_l) + a_4 \frac{x_i - x_l}{|x_i - x_l|^{\beta + 1}} \Big\rangle.
\end{aligned}
\end{equation*}
We apply Lemma~\ref{L:s+1times1} to find that
\begin{equation*}
    -a_4 \sum_{i=1}^{N} \Big\langle  \sum_{j \neq i} \frac{x_i - x_j}{|x_i - x_j|} ,  \sum_{l \neq i} \frac{x_i - x_l}{|x_i - x_l|^{\beta_1 + 1}}  \Big\rangle \leq -2a_4 \sum_{1 \leq i < j \leq N} \frac{1}{|x_i - x_j|^{\beta_1}}.
\end{equation*}
For the second term, recall~\eqref{cond:G:|grad.G(x)+q/|x|^beta_1|<1/|x|^beta_2:Classic} from Assumption~\ref{A:G},
\begin{align*}
\sum_{i=1}^N \Big\langle \sum_{j \neq i} \frac{x_i - x_j}{|x_i - x_j|}, \sum_{l \neq i} \nabla G(x_i - x_l) + a_4 \frac{x_i - x_l}{|x_i - x_l|^{\beta_1 + 1}} \Big\rangle \\
\leq (N - 1) \sum_{i=1}^N \sum_{l \neq i} \Big| \nabla G(x_i - x_l) + a_4 \frac{x_i - x_l}{|x_i - x_l|^{\beta_1 + 1}} \Big| \\
\leq (N - 1) \Big[ 2 \sum_{1 \leq i < l \leq N} \frac{a_5}{|x_i - x_l|^{\beta_2}} + N(N - 1)a_6 \Big].
\end{align*}
Since $\beta_2<\beta_1$, we can subsume the constants and $|x_i-x_l|^{-\beta_2}$ terms into $|x_i-x_l|^{-\beta_1}$ to see from~\eqref{E:LnSingu} that
\begin{equation}\label{EC:LnSinguReusult}
    \begin{aligned}
        &\mathcal{L}_N^m\Big(-\varepsilon_1 m \sum_{i=1}^N \Big\langle v_i, \sum_{j\neq i} \frac{x_i - x_j}{|x_i - x_j|}\Big\rangle\Big)\\
&\qquad\leq -a_4\varepsilon_1 \sum_{1\leq i<j\leq N} \frac{1}{|x_i - x_j|^{\beta_1}} + C\varepsilon_1\Big(1 + \sum_{i=1}^N |v_i| + \sum_{i=1}^N |x_i|^\lambda \Big).
    \end{aligned}
\end{equation}

Turning back to $\mathcal{L}_N^mV_{N,m}$, we collect~\eqref{EC:LnHResult},~\eqref{EC:LnCrossResult} and~\eqref{EC:LnSinguReusult} to infer
\begin{equation*}
\begin{split}
\mathcal{L}_N^mV_{N,m} &\leq -\underline{\gamma}|v|^2  - a_2\varepsilon_1 \sum_{i=1}^N |x_i|^{\lambda+1} - a_4\varepsilon_1 \sum_{1\leq i<j\leq N} \frac{1}{|x_i - x_j|^{\beta_1}} + C \\
&\quad + C\Big( \varepsilon_1^{1/2}|v|^2 + \varepsilon_1^{3/2}|x|^2 + \varepsilon_1 \sum_{1\leq i<j\leq N} \frac{1}{|x_i - x_j|^{\beta_1-1}} \Big) \\
&\quad + C\varepsilon_1\Big( \sum_{i=1}^N |v_i| + \sum_{i=1}^N |x_i|^{\lambda}  \Big).
\end{split}
\end{equation*}
Consequently, by choosing $\varepsilon_1$ sufficiently small and recalling~\eqref{cond:U:U(x)=O(x^lambda+1)} and~\eqref{E:G_1}, we arrive at the bound
\begin{align*}
    \mathcal{L}_N^mV_{N,m} \leq& -\frac{1}{2}\Big( \gamma|v|^2 + a_2\varepsilon_1 \sum_{i=1}^N |x_i|^{\lambda+1} + a_4\varepsilon_1 \sum_{1\leq i<j\leq N} \frac{1}{|x_i - x_j|^{\beta_1}} \Big) + C\\
    \le &-c_1H_{N,m}+C_2\le -cV_{N,m}+C,
\end{align*}
where the last line follows from~\eqref{EC:Vn>=Hn}. This establishes Lyapunov bound \eqref{ineq:L_N.V_{N,m}:Classic-eq}, thereby finishing the proof.
\end{proof}

\subsection{Small-mass limit} \label{sec: classical LE:Small-mass}
In this subsection, we proceed to prove Theorem \ref{thm: smallLE} establishing the convergence of the classical model \eqref{E:Classical_Eq} toward \eqref{EC:Limiting} in the small mass regime. In order to do so, it is crucial to obtain useful estimates on the nonlinearities. In Lemma \ref{lem:grad.U_grad.G} below, we assert that $\nabla U(q)$ behaves like $|q|^\lambda$ while $\nabla G$ is as singular as $|q|^{-\beta_1}$.

\begin{lemma} \label{lem:grad.U_grad.G}
    Under Assumption~\ref{Assumption U} (U)(i) and Assumption~\ref{A:G} (G)(ii), there exist constants $a_7,a_8,a_9,a_{10}>0$ such that
    
    \begin{align}\label{EC:BdNablaU}
            a_2|q|^\lambda-a_2\vee a_3\le |\nabla U(q)|\le a_1|q|^\lambda+a_1,
    \end{align}
    and
    \begin{align}\label{EC:BoundNablaG}
            a_7\sum_{i\ne j}\frac{1}{|q_i-q_j|^{2\beta_1}}-a_8\le \sum_{i=1}^N \Big| \sum_{j \ne i} \nabla G({q}_i - {q}_j) \Big|^2\le a_9\sum_{i\ne j}\frac{1}{|q_i-q_j|^{2\beta_1}}+a_{10}.
    \end{align}
    In the above, $\lambda$ and $\beta_1$ are respectively the constants from Assumption \ref{Assumption U} and \ref{A:G}.
\end{lemma}
\begin{proof}
    The upper bound of $\nabla U$ is already given by~\eqref{cond:U:U'(x)=O(x^lambda)}. Recalling condition \eqref{cond:U:x.U'(x)>-x^(lambda+1)}, namely
\begin{align*}
    \innerbig{\nabla U(q)}{q}\ge a_2|q|^{\lambda+1}-a_3,
\end{align*}
it holds that
\begin{align*}
    a_2|q|^{\lambda}\le \begin{cases}
        \frac{\innerbig{\nabla U(q)}{q}+a_3}{|q|}\le |\nabla U(q)|+a_3,\quad &\text{if $ |q|>1$},\\
        a_2,\quad&\text{if $|q|\le 1$},
    \end{cases}
\end{align*}
which verifies~\eqref{EC:BdNablaU}, as claimed.

Turning to the interaction potential $G$, we have the following decomposition,
\begin{align*}
    \sum_{i=1}^N \Big| \sum_{j \ne i} \nabla G({q}_i - {q}_j) \Big|^2
    =& \sum_{i=1}^N \Big[ \sum_{j \ne i} \Big( \nabla G ({q}_i - {q}_j) + a_4 \frac{{q}_i - {q}_j}{|{q}_i - {q}_j|^{\beta_1+1}} \Big) - \sum_{j \ne i} a_4 \frac{{q}_i - {q}_j}{|{q}_i - {q}_j|^{\beta_1+1}} \Big]^2 \\
    =& \sum_{i=1}^N \Big[ \sum_{j \ne i} \Big( \nabla G ({q}_i - {q}_j) + a_4 \frac{{q}_i - {q}_j}{|{q}_i - {q}_j|^{\beta_1+1}} \Big) \Big]^2 \\
    &-2  \sum_{i=1}^N \inner{\sum_{j \ne i} \Big( \nabla G ({q}_i - {q}_j) + a_4 \frac{{q}_i - {q}_j}{|{q}_i - {q}_j|^{\beta_1+1}} \Big)}{\sum_{k \ne i} a_4 \frac{{q}_i - {q}_k}{|{q}_i - {q}_k|^{\beta_1+1}}} \\
    &+ a_4 ^2\sum_{i=1}^N \Big[ \sum_{j \ne i} \frac{{q}_i - {q}_j}{|{q}_i - {q}_j|^{\beta_1+1}} \Big]^2\\
    \eqqcolon&K_1+K_2+K_3,
\end{align*}
where
\begin{align*}
    K_1=& \sum_{i=1}^N \Big[ \sum_{j \ne i} \Big( \nabla G ({q}_i - {q}_j) + a_4 \frac{{q}_i - {q}_j}{|{q}_i - {q}_j|^{\beta_1+1}} \Big) \Big]^2 ,\\
    K_2=&-2  \sum_{i=1}^N \inner{\sum_{j \ne i} \Big( \nabla G ({q}_i - {q}_j) + a_4 \frac{{q}_i - {q}_j}{|{q}_i - {q}_j|^{\beta_1+1}} \Big)}{\sum_{k \ne i} a_4 \frac{{q}_i - {q}_k}{|{q}_i - {q}_k|^{\beta_1+1}}},\\
    K_3=&a_4 ^2\sum_{i=1}^N \Big[ \sum_{j \ne i} \frac{{q}_i - {q}_j}{|{q}_i - {q}_j|^{\beta_1+1}} \Big]^2.
\end{align*}
By Assumption \ref{A:G} (ii),
\begin{align*}
    |K_1|\le \sum_{i\neq j} \Big|\frac{a_5}{|q_i-q_j|^{\beta_2}}+a_6\Big|^2\le C\sum_{i=1}^N \frac{1}{|q_i-q_j|^{2\beta_2}}+C.
\end{align*}
Likewise,
\begin{align*}
    |K_2|\le& \sum_{i=1}^N\lr{2a_4a_5\sum_{j\neq i}\frac{1}{|q_i-q_j|^{\beta_1}}\sum_{k\neq i}\frac{1}{|q_i-q_j|^{\beta_2}}+2a_4a_6\sum_{k\neq i}\frac{1}{|q_i-q_j|^{\beta_1}}}.
\end{align*}
Concerning $K_3$, we apply Lemma~\ref{L:s+1timess+1} to obtain
\begin{align*}
    |K_3|\ge \frac{2a_4^2}{N(N-1)^2}\sum_{i\ne j}\frac{1}{|q_i-q_j|^{2\beta_1}}.
\end{align*}
Since $\beta_1>\beta_2$, we may subsume both $K_1$ and $K_2$ into $K_3$ to infer for some $a_7,a_8,a_9, a_{10}>0$ that
\begin{align*}
    \sum_{i=1}^N \Big| \sum_{j \ne i} \nabla G({q}_i - {q}_j) \Big|^2\le |K_1|+|K_2|+|K_3|\le a_9\sum_{1\le i<j\le N}\frac{1}{|q_i-q_j|^{2\beta_1}}+a_{10},
\end{align*}
and that
\begin{align*}
    \sum_{i=1}^N \Big| \sum_{j \ne i} \nabla G({q}_i - {q}_j) \Big|^2=K_3+K_1+K_2\ge K_3-|K_1|-|K_2|\ge a_7\sum_{i\ne j}\frac{1}{|q_i-q_j|^{2\beta_1}}-a_8.
\end{align*}
Altogether, we deduce \eqref{EC:BoundNablaG}. The proof is thus finished.
\end{proof}
 Turning to the diffusion matrix $D$, as we will see later, the inverse matrix $D^{-1}$ plays an important role in the small mass limit. In Lemma \ref{L:LipD_inv}, stated and proven next, we assert important properties of $D^{-1}$ which are implied by Assumption~\ref{Assumption D} (D). 
 \begin{lemma}\label{L:LipD_inv}
    Under Assumption~\ref{Assumption D} (D), the followings hold:

    \textup{(i)} The matrix $D^{-1}$ is Lipschitz, uniformly positive definite and bounded. 

    \textup{(ii)} The vector $\div D^{-1}(x)\in \R^d$ is Lipschitz where
    \begin{align*}
        [\div(D^{-1}(x))]_i=\sum_{j=1}^d\frac{\partial  D^{-1}_{ij}}{\partial x_j}(x).
    \end{align*}

    \textup{(iii)} Furthermore, if $\lambda=1$, then $\div D^{-1}(x)$ is also sublinear.
\end{lemma}
\begin{proof}
    \textup{(i)} It is not difficult to see that $D^{-1}$ is uniformly positive definite and bounded thanks to the fact that $0<\overline{\gamma}^{-1}\le \lambda_d^{-1}(x)\le \dots \le \lambda_1^{-1}(x)\le \underline{\gamma}^{-1}<\infty$.
    By the Lipschitz condition on $D$, we have
    \begin{align}\label{E:LipD_inv}
        \notag\|D^{-1}(x)-D^{-1}(y)\|\le& \|D^{-1}(y)\|\|D(x)-D(y)\|\|D^{-1}(x)\|\\
        \le& \frac{L_D}{\underline{\gamma}^{-2}}|x-y|\eqqcolon L_{D^{-1}}|x-y|,
    \end{align}
    where $L_D$ is the Lipschitz constant of $D$. 
    
    \textup{(ii)} Concerning the Lipschitz continuity of $\div D^{-1}$, we {recall} the matrix identity
    \begin{align*}
        \frac{\partial D^{-1}}{\partial x_j}=-D^{-1}\frac{\partial D}{\partial x_j} D^{-1},
    \end{align*}
    whence
    \begin{align}\label{E_:DivD-1}
        [\div D^{-1}(x)]_i=-\sum_j\lr{D^{-1}\frac{\partial D}{\partial x_j}D^{-1}}_{ij}.
    \end{align}
    If $\lambda=1$, $\frac{\partial D}{\partial x_j}$ is sublinear by Assumption~\ref{Assumption D} (D) (ii) and $D^{-1}$ is bounded, we find that $\div D^{-1}$ is sublinear. Moreover,
    \begin{align*}
        \div  D^{-1}(x) - \div  D^{-1}(y) = -\sum_{j=1}^{d} \Big[ D^{-1}(x) \frac{\partial D}{\partial x_j}(x) D^{-1}(x) - D^{-1}(y) \frac{\partial D}{\partial x_j}(y) D^{-1}(y) \Big] e_j.
    \end{align*}
    Consider the matrix norm inside the summation, we have the following decomposition,
    \begin{align*}
        &\Norm{D^{-1}(x) \frac{\partial D}{\partial x_j}(x) D^{-1}(x) - D^{-1}(y) \frac{\partial D}{\partial x_j}(y) D^{-1}(y)}\\
        &\le  \Norm{D^{-1}(x) \frac{\partial D}{\partial x_j}(x) D^{-1}(x)-D^{-1}(y) \frac{\partial D}{\partial x_j}(x) D^{-1}(x)}\\
        &\quad +\Norm{D^{-1}(y) \frac{\partial D}{\partial x_j}(x) D^{-1}(x)-D^{-1}(y) \frac{\partial D}{\partial x_j}(y) D^{-1}(x)}\\
        &\quad +\Norm{D^{-1}(y) \frac{\partial D}{\partial x_j}(y) D^{-1}(x)-D^{-1}(y) \frac{\partial D}{\partial x_j}(y) D^{-1}(y)}\\
        & \eqqcolon D_1+D_2+D_3.
    \end{align*}
    Concerning $D_1$ and $D_3$, we employ the Lipschitz continuity of $D^{-1}$ from~\eqref{E:LipD_inv} together with the boundedness of $D^{-1}$ to infer
    \begin{align*}
        D_1+D_3\le 2L_{D^{-1}}L_D\underline{\gamma}^{-1}|x-y|.
    \end{align*}
    Turning to $D_2$, since $D\in C^{1,1}$
    \begin{align*}
        D_2\le \underline{\gamma}^{-2}L_{\nabla D}|x-y|.
    \end{align*}
    Collecting the above estimates, we obtain
    \begin{align*}
        \big|\div D^{-1}(x) - \div D^{-1}(y) \big|\le d\lr{\underline{\gamma}^{-2}L_{\nabla D}+2L_{D^{-1}}L_D\underline{\gamma}^{-1}}|x-y|,
    \end{align*}
    thereby finishing the proof. 
\end{proof}

Having collected useful estimates on the nonlinearities, we turn to the limiting system \eqref{EC:Limiting}. In what follows, we state and prove exponential moment bounds on \eqref{EC:Limiting}. The results of which will appear later in the proof of Theorem \ref{thm: smallLE}. 
\begin{lemma}\label{L_:Step2Goal}
Under the same hypothesis of Theorem \ref {thm: smallLE}, let $\q(t)$ be the solution of \eqref{EC:Limiting} with initial condition $\q_0\in \mathcal{D}$. Then, for $\varepsilon_1$ and $\kappa$ sufficiently small and for all $T>0$, there exists a positive constant $C=C(\q_0,\varepsilon_1,\kappa)$ such that

\textup{(i)} if $\beta_1>1$,
    \begin{align} \label{ineq:EC:Limiting:beta_1>1}
        \mathbb{E}\Big[\exp\Big\{\sup_{t\in[0,T]}\Big(\kappa\sum_{i=1}^N|U(q_i(t))|+\kappa\varepsilon_1\sum_{1\leq i<j\leq N}G(q_i(t)-q_j(t)\Big)\Big\}\Big]\leq C.
    \end{align}

\textup{(ii)} if $\beta_1=1$,
    \begin{align} \label{ineq:EC:Limiting:beta_1=1}
        \mathbb{E}\Big[\exp\Big\{\sup_{t\in[0,T]}\Big(\kappa\sum_{i=1}^N|U(q_i(t))|-\kappa\varepsilon_1\sum_{1\leq i<j\leq N}\log|q_i(t)-q_j(t)|\Big)\Big\}\Big]\leq C.
    \end{align}

     In the above, $\beta_1$ is the constant as in Assumption \ref{A:G}.

\end{lemma}
\begin{proof} (i) With regard to the case $\beta_1>1$, define
\begin{align}\label{EC:DefGamma1}
    \Gamma_1(t)=\sum_{i=1}^N U(q_i)+\varepsilon_1\sum_{1\le i<j\le N}G(q_i-q_j)
\end{align}
By It\^{o}'s formula, we have that
    \begin{align}
        \ud \lr{\sum_{i} U(q_i)}=&\sum_{i=1}^N \innerbig{\nabla U(q_i)}{\ud q_i}+\sum_{i=1}^N\tr\lr{\nabla^2 U(q_i) D^{-1}(q_i)}\ud t    \notag \\
    =&-\sum_{i=1}^N \inner{\nabla U(q_i)}{D^{-1}(q_i)\nabla U(q_i)}\ud t-\sum_{i}\inner{\nabla U(q_i)}{D^{-1}(q_i)\sum_{j\ne i}\nabla G(q_i-q_j)}\ud t \notag \\
    &-\sum_{i=1}^N \inner{\div_{q_i} D^{-1}(q_i)}{\nabla U(q_i)}\ud t+\sum_{i=1}^N \inner{\nabla U(q_i)}{\sqrt{2D^{-1}(q_i)}\ud W_i} \notag \\
    &+\sum_{i=1}^N\tr\lr{\nabla^2 U(q_i) D^{-1}(q_i)}\ud t\eqqcolon \sum_{k=1}^5 I_k,\label{EC:dU}
    \end{align}
and that
    \begin{align}
        &\ud \lr{\varepsilon_1\sum_{1\le i<j\le N}G(q_i-q_j)} \notag \\
        &=\varepsilon_1\sum_{i=1}^N \inner{\sum_{j\neq i }\nabla G(q_i-q_j)}{\ud q_i}+\varepsilon_1\sum_{i\neq j}\tr \lr{\nabla^2 G(q_i-q_j)\lrbig{D^{-1}(q_i)}}\ud t \notag   \\
    &=\varepsilon_1\sum_{i=1}^N\inner{\sum_{j\neq i}\nabla G(q_i-q_j)}{-D^{-1}(q_i)\nabla U(q_i)}\ud t \notag   \\
    &\quad -\varepsilon_1\sum_{i=1}^N\inner{\sum_{j\neq i}\nabla G(q_i-q_j)}{D^{-1}(q_i)\sum_{k\neq i}\nabla G(q_i-q_k)}\ud t \notag   \\
    &\quad -\varepsilon_1\sum_{i=1}^N\inner{\sum_{j\neq i}\nabla G(q_i-q_j)}{\div_{q_i}D^{-1}(q_i)}\ud t  +\varepsilon_1\sum_{i=1}^N\inner{\sum_{j\neq i}\nabla G(q_i-q_j)}{\sqrt{2D^{-1}(q_i)}\ud W_i} \notag   \\
    &\quad +\varepsilon_1\sum_{i\neq j}\tr \lr{\nabla^2 G(q_i-q_j)\lr{D^{-1}(q_i)}}\ud t\eqqcolon \sum_{k=1}^5 J_k. \label{EC:dG}
    \end{align}
We see that $I_1$ is negative and produces dissipation. That is,
\begin{align*}
    I_1=-\sum_{i=1}^N \inner{\nabla U(q_i)}{D^{-1}(q_i)\nabla U(q_i)}\le -\overline{\gamma}^{-1}\sum_{i=1}^N |\nabla U(q_i)|^2.
\end{align*}
Since $G$ is even by assumption, $\nabla G$ is odd. Therefore, we can recast $I_2$ as follows.
\begin{align}\label{E_:I_2}
    \notag I_2=&\sum_{1\le i<j\le N}\inner{D^{-1}(q_i)\nabla U(q_i)-D^{-1}(q_j)\nabla U(q_j)}{\nabla G(q_i-q_j)}\\
    \notag =&\sum_{1\le i<j\le N}\inner{D^{-1}(q_i)\nabla U(q_i)-D^{-1}(q_j)\nabla U(q_i)}{\nabla G(q_i-q_j)}\\
    &\qquad+\sum_{1\le i<j\le N}\inner{D^{-1}(q_j)\nabla U(q_i)-D^{-1}(q_j)\nabla U(q_j)}{\nabla G(q_i-q_j)} \notag \\
    \eqqcolon& I_{2,1}+I_{2,2}.
\end{align}
Recalling from Lemma~\eqref{L:LipD_inv} that $D^{-1}$ is both bounded and Lipschitz, by Cauchy-Schwarz inequality, it holds that
\begin{align*}
    \big|D^{-1}(q_i)\nabla U(q_i)-D^{-1}(q_j)\nabla U(q_i)\big|\le C\lrbig{|q_i-q_j|\wedge1}|\nabla U(q_i)|.
\end{align*}
So,
    \begin{align*}
        I_{2,1}=&\sum_{1\le i<j\le N}\inner{D^{-1}(q_i)\nabla U(q_i)-D^{-1}(q_j)\nabla U(q_i)}{\nabla G(q_i-q_j)}\\
    \le& \sum_{1\le i<j\le N}C\lrbig{|q_i-q_j|\wedge1}|\nabla U(q_i)|\big|\nabla G(q_i-q_j)\big|\\
    \le &C\sum_{1\le i<j\le N}|\nabla U(q_i)|\lr{1+\frac{1}{|q_i-q_j|^{\beta_1-1}}}\\
    \le &C\sum_{i=1}^N|\nabla U(q_i)|+C\sum_{1\le i<j\le N}\frac{|\nabla U(q_i)|}{|q_i-q_j|^{\beta_1-1}}\\
    \le &C\sum_{i=1}^N\Big(|\nabla U(q_i)|+|\nabla U(q_i)|^{\frac{2\beta_1-1}{\beta_1}}\Big)+C\sum_{1\le i<j\le N}\frac{1}{|q_i-q_j|^{2\beta_1-1}}.
    \end{align*}
We then employ~\eqref{E:G_2} to infer
\begin{equation}\label{E_:I_21}
    I_{2,1} =o\Big(\sum_{i=1}^N|\nabla U(q_i)|^2+\sum_{i}\Big|\sum_{j\ne i}\nabla G(q_i-q_j)\Big|^2\Big).
\end{equation}
Turning to $I_{2,2}$, by the boundedness of $D^{-1}$, together with conditions~\eqref{E:U_3} and~\eqref{E:G_2}, we estimate
\begin{align*}
    I_{2,2}=&\sum_{1\le i<j\le N}\inner{D^{-1}(q_j)\lrbig{\nabla U(q_i)-\nabla U(q_j)}}{\nabla G(q_i-q_j)}\\
    \le& \sum_{1\le i<j\le N}\underline{\gamma}^{-1}\big|\nabla^2 U(q_i)+\nabla^2 U(q_j)\big||q_i-q_j|\,\big|\nabla G(q_i-q_j) \big|\\
    \le &\underline{\gamma}^{-1}\sum_{1\le i<j\le N}\lrbig{1+|q_i|^{\lambda-1}+|q_j|^{\lambda-1}}\lr{|q_i-q_j|+\frac{1}{|q_i-q_j|^{\beta_1-1}}}.
\end{align*}
On the one hand, if $\lambda=1$,
\begin{equation}\label{E_:I22_1}
    \begin{aligned}
        I_{2,2}\le& C\sum_{1\le i<j\le N}\lr{|q_i-q_j|+\frac{1}{|q_i-q_j|^{\beta_1-1}}}\\
    \le& C\sum_{i=1}^N|q_i|+C\sum_{1\le i<j\le N}\frac{1}{|q_i-q_j|^{\beta_1-1}}\\
    =&o\Big(\sum_{i=1}^N|\nabla U(q_i)|^2+\sum_{i}\Big|\sum_{j\ne i}\nabla G(q_i-q_j)\Big|^2\Big).
    \end{aligned}
\end{equation}
On the other hand, when $\lambda>1$,
\begin{equation}\label{E_:I22_2}
    \begin{aligned}
        I_{2,2}\le& C\sum_{1\le i<j\le N}\lr{|q_i|+|q_j|+|q_i|^\lambda+|q_j|^\lambda+|q_i|^{\lambda-1}|q_j|+|q_j|^{\lambda-1}|q_i|}\\
    &+C\sum_{1\le i<j\le N}\frac{1+|q_i|+|q_j|}{|q_i-q_j|^{\beta_1-1}}\\
    \le& C\sum_{i}\lr{|q_i|^\lambda+|q_i|^{2\lambda-2}+|q_i|^2}+C\sum_{1\le i<j\le N}\frac{1}{|q_i-q_j|^{2\beta_1-2}}+C\\
    =&o\Big(\sum_{i=1}^N|\nabla U(q_i)|^2+\sum_{i}\Big|\sum_{j\ne i}\nabla G(q_i-q_j)\Big|^2\Big).
    \end{aligned}
\end{equation}
We now collect estimates~\eqref{E_:I_2}, \eqref{E_:I_21}, \eqref{E_:I22_1} and~\eqref{E_:I22_2} to deduce that
\begin{equation}\label{E_:oI_2}
    I_2=o\Big(\sum_{i=1}^N|\nabla U(q_i)|^2+\sum_{i}\Big|\sum_{j\ne i}\nabla G(q_i-q_j)\Big|^2\Big).
\end{equation}
Since $D^{-1}$ is self-adjoint, $J_1$ shares the same upper bound with $I_2$ when $\varepsilon_1\le 1$.

\begin{equation}\label{E_:oJ_1}
    J_1=\varepsilon_1 I_2=o\Big(\sum_{i=1}^N|\nabla U(q_i)|^2+\sum_{i}\Big|\sum_{j\ne i}\nabla G(q_i-q_j)\Big|^2\Big).
\end{equation}
Concerning $J_2$, we recall from Lemma \ref{L:LipD_inv} (i) that $D^{-1}$ is uniformly positive definite. Denoting
\begin{align*}
    v_i=\sum_{j\neq i}\nabla G(q_i-q_j)=\sum_{k\neq i}\nabla G(q_i-q_k),
\end{align*}
 we obtain
\begin{align*}
    J_2=-\varepsilon_1\sum_{i=1}^N\inner{v_i}{D^{-1}(q_i)v_i}\le -\varepsilon_1\sum_{i=1}^N \overline{\gamma}^{-1}|v_i|^2
    = -\varepsilon_1\overline{\gamma}^{-1}\sum_{i=1}^N |\sum_{j\neq i}\nabla G(q_i-q_j)|^2.
\end{align*}
With regards to $I_3$ and $J_3$, we recall Lemma~\ref{L:LipD_inv} once again. On the one hand, when $\lambda>1$, we have 
\begin{align*}
    I_3=-\sum_{i=1}^N \inner{\div_{q_i} D^{-1}(q_i)}{\nabla U(q_i)}\le& {L_{D^{-1}}\sum_{i=1}^N (|q_i|+C)|\nabla U(q_i)|}=o\Big(\sum_{i=1}^N|\nabla U(q_i)|^2\Big).
\end{align*}
where we applied~\eqref{EC:BdNablaU} in the last implication.
On the other hand, when $\lambda=1$, $\nabla U(q)$ dominates $|q|$, cf. \eqref{EC:BdNablaU}, whereas $\div D^{-1}$ is a sublinear function, cf Lemma \ref{L:LipD_inv} (ii). It follows that
\begin{align*}
    I_3\le C\sum_{i=1}^N|\div_{q_i}D^{-1}(q_i)||\nabla U(q_i)|=o\Big(\sum_{i=1}^N|\nabla U(q_i)|^2\Big).
\end{align*}
Likewise,
\begin{align*}
    J_3=&-\sum_{i=1}^N\Big\la \varepsilon_1^{3/4}\sum_{j\neq i}\nabla G(q_i-q_j) , \varepsilon_1^{1/4}\div_{q_i}D^{-1}(q_i) \Big\ra \\
    &\le C\varepsilon_1\Big( \varepsilon_1^{1/2}\sum_{i}\Big|\sum_{j\ne i}\nabla G(q_i-q_j)\Big|^2+ o\Big(\sum_{i=1}^N|\nabla U(q_i)^2|\Big) \Big)+C.
\end{align*}
Concerning the trace terms $I_5$ and $J_5$, note that the eigenvalues of $D^{-1}$ is bounded by $\underline{\gamma}^{-1}$. As a consequence, we infer that
\begin{align*}
    I_5=\sum_{i=1}^N\tr\lr{\nabla^2 U(q_i) D^{-1}\lr{q_i}}\le da_1\underline{\gamma}^{-1}\Big( N +\sum_{i=1}^N |q_i|^{\lambda-1} \Big),
\end{align*}
and that
\begin{align*}
    J_5=\varepsilon_1\sum_{i\neq j}\tr \lr{\nabla^2 G(q_i-q_j)\lr{D^{-1}(q_i)}}\le& a_1\varepsilon_1 d\underline{\gamma}^{-1}\Big( N^2+\sum_{i\neq j}\frac{1}{|q_i-q_j|^{\beta_1+1}} \Big).
\end{align*}
Recall that $\Gamma_1(t)$ is defined by~\eqref{EC:DefGamma1}. Combining the estimates of $I_k$, $J_k$, $k=1,2,3,5$, from~\eqref{EC:dU} and~\eqref{EC:dG}, we get
\begin{equation}\label{EC:dU+dG}
    \begin{aligned}
         \ud \Gamma_1\le&  -\overline{\gamma}^{-1}\sum_{i=1}^N |\nabla U(q_i)|^2\ud t+\underline{\gamma}\varepsilon_1\sum^N_{i=1}\Big|\nabla U(q_i)\Big|^2\ud t-\varepsilon_1\overline{\gamma}^{-1}\sum_{i=1}^N \Big|\sum_{j\neq i}\nabla G(q_i-q_j)\Big|^2\ud t\\
        &+C\varepsilon_1^{3/2}\sum_{i}\Big|\sum_{j\ne i}\nabla G(q_i-q_j)\Big|^2\ud t+o\Big(\sum_{i=1}^N|\nabla U(q_i)|^2\Big)  \ud t+da_1\underline{\gamma}^{-1}\sum |q_i|^{\lambda-1}\ud t \\
        &+\varepsilon_1 a_1 d\underline{\gamma}^{-1}\sum_{i\ne j}\frac{1}{|q_i-q_j|^{\beta_1+1}}\ud t+C \ud t+\ud M_1(t),
    \end{aligned}
\end{equation}
and
\begin{equation}\label{EC:Martingale}
    \ud M_1(t)=\sum_{i=1}^N\Big\la \nabla U(q_i)+\varepsilon_1\sum_{j\neq i}\nabla G(q_i-q_j),\sqrt{2D^{-1}(q_i)}\ud W_i\Big\ra.
\end{equation}
Letting $\kappa>0$ be given and be chosen later, in accordance with~\eqref{EC:BdNablaU} and~\eqref{EC:BoundNablaG} together with~\eqref{E_:oI_2} and~\eqref{E_:oJ_1}, by taking $\varepsilon_1$ sufficiently small, we can simplify~\eqref{EC:dU+dG} as
    \begin{align*}
        \kappa\ud \Gamma_1(t)\le&  -c\kappa\sum_{i=1}^N |\nabla U(q_i)|^2\ud t-c\kappa\sum_{i=1}^N |\sum_{j\neq i}\nabla G(q_i-q_j)|^2\ud t\\
        &+C\kappa \ud t+\kappa\ud M_1(t),
    \end{align*}
Furthermore, the variation process associated with the semi-Martingale term given by~\eqref{EC:Martingale} can be estimated as follows.
    \begin{align*}
        \ud \langle\kappa M_1(t)\rangle=&2\kappa^2\sum_{i}\Big\la \Big(\nabla U(q_i)+\varepsilon_1\sum_{j\ne i}\nabla G(q_i-q_j)\Big)  ,D^{-1}(q_i)\Big( \nabla U(q_i)+\varepsilon_1\sum_{j\ne i}\nabla G(q_i-q_j) \Big) \Big\ra \ud t\\
    \le& 2\kappa^2\underline{\gamma}^{-1}\sum_{i=1}^N\Big|\nabla U(q_i)+\sum_{j\ne i}\nabla G(q_i-q_j)\Big|^2\ud t\\
    \le& \kappa^2 c\Big( \sum_{i=1}^N|\nabla U(q_i)|^2+\sum_{j\ne i}|\nabla G(q_i-q_j)|^2 \Big) \ud t.
    \end{align*}
It follows that
\begin{align*}
    \kappa\ud \Gamma_1(t)\le C\kappa\ud t-\frac{c}{\kappa}\ud\langle \kappa M_1(t)\rangle +\ud  (\kappa M_1(t)).
\end{align*}
Integrating both sides to get
\begin{equation}\label{EC:Gamma_1}
    \kappa M_1(t)-\frac{c}{\kappa}\langle \kappa     M_1(t)\rangle\ge \kappa\Gamma_1(t)-\kappa\Gamma_1(0)-C\kappa  t.
\end{equation}
We apply exponential Martingale inequality to find
\begin{equation*}
    \mathbb{P}\Big( \sup_{t \geq 0} \Big[ \kappa M_1(t) - \frac{c}{\kappa} \langle \kappa M_1 \rangle(t) \Big] > r \Big) \leq e^{-\frac{2c}{\kappa} r}, \quad r \geq 0.
\end{equation*}
In view of~\eqref{EC:Gamma_1}, we have
\begin{equation*}
    \mathbb{P}\Big( \sup_{t \in [0, T]} \Big[ \kappa \Gamma_1(t) - \kappa \Gamma_1(0) - \kappa Ct \Big] > r \Big) \leq e^{-\frac{2c}{\kappa} r}.
\end{equation*}
By tail integration and choosing $\kappa$ small enough
\begin{align*}
    \E \exp\Big\{\sup_{t\in[0,T]}\kappa \Gamma_1(t)\Big\}\le e^{\kappa\Gamma_1(0)+\kappa CT  }+\int_0^\infty e^{r+\kappa\Gamma_1(0)+\kappa C T}e^{-\frac{2c}{\kappa}r}\ud r\le C.
\end{align*}
This establishes \eqref{ineq:EC:Limiting:beta_1>1}, thus completing part (i).

(ii) Turning to the case $\beta_1=1$, we introduce the function
\begin{align*}
    \Gamma_2=\sum_{i=1}^N U(q_i)-\varepsilon_1 \bar{\gamma}\sum_{1\le i<j\le N}\log |q_i-q_j|.
\end{align*}
We apply It\^{o}'s formula to the log term on the above right hand side and obtain
\begin{align*}
    &\ud \Big(-\varepsilon_1 \overline{\gamma}\sum_{1\le i<j\le N}\log |q_i-q_j|\Big)\\
    &=-\varepsilon_1\overline{\gamma}\sum_{i}\Big\langle\sum_{j\neq i}\frac{q_i-q_j}{|q_i-q_j|^{2}}, \big[-D^{-1}(q_i)\big(\nabla U(q_i)+\sum_{l\neq i}\nabla G(q_l-q_i)\big)\\
    &\hspace{3cm}-\div D^{-1}(q_i)\big]\ud t+\sqrt{2}\sqrt{D(q_i)}^{-1}\ud W_i\Big\rangle\\
    &\qquad +\varepsilon_1 \overline{\gamma}\sum_{i=1}^N\tr\bigg(D^{-1}(q_i)\sum_{j\neq i}\frac{|q_i-q_j|^2I-2(q_i-q_j)\otimes (q_i-q_j) }{|q_i-q_j|^4}\bigg)\ud t\\
    &\eqqcolon (L_1+L_2+L_3)\ud t -\varepsilon_1\overline{\gamma}\sum_{i}\Big\langle\sum_{j\neq i}\frac{q_i-q_j}{|q_i-q_j|^{2}},\sqrt{2}\sqrt{D(q_i)}^{-1}\ud W_i\Big\rangle+L_4 \ud t.
\end{align*}
Concerning $L_1$, we recast this term as follows.
\begin{align*}
    L_1=&\varepsilon_1\overline{\gamma}\sum_{1\le i<j\le N }\inner{\frac{q_i-q_j}{|q_i-q_j|^2}}{D^{-1}(q_i)\nabla U(q_i)-D^{-1}(q_j)\nabla U(q_j)}\\
    =&\varepsilon_1\overline{\gamma}\sum_{1\le i<j\le N }\inner{\frac{q_i-q_j}{|q_i-q_j|^2}}{D^{-1}(q_i)\nabla U(q_i)-D^{-1}(q_i)\nabla U(q_j)}\\
    &\qquad+\varepsilon_1\overline{\gamma}\sum_{1\le i<j\le N }\inner{\frac{q_i-q_j}{|q_i-q_j|^2}}{D^{-1}(q_i)\nabla U(q_j)-D^{-1}(q_j)\nabla U(q_j)}\\
    =&L_{1,1}+L_{1,2}.
\end{align*}
Applying the mean value theorem to $\nabla U$ yields
\begin{align*}
    L_{1,1}=&\varepsilon_1\overline{\gamma}\sum_{1\le i<j\le N }\inner{\frac{q_i-q_j}{|q_i-q_j|^2}}{D^{-1}(q_i)\nabla U(q_i)-D^{-1}(q_i)\nabla U(q_j)}\\
    \le&\varepsilon_1\overline{\gamma}\sum_{1\le i<j\le N }\frac{\underline{\gamma}^{-1}}{|q_i-q_j|}\nabla^2 U(\xi_{i,j})|q_i-q_j|\\
    \le &\varepsilon_1\overline{\gamma}\underline{\gamma}^{-1}a_1\lr{1+|q_i|^{\lambda-1}+|q_j|^{\lambda-1}}.
\end{align*}
Recalling from~\eqref{E:LipD_inv} that $D^{-1}$ is Lipschitz, this implies that
\begin{align*}
    L_{1,2}=&\varepsilon_1\overline{\gamma}\sum_{1\le i<j\le N }\inner{\frac{q_i-q_j}{|q_i-q_j|^2}}{D^{-1}(q_i)\nabla U(q_j)-D^{-1}(q_j)\nabla U(q_j)}\\
    \le & \varepsilon_1\overline{\gamma}\sum_{1\le i<j\le N }\frac{1}{|q_i-q_j|}L_{D^{-1}}|q_i-q_j|a_1\lrbig{1+|q_j|^{\lambda}}\\
    \le &\varepsilon_1\overline{\gamma}a_1L_{D^{-1}}N\sum_{i=1}^N |q_i|^{\lambda}.
\end{align*}
Similarly, concerning $L_2$, we have
\begin{align*}
    L_2=\varepsilon_1
    \overline{\gamma}\sum_{i}\inner{\sum_{j\neq k}\frac{q_i-q_j}{|q_i-q_j|^2}}{D^{-1}(q_i)\sum_{k\neq i}\nabla G(q_i-q_k)}\eqqcolon L_{2,1}+L_{2,2},
\end{align*}
where
\begin{align*}
    L_{2,1}=-&\varepsilon_1\overline{\gamma}\sum_{i}\inner{D^{-1}(q_i)\sum_{j\neq i}\frac{q_i-q_j}{|q_i-q_j|^{2}}}{a_4\sum_{l\neq i}\frac{q_i-q_l}{|q_i-q_l|^{2}}},
\end{align*}
and
\begin{align*}
    L_{2,2}&=\varepsilon_1\overline{\gamma}\sum_{i}\inner{D^{-1}(q_i)\sum_{j\neq i}\frac{q_i-q_j}{|q_i-q_j|^{2}}}{\sum_{l\neq i}\nabla G(q_i-q_j)+a_4\sum_{l\neq i}\frac{q_i-q_l}{|q_i-q_l|^{2}}}.
\end{align*}
By writing $w_i=\sum_{j\ne i}\frac{q_i-q_j}{|q_i-q_j|^2}=\sum_{l\ne i}\frac{q_i-q_l}{|q_i-q_l|^2}$, and by the uniform positivity of $D^{-1}$, it holds that
\begin{align*}
    L_{2,1}\le -\varepsilon_1\overline{\gamma}a_4\sum_{i=1}^N \innerbig{D^{-1}(q_i)w_i}{w_i}
    \le& -\varepsilon_1 \overline{\gamma} \overline{\gamma}^{-1}a_4\sum_{i=1}^N |w_i|^2 \\
    = &-\varepsilon_1 a_4 \sum_{i=1}^N \Big|\sum_{j\ne i}\frac{q_i-q_j}{|q_i-q_j|^2}\Big|^2\\
    \le &-2\varepsilon_1 a_4\sum_{1\le i<j\le N}\frac{1}{|q_i-q_j|^2},
\end{align*}
where the last line follows Lemma~\ref{L:s+1timess+1} by taking $s=1$.
To estimate $L_{2,2}$, applying Cauchy-Schwarz inequality, the uniform boundedness of $D^{-1}$ and Assumption~\ref{A:G}, we get
\begin{align*}
        L_{2,2}&=\varepsilon_1\overline{\gamma}\sum_{i}\inner{D^{-1}(q_i)\sum_{j\neq i}\frac{q_i-q_j}{|q_i-q_j|^{2}}}{\sum_{l\neq i}\nabla G(q_i-q_j)+a_4\sum_{l\neq i}\frac{q_i-q_l}{|q_i-q_l|^{2}}}\\
    &\le \varepsilon_1\overline{\gamma}\sum_{i}\Big|{D^{-1}(q_i)\sum_{j\neq i}\frac{q_i-q_j}{|q_i-q_j|^{2}}}\Big|\,\Big|{\sum_{l\neq i}\nabla G(q_i-q_j)+a_4\sum_{l\neq i}\frac{q_i-q_l}{|q_i-q_l|^{2}}}\Big|\\
    &\le \varepsilon_1\overline{\gamma}\underline{\gamma}^{-1}\sum_{i}\Big|{\sum_{j\neq i}\frac{1}{|q_i-q_j|}}\Big|\,\Big|{a_5\sum_{l\neq i}\frac{1}{|q_i-q_l|^{\beta_2}}+a_6}\Big|.
    \end{align*}
We invoke Young's inequality to further deduce
\begin{align*}
    L_{2,2}
    &\le C\overline{\gamma}\underline{\gamma}^{-1}\Big(\varepsilon_1^{3/2}\sum_{1\le i<j\le N}\frac{1}{|q_i-q_j|^{2}}+\varepsilon_1^{1/2}a_5^2\sum_{1\le i<j\le N}\frac{1}{|q_i-q_j|^{2\beta_2}}+\varepsilon_1 a_6\sum_{1\le i<j\le N}\frac{1}{|q_i-q_j|}\Big).
\end{align*}
It follows that
\begin{align*}
    &L_2\le -2\varepsilon_1 a_4\sum_{1\le i<j\le N}\frac{1}{|q_i-q_j|^2} \\
    & +C\overline{\gamma}\underline{\gamma}^{-1}\Big(\varepsilon_1^{3/2}\sum_{1\le i<j\le N}\frac{1}{|q_i-q_j|^{2}}+\varepsilon_1^{1/2}a_5^2\sum_{1\le i<j\le N}\frac{1}{|q_i-q_j|^{2\beta_2}}+\varepsilon_1 a_6\sum_{1\le i<j\le N}\frac{1}{|q_i-q_j|}\Big).
\end{align*}
Turning to $L_3$, since $\div D^{-1}$ is Lipschitz, cf. Lemma~\ref{L:LipD_inv}, we obtain
\begin{align*}
    L_3=\varepsilon_1\overline{\gamma}\sum_{1\le i<j\le N}\inner{\frac{q_i-q_j}{|q_i-q_j|^2}}{\div D^{-1}(q_i)-\div D^{-1}(q_j)}\le C_N\varepsilon_1\overline{\gamma}L_{D^{}-1}.
\end{align*}
Also, thanks to the fact that the eigenvalues of $D^{-1}$ are uniformly bounded, we can estimate $L_4$ by
\begin{align*}
    L_4=&-\varepsilon_1\overline{\gamma}\sum_{i=1}^N\sum_{j\neq i}\frac{\tr D^{-1}(q_i)}{|q_i-q_j|^2}-2\frac{ (q_i-q_j)^t D^{-1}(q_i)(q_i-q_j) }{|q_i-q_j|^4}\\
    \le&-2\varepsilon_1\overline{\gamma}\sum_{i< j}\frac{d\overline{\gamma}^{-1}-2\underline{\gamma}^{-1}}{|q_i-q_j|^2}.
\end{align*}

Now, we collect the bounds on $L_1,\dots, L_4$, together with the estimates of $I_1,\dots, I_5$ in the argument of part (i), to infer
\begin{align*}
    &\ud \Gamma_2(t)\\
    &\le  -\overline{\gamma}^{-1}\sum_{i=1}^N |\nabla U(q_i)|^2\ud t+I_2\ud t+ \sum_{i=1}^N (|q_i|+a_0)|\nabla U(q_i)|\ud t+da_1\underline{\gamma}^{-1}\lr{N +\sum_{i=1}^N |q_i|^{\lambda-1}}\ud t\\
    &+\varepsilon_1\overline{\gamma}\underline{\gamma}^{-1}a_1\lr{1+|q_i|^{\lambda-1}+|q_j|^{\lambda-1}}\ud t+\varepsilon_1\overline{\gamma}a_1L_{D^{-1}}N\sum_{i=1}^N |q_i|^{\lambda}\ud t-2\varepsilon_1 a_4 \sum_{1\le i<j\le N}\frac{1}{|q_i-q_j|^2}\ud t\\
    &+C\overline{\gamma}\underline{\gamma}^{-1}\lr{\varepsilon_1^{3/2}\sum_{1\le i<j\le N}\frac{1}{|q_i-q_j|^{2}}+\varepsilon_1^{1/2}a_5^2\sum_{1\le i<j\le N}\frac{1}{|q_i-q_j|^{2\beta_2}}+\varepsilon_1 a_6\sum_{1\le i<j\le N}\frac{1}{|q_i-q_j|}}\ud t\\
    &-2\varepsilon_1\overline{\gamma}\sum_{i< j}\frac{d\overline{\gamma}^{-1}-2\underline{\gamma}^{-1}}{|q_i-q_j|^2}\ud t+C\ud t + M_2(t),
\end{align*}
where 
\begin{align*}
    M_2(t)=\sum_{i=1}^N\inner{\nabla U (q_i)-\varepsilon_1\overline{\gamma}\sum_{j\ne i}\frac{q_i-q_j}{|q_i-q_j|^2}}{\sqrt{2}\sqrt{D(q_i)}^{-1}\ud W_i}.
\end{align*}
Recalling the growth estimates of $\nabla U$ and $\nabla G$ respectively in \eqref{EC:BdNablaU} and \eqref{EC:BoundNablaG}, by subsuming the lower order terms, the above inequality can be simplified as
\begin{align*}
    \ud \Gamma_2(t)\le -c\sum_{i=1}^N|\nabla U(q_i)|^2-\varepsilon_1(a_4+d-2\overline{\gamma}\underline{\gamma}^{-1})\sum_{1\le i<j\le N}\frac{1}{|q_i-q_j|^2}\ud t+C\ud t+M_2(t).
\end{align*}
It is important to note that $    a_4+d-2\overline{\gamma}\underline{\gamma}^{-1}>0$, 
as stated in Assumption~\ref{Assumption D} (D) (iii). Similar to the argument for \eqref{ineq:EC:Limiting:beta_1>1}, we employ once again the exponential martingale inequality to deduce that for $\kappa>0$ small enough, there exists a positive constant $C$ such that
\begin{align*}
    \E\Big[\exp\Big\{\sup_{t\in[0,T]}\kappa\Gamma_2(t)\Big\}\Big]\le C.
\end{align*}
This establishes \eqref{ineq:EC:Limiting:beta_1=1}, as claimed.

\end{proof}

Having collected auxiliary estimates in Lemmas \ref{lem:grad.U_grad.G}, \ref{L:LipD_inv} and \ref{L_:Step2Goal}, we are now ready to present the proof of Theorem \ref{thm: smallLE}, establishing validity of the small-mass limit of the classical model \eqref{E:Classical_Eq} toward \eqref{EC:Limiting}.
\begin{proof}[Proof of Theorem \ref{thm: smallLE}]
    The argument consists of two main steps as follows.
    
    \noindent \textbf{Step 1}: For $R>0$, let $\theta_R$ be a smooth cut-off function defined as
\begin{equation}\label{E:truncation function}
    \theta_R(t)=\begin{cases}
        1, & |t|\le R,\\
        \text{monotonicity},&R<|t|\le R+1,\\
        0,&|t|> R+1.
    \end{cases}
\end{equation}
    Consider the following truncated system
    \begin{align}
        \ud x_i^R&=v_i^R\ud t, \notag \\
    m\ud v_i^R&=-D(x_i^R)v_i^R \ud t-\theta_R(|x_i^R|)\nabla U(x_i^R)\ud t-\sum_{j\neq i}\theta_R(|x_i^R-x_j^R|^{-1})\nabla G(|x_i^R-x_j^R|)\ud t \label{EC:Truncated}\\
    &\qquad\qquad+\sqrt{2D(x_i^R)}\ud W_i, \notag 
    \end{align}
and truncated limiting process
\begin{equation}\label{EC:TrLimiting}
    \begin{aligned}
        \ud q_i^R&=[-D^{-1}(q_i^R)(\theta_R(|q_i^R|)\nabla U(q_i^R)+\sum_{j\neq i}\theta_R(|q_i^R-q_j^R|^{-1})\nabla G(q_j^R-q_i^R))
        \\&\qquad-\div D^{-1}(q_i)]\ud t+\sqrt{2}\sqrt{D}^{-1}(q_i)\ud W_i.
    \end{aligned}
\end{equation}
We aim to employ \cite[Theorem 1]{hottovy2015smoluchowski} to show that
\begin{align}\label{L2 convergence}
    \lim_{m\to 0}\mathbb{E}\Big[\sup_{t\in [0,T]}|\x_m^R(t)-\q^R(t)|)^2\Big]=0.
\end{align}
To do this, we need to verify that Assumptions 1-3 in \cite[Theorem 1]{hottovy2015smoluchowski} hold.

\textup{(i)} We have assumed that $\nabla U$, $\nabla G$, $D$ are continuously differentiable. Also, the uniform positivity of the eigenvalues of $D$ is implied by the uniform ellipticity. This verifies Assumption~1 of \cite[Theorem 1]{hottovy2015smoluchowski}.

\textup{(ii)} Thanks to the truncation approach, we have avoided the singularities of the potentials $U$ and $G$. For any $m>0$, all coefficients in the system~\eqref{EC:Truncated} are globally Lipschitz, which ensures the existence of global unique solutions on any finite time window $t\in [0,T]$. In particular, this shows that Assumption~2 of \cite[Theorem 1]{hottovy2015smoluchowski} holds.

\textup{(iii)} Turning to Assumption~3 of \cite[Theorem 1]{hottovy2015smoluchowski}, we note that the existence of a compact set $\mathcal{K}$ such that $\x (t)\in \mathcal{K}\subsetneq \mathcal{D}$ may not hold. Nevertheless, it is sufficient to check that \cite[Equation (8)]{hottovy2015smoluchowski} is satisfied. Indeed, we have that
    \begin{align*}
        -\theta_R(|x_i^R(t)|)\nabla U(x_i^R(t)) \leq C_R,\quad-\theta_R(|x_i^R(t)-x_j^R(t)|^{-1})\nabla G(x_i^R(t)-x_j^R(t)) \leq C_R,
        \end{align*}
        and that
        \begin{align*}
            |D(x_i^R(t))|\leq \overline{\gamma},\qquad|\sqrt{2D(x(t))}|\leq \sqrt{2\overline{\gamma}}.
    \end{align*}
    Moreover, the constants in the above inequalities are independent of the parameter $m$.

Altogether, we may employ \cite[Theorem 1]{hottovy2015smoluchowski} to deduce that limit \eqref{L2 convergence} holds, thereby completing Step 1.

\noindent\textbf{Step 2.~} Define stopping times
\begin{align}\label{stopping_time}
    \sigma^R=\inf_{t\ge 0}\{|\q(t)|+\sum_{1\le i<j\le N}|q_i(t)-q_j(t)|^{-1}\ge R\},
\end{align}
and
\begin{align}\label{stopping_time_m}
    \sigma_m^R=\inf_{t\ge 0}\{|\x_m(t)|+\sum_{1\le i<j\le N}|x_{i,m}(t)-x_{j,m}(t)|^{-1}\         
                             \ge R\}.
\end{align}
For any $T,\xi>0$, we have
\begin{align}\label{EC:P1+P2}
    \notag \mathbb{P}\lr{\sup_{t\in [0,T]}|\x_m(t)-\q(t)|>\xi}&\le \mathbb{P}\lr{\sup_{t\in[0,T]}|\x_m(t)-\q(t)|>\xi,\sigma^R\wedge\sigma^R_m\ge T}+\mathbb{P}\lr{\sigma^R\wedge\sigma^R_m<T}\\
    &\eqqcolon P_1+P_2.
\end{align}
By definition of the truncation $\theta_R$, it holds that
\begin{align*}
    \mathbb{P}\lr{0\le t\le \sigma^R\wedge\sigma^R_m, \q(t)=\q^R(t),\x_m(t)=\x_m^R(t)}=1.
\end{align*} 
In the above, $\x_m^R$ and $\q^R$ are respectively the solutions of \eqref{EC:Truncated} and \eqref{EC:TrLimiting}. Concerning $P_1$, we employ Chebyshev inequality to infer
\begin{align}\label{Chebyshev}
    P_1\le \mathbb{P}\lr{\sup_{t\in[0,T]}|\x_m^R(t)-\q^R(t)|>\xi}\le \frac{\mathbb{E}\lr{\sup_{t\in[0,T]}|\x_m^R(t)-\q^R(t)|^2}}{\xi^2}.
\end{align}
With regard to $P_2$, we note that
\begin{align}\label{EC:P2<=P21+P22+P23}
    \notag P_2\le & \mathbb{P}\lr{\sup_{t\in[0,T]}|\x_m^R(t)-\q^R(t)|
    \le \frac{\xi}{R},\sigma^R\wedge\sigma_m^R<T}+\mathbb{P}\lr{\sum_{t\in [0,T]}|\x_m^R-q^R|> \frac{\xi}{R}}\\
    \notag \le &\mathbb{P}\lr{\sup_{t\in[0,T]}|\x_m^R(t)-\q^R(t)|
    \le \frac{\xi}{R},\sigma_m^R<T\le \sigma_m^R}+\mathbb{P}\lrbig{\sigma^R<T}\\
    &+\mathbb{P}\lr{\sum_{t\in [0,T]}|\x_m^R(t)-\q^R(t)|>\frac{\xi}{R}}=P_{2,1}+P_{2,2}+P_{2,3}.
\end{align}
An argument similar to~\eqref{Chebyshev} produces
\begin{align*}
    P_{2,3}\le \frac{\mathbb{E}\lr{\sup_{t\in[0,T]}|\x_m^R(t)-q^R(t)|^2}}{\xi^2}.
\end{align*}
Considering $P_{2,2}$, from expression~\eqref{stopping_time} of $\sigma^R$, observe that
\begin{align*}
    \{\sigma^R<T\}=\Big\{\sup_{t\in[0,T]}|\q(t)|+\sum_{1\le i<j\le N}|q_i(t)-q_j(t)|^{-1}\ge R\Big\},
\end{align*}
whence,
\begin{align}\label{sigma_R}
    \{\sigma^R<T\}\subset \Big\{\sup_{t\in[0,T]}|\q(t)|\ge \frac{R}{N^2}\Big\}\bigcup_{1\le i<j\le N}\Big\{\sup_{t\in[0,T]}-\varepsilon_1\log|q_i(t)-q_j(t)|\ge \varepsilon_1\log\lr{\frac{R}{N^2}}\Big\}.
\end{align}
Concerning the event $\big\{\sup_{t\in[0,T]}|\q(t)|\ge \frac{R}{N^2}\big\}$, we first notice that 
\begin{align*}
    \log|q_i(t)-q_j(t)|\le \log2+\log_+\lrbig{{|q_i|\vee|q_j|}}\le \log 2+\lr{\log_+|q_i|+\log_+|q_j|},
\end{align*}
where $\log_+(x)\coloneqq\log (1\vee x)$. Consequently, we have
\begin{align}\label{subsum_log_1}
    &|\q(t)|^2-\varepsilon_1 \sum_{1\le i<j\le N}\log|q_i(t)-q_j(t)|\notag\\
    \ge &\sum_{i}|q_i(t)|^2-\varepsilon_1\sum_{1\le i<j\le N}\lrbig{\log 2+\lrbig{\log_+|q_i|+\log_+|q_j|}}\notag\\
    =&\sum_{i=1}^N |q_i|^2-\varepsilon_1 (N-1)\sum_{i=1}^N\log_+|q_i|-\varepsilon_1(\log 2) \lr{\frac{N(N-1)}{2}}\notag \\
    \ge&
    \frac{1}{2}\sum_{i=1}^N|q_i|^2-\varepsilon_1 C.
\end{align}
Letting $R$ be large enough such that $\frac{R}{N^2}\ge 4$, suppose that $\sup_{t\in[0,T]}|\q(t)|\ge \frac{R}{N^2}$. This implies that
\begin{align}\label{EC:CompareQwithQ2-log}
    \sup_{t\in[0,T]}|\q(t)|^2-\varepsilon_1 \sum_{1\le i<j\le N}\log|q_i(t)-q_j(t)|\ge \frac{1}{2}\frac{R^2}{N^4}-\varepsilon_1 C\ge \frac{R}{N^2}.
\end{align}
In other words, for $\varepsilon_1$ sufficiently small and $R$ sufficiently large,
\begin{align*}
    \Big\{\sup_{t\in[0,T]}|\q(t)|\ge \frac{R}{N^2}\Big\}\subset\Big\{\sup_{t\in[0,T]}|\q(t)|^2-\varepsilon_1 \sum_{1\le i<j\le N}\log|q_i(t)-q_j(t)|\ge \frac{R}{N^2}\Big\}.
\end{align*}
It is clear that for $R$ sufficiently large,
\begin{align*}
    &\Big\{\sup_{t\in[0,T]}-\varepsilon_1 \log|q_i(t)-q_j(t)|\ge \varepsilon_1\log\lr{\frac{R}{N^2}}\Big\}\\
    &\qquad\subset\Big\{\sup_{t\in[0,T]}|\q(t)|^2-\varepsilon_1 \log|q_i(t)-q_j(t)|\ge \varepsilon_1\log\lr{\frac{R}{N^2}}\Big\}.
\end{align*}
From~\eqref{sigma_R}, we deduce
\begin{align}\label{sigma_R_2}
    \{\sigma^R< T\}\subset& \Big\{\sup_{t\in[0,T]}|\q(t)|^2-\varepsilon_1 \sum_{1\le i<j\le N}\log|q_i(t)-q_j(t)|
    \ge \frac{R}{N^2}\Big\}\notag\\
    &\bigcup_{1\le i<j\le N}\Big\{\sup_{t\in[0,T]}|\q(t)|^2-\varepsilon_1 \log|q_i(t)-q_j(t)|\ge \varepsilon_1\log\lr{\frac{R}{N^2}}\Big\}.
\end{align}
Again for $R$ sufficiently large, it holds that
\begin{align*}
    &\Big\{\sup_{t\in[0,T]}|\q(t)|^2-\varepsilon_1 \sum_{1\le i<j\le N}\log|q_i(t)-q_j(t)|
    \ge \frac{R}{N^2}\Big\}\notag \\
    &\subset \Big\{\sup_{t\in[0,T]}|\q(t)|^2-\varepsilon_1 \sum_{1\le i<j\le N}\log|q_i(t)-q_j(t)|
    \ge \frac{\varepsilon_1}{8}\log\lr{\frac{R}{N^2}}\Big\}.
\end{align*}
Also, for each $i_0\neq j_0$,
\begin{align*}
        &\Big\{\sup_{t\in[0,T]}|\q(t)|^2-\varepsilon_1 \log|q_{i_0}(t)-q_{j_0}(t)|\ge \varepsilon_1\log\lr{\frac{R}{N^2}}\Big\}\\
    &\subset \Big\{\sup_{t\in[0,T]}|\q(t)|^2\ge \frac{\varepsilon_1}{2}\log\lr{\frac{R}{N^2}}\Big\}\bigcup\Big\{\sup_{t\in[0,T]}-\varepsilon_1 \log|q_{i_0}(t)-q_{j_0}(t)|\ge \frac{\varepsilon_1}{2}\log\lr{\frac{R}{N^2}}\Big\}\\
    &\eqqcolon E_1\bigcup E_2.
    \end{align*}
In light of~\eqref{subsum_log_1}, for $R$ sufficiently large, we have
\begin{align}\label{sigma_R_4}
    E_1\subset \Big\{\sup_{t\in[0,T]}|\q(t)|^2-\varepsilon_1 \sum_{1\le i<j\le N}\log|q_i(t)-q_j(t)|\ge \frac{\varepsilon_1}{8}\log\lr{\frac{R}{N^2}}\Big\}.
\end{align}
Moreover, it holds that for any $i_0\ne j_0$ shown in $E_2$,
\begin{align*}
    \sup_{t\in[0,T]}&|\q(t)|^2-\varepsilon_1 \sum_{1\le i<j\le N}\log|q_i(t)-q_j(t)|\\
    \ge &\sup_{t\in[0,T]} -\varepsilon_1\log|q_{i_0}-q_{j_0}|\\
    &+\Big[\varepsilon_1\sum_{i=1}^N|q_i|^2-\varepsilon_1\sum_{\substack{1\le i<j\le N\\(i,j)\neq (i_0,j_0)}}\lr{\log 2+\lr{\log_+|q_i|+\log_+|q_j|}}\Big]\\
    \ge&\sup_{t\in[0,T]} -\varepsilon_1\log|q_{i_0}-q_{j_0}|-\varepsilon_1 C.
\end{align*}
As a result,
\begin{align*}
    E_2\subset \Big\{\sup_{t\in[0,T]}|\q(t)|^2-\varepsilon_1 \sum_{1\le i<j\le N}\log|q_i(t)-q_j(t)|\ge \frac{\varepsilon_1}{8}\log\lr{\frac{R}{N^2}}\Big\}.
\end{align*}
Now, we collect \eqref{sigma_R}--~\eqref{sigma_R_4} to obtain the inclusion
\begin{align*}
    \{\sigma^R<T\}\subset \Big\{\sup_{t\in[0,T]}|\q(t)|^2-\varepsilon_1 \sum_{1\le i<j\le N}\log|q_i(t)-q_j(t)|\ge \frac{\varepsilon_1}{8}\log\lr{\frac{R}{N^2}}\Big\}.
\end{align*}
In view of Lemma~\ref{L_:Step2Goal}, we obtain
\begin{align*}
    \E\Big\{\sup_{t\in[0,T]}|\q(t)|^2-\varepsilon_1 \sum_{1\le i<j\le N}\log|q_i(t)-q_j(t)|\Big\}\le C.
\end{align*}
An application of Markov inequality shows that
\begin{align*}
   P_{2,2}= \mathbb{P}\{\sigma^R<T\}<\frac{C}{\varepsilon_1\log R}.
\end{align*}
Finally we turn to $P_{2,1}$. Recall the definition of stopping time~\eqref{stopping_time_m}, we have the following implication.
\begin{align*}
&\Big\{ \sup_{t \in [0, T]} \left| \x_{m}^{R}(t) - \q^{R}(t) \right| \leq \frac{\xi}{R}, \sigma_{m}^{R} < T \leq \sigma^{R} \Big\} \\
&= \Big\{ \sup_{t \in [0, T]} \left| \x_{m}^{R}(t) - \q(t) \right| \leq \frac{\xi}{R},  \sigma_{m}^{R} < T \leq \sigma^{R} \Big\} \\
&\qquad\bigcap \Big\{\sup_{t\in[0,T]} \Big( \left| \x_{m}^{R}(t) \right| + \sum_{1 \leq i < j \leq N} \left| x_{j,m}^{R}(t) - x_{j,m}^{R}(t) \right|^{-1} \Big) \geq R\Big\}\\
&\subseteq \Big\{ \sup_{t \in [0, T]} \left| \x_{m}^{R}(t) - \q(t) \right| \leq \frac{\xi}{R}, \sup_{t\in[0,T]} \left| \x_{m}^{R}(t) \right|  \geq \frac{R}{N^2} \Big\} \\
& \bigcup_{1 \leq i < j \leq N} \Big\{ \sup_{t \in [0, T]} \left| \x_{m}^{R}(t) - \q(t) \right| \leq \frac{\xi}{R}, \sup_{t \in [0, T]} \left| x_{i,m}^{R}(t) - x_{j,m}^{R}(t) \right|^{-1} \ge \frac{R}{N^2} \Big\}. 
\end{align*}
In other words, it holds that
\begin{align}\label{EC:F0+Fij}
&\Big\{ \sup_{t \in [0, T]} \left| \x_{m}^{R}(t) - \q^{R}(t) \right| \leq \frac{\xi}{R}, \sigma_{m}^{R} < T \leq \sigma^{R} \Big\} \notag \\
&= \Big\{ \sup_{t \in [0, T]} \left| \x_{m}^{R}(t) - \q(t) \right| \leq \frac{\xi}{R}, \sup_{t \in [0, T]} \left| \x_{m}^{R}(t) \right| \geq \frac{R}{N^2} \Big\} \notag \\
&\qquad \bigcup_{1 \leq i < j \leq N} \Big\{ \sup_{t \in [0, T]} \left| \x_{m}^{R}(t) - \q(t) \right| \leq \frac{\xi}{R}, \inf_{t \in [0, T]} \left| x_{i,m}^{R}(t) - x_{j,m}^{R}(t) \right| \le \frac{N^2}{R} \Big\} \notag \\
&\eqqcolon F_0\bigcup_{1 \leq i < j \leq N} F_{i,j}.
\end{align}
For any $\xi$ and $N$ fixed, we take $R$ large enough such that $\frac{R}{N^2}-\frac{\xi}{R}\ge \sqrt{R}$. We have
\begin{align}\label{EC:F0}
    \begin{aligned}
        F_0=&\Big\{ \sup_{t \in [0, T]} \left| \x_{m}^{R}(t) - \q(t) \right| \leq \frac{\xi}{R}, \sup_{t \in [0, T]} \left| \x_{m}^{R}(t) \right| \geq \frac{R}{N^2} \Big\}\\
    \subseteq& \Big\{\sup_{t\in[0,T]}|\q(t)|\ge \sqrt{R}\Big\}\subseteq\Big\{\sup_{t\in[0,T]}|\q(t)|^2-\varepsilon_1 \sum_{1\le i<j\le N}\log|q_i(t)-q_j(t)|\ge \sqrt{R}\Big\}.
    \end{aligned}
\end{align}
where the last implication follows from the same argument as in~\eqref{EC:CompareQwithQ2-log}.
For any $1\le i<j\le N$, by triangle inequality,
\begin{equation*}
    \inf_{t \in [0,T]} |q_i(t) - q_j(t)| \leq 2 \sup_{t \in [0,T]} |\x_m^R(t) - \q(t)| + \inf_{t \in [0,T]} |x_{i,m}^R(t) - x_{j,m}^R(t)|.
\end{equation*}
As a result, for $R$ sufficiently large such that $\frac{2 \xi + N^2}{R}\le \frac{1}{\sqrt{R}}$, it holds that
\begin{equation}\label{EC:Fij}
\begin{aligned}
F_{i,j}=&\Big\{ \sup_{t \in [0,T]} \left| \x_m^R(t) - \q(t) \right| \leq \frac{\xi}{R}, \inf_{t \in [0,T]} \left| x_{i,m}^R(t) - x_{j,m}^R(t) \right| \leq \frac{N^2}{R} \Big\} \\
\subseteq& \Big\{ \inf_{t \in [0,T]} \left| q_i(t) - q_j(t) \right| \leq \frac{2 \xi + N^2}{R} \Big\} \subseteq \left\{ \inf_{t \in [0,T]} \left| q_i(t) - q_j(t) \right| \leq \frac{1}{\sqrt{R}} \right\} \\
 =& \Big\{ -\varepsilon_1 \sup_{t \in [0,T]} \log \left| q_i(t) - q_j(t) \right| \geq \frac{1}{2} \varepsilon_1 \log R \Big\}\\
 \subseteq& \Big\{ |\q(t)|^2-\varepsilon_1 \sup_{t \in [0,T]} \log \left| q_i(t) - q_j(t) \right| \geq \frac{1}{2} \varepsilon_1 \log R \Big\}.
\end{aligned}
\end{equation}
We combine~\eqref{EC:F0},~\eqref{EC:Fij} with~\eqref{EC:F0+Fij} to find that
\begin{align*}
    \Big\{ \sup_{t \in [0, T]} \left| \x_{m}^{R}(t) - \q^{R}(t) \right| \leq \frac{\xi}{R}, \sigma_{m}^{R} < T \leq \sigma^{R} \Big\}\subseteq& \Big\{ |\q(t)|^2-\varepsilon_1 \sup_{t \in [0,T]} \log \left| q_i(t) - q_j(t) \right| \geq \frac{1}{2} \varepsilon_1 \log R \Big\}.
\end{align*}
We apply Markov inequality to get
\begin{align*}
    P_{2,1}=\P\Big\{ \sup_{t \in [0, T]} \left| \x_{m}^{R}(t) - \q^{R}(t) \right| \leq \frac{\xi}{R}, \sigma_{m}^{R} < T \leq \sigma^{R} \Big\}\le \frac{2C}{\varepsilon_1\log R}.
\end{align*}
Finally, recall~\eqref{EC:P1+P2} and~\eqref{EC:P2<=P21+P22+P23},
\begin{align*}
    \mathbb{P}\lr{\sup_{t\in [0,T]}|\x_m(t)-\q(t)|>\xi}\le& P_1+P_2\le P_1+P_{2,1}+P_{2,2}+P_{2,3}\\
    \le&\frac{2\mathbb{E}\lr{\sup_{t\in[0,T]}|\x_m^R(t)-\q^R(t)|^2}}{\xi^2}+\frac{C}{\varepsilon_1\log R}.
\end{align*}
In light of the $L^2$ convergence~\eqref{L2 convergence}, the proof of Theorem \ref{thm: smallLE} is completed by passing $m$ to zero and $R$ to infinity.  
\end{proof}

\section{Relativistic Langevin Equation with multiplicative noise}
\label{sec: RLE}

In this section, we address the asymptotic limits of the relativistic model \eqref{E:RelativisticSys:epsilon}, namely, in Section \ref{sec: RLE:ergodicity}, we construct Lypunov function for \eqref{E:RelativisticSys:epsilon} and prove Theorem \ref{thm: ergodicityRLE} giving the mixing rate of \eqref{E:RelativisticSys:epsilon} toward equilibrium, whereas in Section \ref{sec: RLE:Newtonian_limit}, we discuss the proof of Theorem \ref{thm:NewtonianLimit} validating the approximation of \eqref{E:RelativisticSys:epsilon} by \eqref{E:LimitingRelativisticSys} in the Newtonian limit. To avoid confusion in notation with the previous Section~\ref{sec: classical LE}, we denote $\mathcal{L}_N^\varepsilon$ (resp. $\mathcal{L}_{1}^\varepsilon$) and $H_{N,\epsilon}$ (resp. $H_{1,\varepsilon}$) as the infinitesimal generator of~\eqref{E:RelativisticSys:epsilon} (resp.~\eqref{E:RelativisticSysN=1}) and the corresponding Hamiltonian.

\subsection{Ergodicity} \label{sec: RLE:ergodicity}
In this subsection, we proceed to establish Theorem \ref{thm: ergodicityRLE} giving the polynomial mixing for the relativistic model \eqref{E:RelativisticSys:epsilon} (which is \eqref{E:RelativisticSys:epsilon} with $m=\gamma=1$ and $\varepsilon=1/\mathbf{c}^2$). As typical in the literature of mixing rates, the main crucial ingredient is the existence of suitable Lyapunov functions, which characterize the convergence speed. In our setting of the relativistic Langevin dynamics with multiplicative noises, the construction of Lyapunov functions will be given in Propositions \ref{P:RLyapunovSing} and \ref{P:RLyapunovMulti}, respectively corresponding to the case $N=1$ and $N\ge 2$. Together with the minorization property formulated in Theorem \ref{T:ExistandUnique}, we will be able to ultimately conclude the ergodicity result of Theorem \ref{thm: ergodicityRLE}. Since the proof of Theorem \ref{thm: ergodicityRLE} is also standard, we refer the readers to \cite[Theorem 3.5]{hairer2009hot} for a more detailed discussion.

With regard to the case $N=1$, the generator $\mathcal{L}^\varepsilon_1$ of~\eqref{E:RelativisticSys:epsilon} is given by
    \begin{align}\label{E:L_1_Generator}
        \notag \mathcal{L}^\varepsilon_1 f=&\frac{p}{\sqrt{1+\varepsilon|p|^2}}\cdot\partial_q f-D(p)\frac{p}{\sqrt{1+\varepsilon|p|^2}}\cdot\partial_p f+\gamma\div D(p)\cdot\partial_p f\\
        &-\nabla U(q)\cdot \partial_p f - \nabla G(q)\cdot \partial_p f+ \tr(D(p)\nabla^2_p f).
    \end{align}
    The existence of Lyapunov functions for the case $N=1$ is stated below through Proposition \ref{P:RLyapunovSing}. 
\begin{proposition} \label{P:RLyapunovSing}
   When $N=1$, suppose that Assumptions \ref{Assumption U} and \ref{A:G} respectively on the external potential $U$ and the interaction potential $G$ hold and that the diffusion matrix $D$ satisfies \eqref{E:DefD}. Then,
    \begin{align}\label{ER:V_1epsilon}
       V_{1,\varepsilon}=H_{1,\varepsilon}^2+\varepsilon_1\langle p,q\rangle -\frac{\langle p,q \rangle}{|q|}+\kappa_1
    \end{align}
    is a Lyapunov function of~\eqref{E:RelativisticSys:epsilon} for $\varepsilon_1=\varepsilon_1(\varepsilon)$ small enough and $\kappa_1=\kappa_1(\varepsilon)$ large enough,
    where
    \begin{align} \label{ER:H_1epsilon}H_{1,\varepsilon}=\sqrt{1+\varepsilon|p|^2}+\varepsilon U(q)+\varepsilon G(q).
    \end{align}
    In particular, for all $n\ge 1$ and $\varepsilon_1=\varepsilon_1(n)$ sufficiently small, the following holds
    \begin{align} \label{ineq:L_1V_1:Relativistic}
        \mathcal{L}^\varepsilon_1 V_{1,\varepsilon}^n\le -c_n V_{1,\varepsilon}^{n-\frac{1}{2}}+C_n, 
    \end{align}
    for some positive constants $c_n,C_n$ independent of $X=(p,q)\in\X$.
\end{proposition}
\begin{proof}
Recalling $\mathcal{L}^\varepsilon_1$ from \eqref{E:L_1_Generator}, we apply $\mathcal{L}^\varepsilon_1$ to $H_{1,\varepsilon}$ to obtain
    \begin{align*}
        \mathcal{L}^\varepsilon_1 H_{1,\varepsilon}=-\Big\langle D(p)\frac{p}{\sqrt{1+\varepsilon|p|^2}},\frac{\varepsilon p }{\sqrt{1+\varepsilon|p|^2}} \Big\rangle+\Big\langle\frac{\varepsilon dp}{\sqrt{1+\varepsilon|p|^2}},\frac{\varepsilon p }{\sqrt{1+\varepsilon|p|^2}}\Big\rangle+\tr \Big(D(p)\nabla_p^2 H_{1,\varepsilon}\Big),
    \end{align*}
    where
    \begin{align*}
        (\nabla_p^2 H_{1,\varepsilon})_{ij}=\nabla_j\frac{\varepsilon p_i}{\sqrt{{1+\varepsilon|p|^2}}}=\varepsilon\frac{\delta_{ij}\sqrt{{1+\varepsilon|p|^2}}-\varepsilon p_i p_j/\sqrt{{1+\varepsilon|p|^2}}}{{1+\varepsilon|p|^2}}.
    \end{align*}
    Letting $p^*$ denote the transpose of $p$, from expression~\eqref{E:DefD} of $D$, we have
    \begin{align*}
        \mathcal{L}^\varepsilon_1 H_{1,\varepsilon}&=-\frac{\varepsilon}{1+\varepsilon|p|^2}p^*\Big(\frac{1}{\sqrt{{1+\varepsilon|p|^2}}}\Big(I+\varepsilon p p^*\Big)\Big)p+\langle\frac{\varepsilon dp}{\sqrt{1+\varepsilon|p|^2}},\frac{\varepsilon p }{\sqrt{1+\varepsilon|p|^2}}\rangle\\
        &\qquad+\tr\Big(\frac{1}{\sqrt{1+\varepsilon|p|^2}}(I+\varepsilon p p^*)\frac{\varepsilon}{(1+\varepsilon|p|^2)^{\frac{3}{2}}}\Big((1+\varepsilon|p|^2)I-\varepsilon p p^*\Big)\Big)\\
        &=-\frac{\varepsilon |p|^2}{\sqrt{1+\varepsilon|p|^2}}+\frac{\varepsilon^2d}{1+\varepsilon|p|^2}|p|^2+\frac{d\varepsilon}{1+\varepsilon|p|^2}\\
         &\le-\frac{\varepsilon|p|^2}{\sqrt{1+\varepsilon|p|^2}}+2\varepsilon d.
    \end{align*}
    Next, considering $H_{1,\varepsilon}^2$, a routine calculation gives
    \begin{align*}
        \mathcal{L}^\varepsilon_1 H_{1,\varepsilon}^2=&2H_{1,\varepsilon}\Big[\frac{p}{\sqrt{1+\varepsilon|p|^2}}\cdot\partial_q H_{1,\varepsilon}-D(p)\frac{p}{\sqrt{1+\varepsilon|p|^2}}\cdot\partial_p H_{1,\varepsilon}+\gamma\div D(p)\cdot\partial_p H_{1,\varepsilon}\\
        &\qquad\qquad-\nabla U(q)\cdot \partial_p H_{1,\varepsilon} - \nabla G(q)\cdot \partial_q H_{1,\varepsilon}\Big]+ \tr(D(p)\nabla^2_p H
        _1^2)\\
        =&2H_{1,\varepsilon} \mathcal{L}^\varepsilon_1 H_{1,\varepsilon} + \frac{2\varepsilon |p|^2}{1+\varepsilon |p|^2}.
    \end{align*}
   Since $H_{1,\varepsilon}\geq \sqrt{1+\varepsilon|p|^2}$, we find
    \begin{align}\label{ER:L_1H_1^2}
        \mathcal{L}^\varepsilon_1 H_{1,\varepsilon}^2\leq -\varepsilon |p|^2 +2d\varepsilon H_{1,\varepsilon}+2.
    \end{align}
  Letting $\varepsilon_1$ be given and be chosen later, we consider the cross term $\varepsilon_1\la p,q\ra$ and observe that
    \begin{align*}
        \mathcal{L}^\varepsilon_1 (\varepsilon_1 \langle p, q \rangle)=&\Big\langle\frac{\varepsilon_1 p }{\sqrt{1+\varepsilon|p|^2}},p \Big\rangle-\Big\langle D(p)\frac{\varepsilon_1 p}{\sqrt{1+\varepsilon|p|^2}},q  \Big\rangle+\varepsilon \langle\div D(p),q \rangle\\
        &-\varepsilon_1\langle \nabla U(q),q \rangle-\varepsilon_1\langle \nabla G(q),q \rangle.
    \end{align*}
    Recalling $D$ from expression \eqref{E:DefD}, it holds that
    \begin{align*}
        -\Big\langle D(p)\frac{\varepsilon_1 p}{\sqrt{1+\varepsilon|p|^2}},q  \Big\rangle
        &=-\frac{\varepsilon_1}{\sqrt{1+\varepsilon|p|^2}}p^*\frac{1}{\sqrt{1+\varepsilon|p|^2}}(I+\varepsilon p p^*)q\\
        &=\frac{-\varepsilon_1}{1+\varepsilon|p|^2}\langle p ,q \rangle-\frac{\varepsilon_1 \varepsilon}{1+\varepsilon|p|^2}|p|^2 \langle p,q \rangle\\
        &=-\varepsilon_1\langle p,q \rangle,
    \end{align*}
   whence,
    \begin{align*}
        \mathcal{L}^\varepsilon_1 (\varepsilon_1 \langle p, q \rangle)=&\frac{\varepsilon_1 |p|^2}{\sqrt{1+\varepsilon|p|^2}}-\varepsilon_1\Big(1+\frac{(1-\varepsilon(1+d))}{\sqrt{1+\varepsilon|p|^2}}\Big)\langle p,q \rangle\\
        &-\varepsilon_1\langle \nabla U(q),q \rangle-\varepsilon_1\langle \nabla G(q),q \rangle.
    \end{align*}
    By Cauchy-Schwarz inequality and conditions to $U$ and $G$, for $\varepsilon$ sufficiently small such that $\varepsilon(1+d)\le 1/2$
    \begin{align*}
        \mathcal{L}^\varepsilon_1 (\varepsilon_1 \langle p, q \rangle)&\leq \frac{\varepsilon_1}{\sqrt{\varepsilon}}|p|+\varepsilon_1\lr{C(a_2)|p|^2+\frac{1}{4}a_2|q|^2}+ \frac{\varepsilon_1}{\sqrt{\varepsilon}}|q|\\
        &\qquad+\varepsilon_1\Big(-\frac{1}{2}a_2 |q|^{\lambda+1}+a_1|q|+\frac{a_1}{|q|^{\beta_1-1}}+C\Big),
    \end{align*}
    where the constant $C(a_2)$ depends only on $a_2$.
    
With regard to the singular term $-\frac{\la p,q\ra }{|q|}$, we have
    \begin{align*}
        \mathcal{L}^\varepsilon_1\Big(-\frac{\langle p ,q \rangle}{|q|}\Big)&=-\Big\langle \frac{p}{\sqrt{1+\varepsilon|p|^2}},\frac{p}{|q|}  \Big\rangle+\Big\langle \frac{p}{\sqrt{1+\varepsilon|p|^2}},\langle p, q \rangle\frac{q}{|q|^3}\Big\rangle\\
        &+\Big\langle D(p)\frac{p}{\sqrt{1+\varepsilon|p|^2}},\frac{q}{|q|}\Big\rangle -\Big\langle \div D(p), \frac{q}{|q|}\Big\rangle + \Big\langle \nabla U, \frac{q}{|q|}\Big\rangle + \Big\langle \nabla G, \frac{q}{|q|}\Big\rangle.
    \end{align*}
    As $D$ is given by~\eqref{E:DefD}, the above identity is equivalent to
    \begin{align*}
        \mathcal{L}^\varepsilon_1\Big(-\frac{\langle p ,q \rangle}{|q|}\Big)&=-\frac{|p|^2}{|q|\sqrt{1+\varepsilon|p|^2}}+\frac{|\langle p,q \rangle|^2}{|q|^3 \sqrt{1+\varepsilon|p|^2}}+\frac{\langle p,q \rangle}{|q|}\\
        &-\frac{\varepsilon d}{|q|\sqrt{1+\varepsilon|p|^2}}\langle p,q \rangle+\frac{\langle \nabla U ,q\rangle}{|q|}+\frac{ \langle \nabla G, q \rangle}{|q|}.
    \end{align*}
    It is clear that
    \begin{align*}
        -\frac{|p|^2}{|q|\sqrt{1+\varepsilon|p|^2}}+\frac{|\langle p,q \rangle|^2}{|q|^3 \sqrt{1+\varepsilon|p|^2}}\leq 0,
    \end{align*}
    and that
    \begin{equation}\label{E:G_trick}
        \begin{aligned}
            \frac{ \langle \nabla G, q \rangle}{|q|}=&-\frac{a_4}{|q|^{\beta_1}}+\frac{\innerbig{\nabla G(q)+a_4\frac{q}{|q|^{\beta_1+1}}}{q}}{|q|}\\
        \le &-\frac{a_4}{|q|^{\beta_1}}+\frac{|a_5|}{|q|^{\beta_2}}+a_6\\
        \le &-\frac{a_4}{2|q|^{\beta_1}}+C.
        \end{aligned}
    \end{equation}
   As a consequence, 
    \begin{align*}
         \mathcal{L}^\varepsilon_1\Big(-\frac{\langle p ,q \rangle}{|q|}\Big)\leq |p|+\sqrt{\varepsilon}d+a_1|q|^\lambda-\frac{a_4}{2|q|^{\beta_1}}+C.
    \end{align*}
    Altogether, we obtain
    \begin{align*}
        \mathcal{L}^\varepsilon_1 V_{1,\varepsilon} \leq& -(1+\varepsilon|p|^2)+d\varepsilon H +\frac{\varepsilon_1}{\sqrt{\varepsilon}}+\varepsilon_1\lr{C(a_2)|p|^2+\frac{1}{4}a_2|q|^2}+\frac{\varepsilon_1}{\sqrt{\varepsilon}}|q|\\
        &+\varepsilon_1\Big(-\frac{1}{2}a_2|q|^{\lambda+1}+\frac{a_1}{|q|^{\beta_1-1}}\Big)+|p|+\frac{1}{2\sqrt{\varepsilon}}+a_1|q|^\lambda-\frac{a_4}{2|q|^{\beta_1}}+C.
    \end{align*}
    We can subsume $|p|$ into $|p|^2$ and $|q|^\lambda$ into $|q|^{\lambda+1}$ and get
    \begin{align*}
        \mathcal{L}^\varepsilon_1 V_{1,\varepsilon}\leq -\frac{\varepsilon-2\varepsilon_1 C(a_2)}{2\varepsilon}(1+\varepsilon|p|^2)+d\varepsilon H -\frac{\varepsilon_1}{8}|q|^{\lambda+1}-\frac{a_4}{4}\frac{1}{|q|^{\beta_1}}+C.
    \end{align*}
 Since $ 1+\varepsilon |p|^2\geq \sqrt{1+\varepsilon |p|^2}$, we may choose $\varepsilon_1\le \frac{C(a_2)\varepsilon}{4}$ and get
    \begin{align*}
        \mathcal{L}^\varepsilon_1 V_{1,\varepsilon}\leq& -\frac{1}{4}\sqrt{1+\varepsilon|p|^2}+d\varepsilon H -\frac{\varepsilon}{8}|q|^{\lambda+1}-\frac{a_4}{4}\frac{1}{|q|^{\beta_1}}+C\\
        \leq&-cH_{1,\varepsilon}+C
    \end{align*}
    where the constant $c$ does not depend on $\varepsilon$. By setting $\varepsilon_1$ sufficiently small and $\kappa_1$ sufficiently large in $V_{1,\varepsilon}$~\eqref{ER:V_1epsilon}, we see that $V_{1,\varepsilon}\ge 1$ and
    \begin{align}\label{ER:H^2V1H^2}
        \frac{1}{2}H_{1,\varepsilon}^2-C\le V_{1,\varepsilon}\le \frac{3}{2}H_{1,\varepsilon}^2+C,
    \end{align}
    for certain large $C$ depending on $\varepsilon$. As a consequence, we see that
    \begin{align}\label{ER:L1V11}
        \mathcal{L}^\varepsilon_1 V_{1,\varepsilon}\leq -c \sqrt{V_{1,\varepsilon}}+C,
    \end{align}
    where $c>0$ is independent of $\varepsilon$. This proves \eqref{ineq:L_1V_1:Relativistic} for the case $n=1$.
    For $n\ge 2$, we apply It\^{o}'s formula to $V_{1,\varepsilon}^n$ to get
    \begin{align}\label{ER:L1V1n}
        \notag \mathcal{L}_{1}^\varepsilon V_{1,\varepsilon}^n=&n V_{1,\varepsilon}^{n-1}\mathcal{L}_{1}^\varepsilon V_{1,\varepsilon}+n(n-1)V_{1,\varepsilon}^{n-2}\tr\Big(D(p)\nabla_p V_{1,\varepsilon}\otimes \nabla_p V_{1,\varepsilon}\Big)\\
        =&n V_{1,\varepsilon}^{n-1}\mathcal{L}_{1}^\varepsilon V_{1,\varepsilon}+n(n-1)V_{1,\varepsilon}^{n-2}\langle\nabla_p V_{1,\varepsilon},D(p)\nabla_p V_{1,\varepsilon}\rangle.
    \end{align}
    Considering the inner product on the above right-hand side, recall the diffusion matrix $D(p)$ given by~\eqref{E:DefD} whose largest eigenvalue is $\sqrt{1+\varepsilon|p|^2}$. We obtain from \eqref{ER:H^2V1H^2} that
    \begin{align*}
        \langle\nabla_p V_{1,\varepsilon},D(p)\nabla_p V_{1,\varepsilon}\rangle\le \sqrt{1+\varepsilon|p|^2}|\nabla_p V_{1,\varepsilon}|^2\le H_{N,\varepsilon}|\nabla_p V_{1,\varepsilon}|^2\le cV_{1,\varepsilon}^{1/2}|\nabla_p V_{1,\varepsilon}|^2+C|\nabla_p V_{1,\varepsilon}|^2.
    \end{align*}
    Since $V_{1,\varepsilon}$ is given by~\eqref{ER:V_1epsilon},
    \begin{align*}
        |\nabla_p V_{1,\varepsilon}|^2\le &3\Big|2H_{1,\varepsilon}\frac{\varepsilon p}{\sqrt{1+\varepsilon|p|^2}}\Big|^2+3\varepsilon^2|q|^2+3\\
        \le & 12\varepsilon H_{1,\varepsilon}^2+3a_1\varepsilon H_{1,\varepsilon}+C\\
        \le &\varepsilon cH_{1,\varepsilon}^2+C\\
        \le &\varepsilon cV_{1,\varepsilon}+C,
    \end{align*}
    where the second inequality follows~\eqref{cond:U:U(x)>=O(x^lambda+1)} and the Hamiltonian structure $H_{1,\varepsilon}$~\eqref{ER:H_1epsilon}. Hence from~\eqref{ER:L1V11} and~\eqref{ER:L1V1n}, we get
    \begin{align*}
        \mathcal{L}_{1}^\varepsilon V_{1,\varepsilon}^n\le -ncV_{1,\varepsilon}^{n-\frac{1}{2}}+\varepsilon cn(n-1)V_{1,\varepsilon}^{n-\frac{1}{2}}+ C n(n-1)V_{1,\varepsilon}^{n-\frac{3}{2}}+\varepsilon CV_{1,\varepsilon}+C.
    \end{align*}
    We complete the proof of Proposition~\ref{P:RLyapunovSing} by choosing $\varepsilon$ sufficiently small. 
\end{proof}

Turning to the multi-particle case, the analogue of \eqref{E:L_1_Generator} when $N\ge 2$ is given by
\begin{align}\label{ER:GeneratorN}
    \notag \mathcal{L}^\varepsilon_Nf=&\sum_{i=1}^N\frac{p}{\sqrt{1+\varepsilon|p_i|^2}}\cdot\partial_{q_i}f-\sum_{i=1}^N D(p_i)\frac{p_i}{\sqrt{1+\varepsilon|p_i|^2}}\cdot\partial_{p_i}f+\sum_{i=1}^N\div_{p_i}D(p_i)\cdot\partial_{p_i}f\\
    &-\sum_{i}\nabla U(q_i)\cdot\partial_{p_i}f-\sum_{1\le i<j\le N}\nabla G(q_i-q_j)\cdot[\nabla_{p_i}f-\nabla_{p_j}f]+\sum_{i=1}^N\tr \lr{D(p_i)\nabla_{p_i}^2f}.
\end{align}
In Proposition \ref{P:RLyapunovMulti}, stated and proven next, we establish the existence of Lyapunov functions for \eqref{E:RelativisticSys:epsilon} when $N\ge 2$. We note that in this situation, we have to further restrict the class of singular potential $G$ by imposing Assumption \ref{A:G:relativistic:N>1} in addition to Assumption \ref{A:G}.
\begin{proposition} \label{P:RLyapunovMulti}
      When $N\ge 2$, suppose that $U$ satisfies Assumption \ref{Assumption U}, $G$ satisfies Assumptions \ref{A:G} and \ref{A:G:relativistic:N>1} and that the diffusion matrix $D$ is defined as in \eqref{E:DefD}. Then,
    \begin{align}\label{ER:DefVN}
       V_{N,\varepsilon}=A_1 H_{N,\varepsilon}^3+\varepsilon H_{N,\varepsilon}\la \q,\p\ra-A_2\varepsilon^2\sum_{i\neq j}\frac{\innerbig{q_i-q_j}{p_i-p_j}}{|q_i-q_j|^{\beta_1-1}}+\kappa_N
    \end{align}
    is a Lyapunov function of~\eqref{E:RelativisticSys:epsilon} for $\varepsilon$ sufficiently small and $\kappa_N$ large enough, and suitable positive constants $A_1,A_2$. In the above,
    \begin{align*}
        H_{N,\varepsilon}=\sum_{i=1}^N\sqrt{1+\varepsilon|p_i|^2}+\sum_{i=1}^N\varepsilon U(q_i)+\sum_{i\ne j}\varepsilon G(q_i-q_j).
    \end{align*}
    Particularly, for all $n\ge 1$, the following holds
    \begin{align} \label{ineq:L_NV_N:Relativistic}
        \mathcal{L}^\varepsilon_N V_{N,\varepsilon}^n\le -c_n V_{N,\varepsilon}^{n-\frac{1}{3}}+C_n, 
    \end{align}
    for some positive constants $c_n=c_n(\varepsilon,N)$ and $C_n=C_n(\varepsilon,N)$.
    
\end{proposition}
\begin{proof}

Similar to the single-particle system, we first apply $\mathcal{L}^\varepsilon_N$ to $H_{N,\varepsilon}$ to get
\begin{align*}
    L_NH_{N,\varepsilon}=-\sum_{i=1}^N\frac{\varepsilon|p_i|^2}{\sqrt{1+\varepsilon|p_i|^2}}+\sum_{i=1}^N\frac{\varepsilon(\varepsilon d)|p_i|^2}{1+\varepsilon|p_i|^2}+\sum_{i=1}^N \frac{d\varepsilon}{1+\varepsilon|p_i|^2}.
\end{align*}
It follows that 
\begin{align*}
    L_NH_{N,\varepsilon}^3=&3H_{N,\varepsilon}^2\lr{-\sum_{i=1}^N\frac{\varepsilon|p_i|^2}{\sqrt{1+\varepsilon|p_i|^2}}+\sum_{i=1}^N\frac{\varepsilon(\varepsilon d)|p_i|^2}{1+\varepsilon|p_i|^2}+\sum_{i=1}^N \frac{d\varepsilon}{1+\varepsilon|p_i|^2}}\\
    &\qquad+6H_{N,\varepsilon}\sum_{i=1}^N\frac{\varepsilon^2|p_i|^2}{\sqrt{1+\varepsilon|p_i|^2}}\\
    \le &3H_{N,\varepsilon}^2\lr{-\sum_{i=1}^N\frac{\varepsilon|p_i|^2}{\sqrt{1+\varepsilon|p_i|^2}}+2\varepsilon Nd}+6\varepsilon H_{N,\varepsilon}^2.
\end{align*}
Next, concerning the cross term $\varepsilon H_{N,\varepsilon}\la \p,\q\ra=\sum_{i=1}^N\varepsilon H_{N,\varepsilon}\la p_i,q_i\ra$, we have
\begin{align*}
    &\mathcal{L}^\varepsilon_N\lrbig{\varepsilon H_{N,\varepsilon}\la \p,\q\ra }
    \\&=\varepsilon H_{N,\varepsilon} \mathcal{L}^\varepsilon_N\la \p,\q\ra +\varepsilon \lrbig{L_NH_{N,\varepsilon}}\la \p,\q\ra +\frac{1}{2}\sum_{i=1}^N\tr \lr{D(q_i)\lr{\nabla_{p_i}H_{N,\varepsilon}\otimes \nabla_{p_i}\la \p,\q\ra }}\\
    &=\varepsilon H_{N,\varepsilon} \mathcal{L}^\varepsilon_N\la \p,\q\ra +\varepsilon\lr{-\sum_{i=1}^N\frac{\varepsilon|p_i|^2}{\sqrt{1+\varepsilon|p_i|^2}}+\sum_{i=1}^N\frac{\varepsilon(\varepsilon d)|p_i|^2}{1+\varepsilon|p_i|^2}+\sum_{i=1}^N \frac{d\varepsilon}{1+\varepsilon|p_i|^2}}\la \p,\q\ra +c\varepsilon\la \p,\q\ra \\
    &\le  \varepsilon H_{N,\varepsilon} \mathcal{L}^\varepsilon_N\la \p,\q\ra +c\varepsilon\la \p,\q\ra \le \varepsilon H_{N,\varepsilon} \mathcal{L}^\varepsilon_N\la \p,\q\ra +C\varepsilon \la \p,\q\ra .
\end{align*}
Morevoer, the first term on the above right hand side can be decomposed as follows.
\begin{align*}
    \varepsilon H_{N,\varepsilon} \mathcal{L}^\varepsilon_N\la p,q\ra=&\varepsilon H_{N,\varepsilon}\bigg[\sum_{i=1}^N \frac{|p_i|^2}{\sqrt{1+\varepsilon|p_i|^2}}-\sum_{i=1}^N\inner{D(p_i)\frac{p_i}{{\sqrt{1+\varepsilon|p_i|^2}}}}{q_i}+\sum_{i=1}^N\inner{\div_{p_i}D(p_i)}{q_i}\\
    &-\sum_{i=1}^N\inner{\nabla U(q_i)}{q_i}-\sum_{1\le i<j\le N}\inner{\nabla G(q_i-q_j)}{q_i-q_j}\bigg]\\
    \eqqcolon& H_{N,\varepsilon}\sum_{k=1}^5   R_k.
\end{align*}
With regard to $R_1$, it is clear that
\begin{align*}
    R_1=\varepsilon\sum_{i=1}^N\frac{|p_i|^2}{\sqrt{1+\varepsilon|p_i|^2}}\le \sum_{i=1}^N \sqrt{\varepsilon}|p_i|.
\end{align*}
From \eqref{E:DefD}, we get
\begin{align*}
    R_2=\varepsilon\sum_{i=1}^N\inner{D(p_i)\frac{p_i}{{\sqrt{1+\varepsilon|p_i|^2}}}}{q_i}=&\varepsilon\sum_{i=1}^N\frac{\la q_i,q_i\ra +\varepsilon|p_i|^2\la q_i,q_i\ra }{1+\varepsilon |p_i|^2}=\varepsilon\sum_{i=1}^N\la q_i,q_i\ra .
\end{align*}
Concerning $R_3$, we invoke~\eqref{E:U_3} to infer
\begin{align*}
    R_3=-\varepsilon\sum_{i=1}^N\innerbig{\nabla U(q_i)}{q_i}\le -\varepsilon\sum_{i=1}^N a_2|q_i|^{\lambda+1}+\varepsilon Na_3.
\end{align*}
Concerning $R_4$, we employ~\eqref{E:G_2} to see that
\begin{align*}
    R_4=-\varepsilon\sum_{1\le i<j\le N}\innerbig{\nabla G(q_i-q_j)}{q_i-q_j}\le \varepsilon\sum_{1\le i<j\le N}a_1\lr{|q_i-q_j|+\frac{1}{|q_i-q_j|^{\beta_1-1}}}.
\end{align*}
A routine calculation yields $\div D(p)=\frac{\varepsilon d }{\sqrt{1+\varepsilon|p|^2}}p$, which implies that
\begin{align*}
    R_5=\sum_{i=1}^N \varepsilon\div_{p_i} D(p_i)q_i=\varepsilon&\sum_{i=1}^N\frac{\varepsilon d}{\sqrt{1+\varepsilon|p_i|^2}}\la q_i,q_i\ra \\
    \le& \varepsilon\sum_{i=1}^N\sqrt{\varepsilon}d|q_i|\le \varepsilon\sum_{i=1}^N\frac{d}{2}\lrbig{\varepsilon|q_i|^2+1}.
\end{align*}
Altogether, we obtain
\begin{align*}
    \varepsilon H_{N,\varepsilon} \mathcal{L}^\varepsilon_N\la p,q\ra\le H_{N,\varepsilon}\Big(&\sum_{i=1}^N \sqrt{\varepsilon}|p_i|+\la \p,\q\ra -\varepsilon\sum_{i=1}^N\frac{1}{2}a_2|q_i|^{\lambda+1}\\
    &\qquad\qquad+a_1\varepsilon\sum_{1\le i<j\le N}\frac{1}{|q_i-q_j|^{\beta_1-1}}+C\Big).
\end{align*}
Next for the singular term of $V_{N,\varepsilon}$, we have
\begin{align*}
    &\mathcal{L}^\varepsilon_N\lr{\sum_{1\le i<j\le N}\frac{\la q_i-q_j,p_i-p_j \ra }{|q_i-q_j|^{\beta_1-1}}}\\
    &=\sum_{1\le i<j\le N}\inner{\frac{p_i}{\sqrt{1+\varepsilon|p_i|^2}}-\frac{p_j}{\sqrt{1+\varepsilon|p_j|^2}}}{\frac{p_i-p_j}{|q_i-q_j|^{\beta_1-1}}}\\
    &\qquad-(\beta_1-1)\sum_{1\le i<j\le N}\la q_i-q_j,p_i-p_j \ra \inner{\frac{p_i}{\sqrt{1+\varepsilon|p_i|^2}}-\frac{p_j}{\sqrt{1+\varepsilon|p_j|^2}}}{\frac{p_i-p_j}{|q_i-q_j|^{\beta_1+1}}}\\
    &\qquad+\sum_{1\le i<j\le N}\inner{D(p_i)\frac{p_i}{\sqrt{1+\varepsilon|p_i|^2}}-D(p_j)\frac{p_j}{\sqrt{1+\varepsilon|p_j|^2}}}{\frac{q_i-q_j}{|q_i-q_j|^{\beta_1-1}}}\\
    &\qquad+\sum_{1\le i<j\le N}\inner{\div_{p_i}D(p_i)-\div_{p_j}D(p_j)}{\frac{q_i-q_j}{|q_i-q_j|^{\beta_1-1}}}\\
    &\qquad+\sum_{1\le i<j\le N}\inner{\nabla U(q_i)-\nabla U(q_j)}{\frac{q_i-q_j}{|q_i-q_j|^{\beta_1-1}}}\\
    &\qquad+\sum_{i=1}^N\inner{\sum_{j\neq i}\frac{q_i-q_j}{|q_i-q_j|^{\beta_1-1}}}{\sum_{k\neq i}\nabla G(q_i-q_j)}\eqqcolon \sum_{k=6}^{11} R_k.
\end{align*}
Since $\frac{p}{\sqrt{1+\varepsilon|p|^2}}\le 1/\sqrt{\varepsilon}$, we see that
\begin{align*}
    \varepsilon^2 (R_6+R_7)\le& \sum_{1\le i<j\le N}\varepsilon^{3/4}|p_i-p_j|\frac{\varepsilon^{5/4}}{|q_i-q_j|^{\beta_1-1}}\\
    \le& \sum_{i}\varepsilon^{3/2} |p_i|^2+\sum_{1\le i<j\le N}\frac{\varepsilon^{5/2}}{|q_i-q_j|^{2\beta_1-2}}.
\end{align*}
Concerning $R_8$ we have
\begin{align*}
    \varepsilon^2 R_8=&\varepsilon^2\inner{p_i-p_j}{\frac{q_i-q_j}{|q_i-q_j|^{\beta_1-1}}}\\
    \le& \varepsilon^2\sum_{1\le i<j\le N}\frac{|p_i-p_j|^2}{2}+\frac{|q_i-q_j|^{4-2\beta_1}}{2}\\
    \le& \varepsilon C\sum_{i=1}^N \varepsilon|p_i|^2+\varepsilon^2\sum_{i=1}^N |q_i|^{4-2\beta_1}\le \varepsilon C\sum_{i=1}^N \varepsilon|p_i|^2+\varepsilon\sum_{i=1}^N \varepsilon U(q_i)+C.
\end{align*}
Turning to $R_9$, it holds that
\begin{align*}
    \varepsilon^2 R_9
    =&\varepsilon^3 d\sum_{1\le i<j\le N}\inner{\frac{p_i}{\sqrt{1+\varepsilon|p_i|^2}}-\frac{p_j}{\sqrt{1+\varepsilon|p_j|^2}}}{\frac{q_i-q_j}{|q_i-q_j|^{\beta_1-1}}}\\
    \le &\varepsilon^{3/2} \sum_{i} \varepsilon|q_i|^{2-\beta_1}\le \varepsilon^{3/2}\sum_{i=1}^N \varepsilon U(q_i)+C.
\end{align*}
Using Assumption \ref{Assumption U}, we obtain
\begin{align*}
    \varepsilon^2 R_{10} \leq \varepsilon ^2C \sum_{i=1}^N |\nabla U(q_i)| \sum_{i=1}^N |q_i|^{2 - \beta_1} \leq \varepsilon C \Big( 1 + \sum_{i=1}^N \varepsilon |q_i|^{\lambda + 2 - \beta_1} \Big)\le \varepsilon C\lr{1+\sum_{i=1}^N\varepsilon U(q_i)}.
\end{align*}
Concerning $R_{11}$, we apply a similar decomposition as in~\eqref{E:G_trick} to get
\begin{align*}
    & \sum_{i=1}^N\inner{\sum_{j\neq i}\frac{q_i-q_j}{|q_i-q_j|^{\beta_1-1}}}{\sum_{k\neq i}\nabla G(q_i-q_k)} \\
&=-a_4\sum_{i=1}^N\inner{\sum_{j\neq i}\frac{q_i-q_j}{|q_i-q_j|^{\beta_1-1}}  }{\sum_{k\neq i}\frac{q_i-q_k}{|q_i-q_k|^{\beta_1+1}}}\\
&\qquad+ \sum_{i=1}^N\inner{\sum_{j\neq i}\frac{q_i-q_j}{|q_i-q_j|^{\beta_1-1}}}{\sum_{k\neq i}\nabla G(q_i-q_k)+ a_4\frac{q_i-q_k}{|q_i-q_k|^{\beta_1+1}}}\\
&\le  R_{11}^{(1)}+R_{11}^{(2)}.
\end{align*}
On the one hand, in view of Lemma~\ref{L:s+1timesGamma} and Assumption~\ref{A:G:relativistic:N>1}, by setting $\gamma=\beta_1-1\in (0,1]$ and $s=\beta_1+1\ge 0$, we have
\begin{align*}
    R_{11}^{(1)}=&-a_4 \sum_{i=1}^{N} \Big\langle \sum_{j \neq i} \frac{q_i - q_j}{|q_i - q_j|^{\beta_1 - 1}}, \sum_{k \neq i} \frac{q_i - q_k}{|q_i - q_k|^{\beta_1 + 1}} \Big\rangle \\
    \leq &   -2 a_4 \sum_{1 \leq i < j \leq N} \frac{1}{|q_i - q_j|^{2\beta_1 - 2}}.
\end{align*}
On the other hand, we employ Assumption \ref{A:G} to infer
\begin{align*}
    R_{11}^{(2)}&\le   \sum_{i=1}^N\Big\la  \sum_{j\neq i}\frac{q_i-q_j}{|q_i-q_j|^{\beta_1-1}}  ,\sum_{k\neq i}\nabla G(q_i-q_k)+ a_4\frac{q_i-q_k}{|q_i-q_k|^{\beta_1+1}}   \Big\ra\\
    &\qquad\le C\sum_{i=1}^N |q_i|^{2-\beta_1} \sum_{i=1}^N\sum_{l\neq i}\Big|\nabla G(q_i-x_l)+ a_4\frac{q_i-q_l}{|q_i-q_l|^{\beta_1+1}}   \Big|\\
    &\qquad\le C\sum_{i=1}^N |q_i|^{2-\beta_1}\Big[ \sum_{1\le i<j\le N}\frac{1}{|q_i-q_j|^{\beta_2}}+1\Big].
\end{align*}
In turn, Young's inequality implies that
\begin{align*}
    R_{11}^{(2)}\le &   C\sum_{i=1}^N |q_i|^{4-2\beta_1}+ C\sum_{1\le i<j\le N}\frac{1}{|q_i-q_j|^{2\beta_2}}+C.
\end{align*}
It follows that
\begin{align*}
\varepsilon^2 R_{11}\le -a_4 \varepsilon^2 \sum_{1\le i<j\le N}\frac{1}{|q_i-q_j|^{2\beta_1-1}} + C\varepsilon\sum_{i=1}^N \varepsilon U(q_i)+C.
\end{align*}
Altogether, we deduce
\begin{align*}
    \mathcal{L}^\varepsilon_NV_{N,\varepsilon}\le& 3A_1H_{N,\varepsilon}^2\lr{-\sum_{i=1}^N\frac{\varepsilon|p_i|^2}{\sqrt{1+\varepsilon|p_i|^2}}+2\varepsilon Nd+2\varepsilon }-\frac{1}{2}a_2H_{N,\varepsilon}\sum_{i=1}^N \varepsilon|q_i|^{\lambda+1}\\
    &\qquad+CH_{N,\varepsilon}\lr{\sum_{i=1}^N \sqrt{\varepsilon}|p_i|+\varepsilon\la \p,\q\ra +\sum_{1\le i<j\le N}\frac{\varepsilon}{|q_i-q_j|^{\beta_1-1}}}\\
    &\qquad+CA_2\sum_{i=1}^N\varepsilon|p_i|^2+CA_2\varepsilon \sum_{i=1}^N \varepsilon U(q_i)\\
    &\qquad-a_4A_2 \sum_{1\le i<j\le N}\frac{\varepsilon^2}{|q_i-q_j|^{2\beta_1-2}}+C.
\end{align*}
Observe that
\begin{align*}
    \sum_{i=1}^N\varepsilon|p_i|^2\le \sum_{i}\sqrt{1+\varepsilon|p_i|^2}\sqrt{\varepsilon}|p_i|\le H_{N,\varepsilon}\sqrt{\varepsilon}|\p|\le C H_{N,\varepsilon}\sum_{i=1}^N\sqrt{\varepsilon}|p_i|.
\end{align*}
Letting $\delta$ be given and be chosen later, we employ the elementary inequality $ab\le \delta a^2+\frac{1}{4\delta}b^2$, $a,b\in \R$ to estimate
\begin{align*}
    \varepsilon\la \p,\q\ra \le C(\delta)\sum_{i=1}^N\varepsilon|p_i|^2+\delta\sum_{i=1}^N\varepsilon |q_i|^2,\quad\text{and}\quad
    \sum_{i=1}^N\varepsilon U(q_i)\le H_{N,\varepsilon}.
\end{align*}
Thus, for arbitrary positive constants $A_1,A_2$, we have
\begin{align}\label{ER:LNVNBound}
    \notag \mathcal{L}^\varepsilon_NV_{N,\varepsilon}\le&3A_1H_{N,\varepsilon}^2\lr{-\sum_{i=1}^N\frac{\varepsilon|p_i|^2}{\sqrt{1+\varepsilon|p_i|^2}}+2\varepsilon Nd+2\varepsilon +\frac{C(\delta)}{A_1}\sum_{i=1}^N\sqrt{\varepsilon}|p_i|}\\
    &\qquad-\frac{1}{2}a_2H_{N,\varepsilon}\sum_{i=1}^N \varepsilon|q_i|^{\lambda+1}-a_4A_2 \sum_{1\le i<j\le N}\frac{\varepsilon^2}{|q_i-q_j|^{2\beta_1-2}}+R_{12},
\end{align}
where
\begin{align}\label{ER:DefI12}
    R_{12}=CH_{N,\varepsilon}\lr{A_2\sum_{i=1}^N \sqrt{\varepsilon}|p_i|+\sum_{1\le i<j\le N}\frac{\varepsilon}{|q_i-q_j|^{\beta_1-1}}+\varepsilon CA_2+\delta\sum_{i}\varepsilon|q_i|^2+C},
\end{align}
and $C$ does not depend on $A_1,A_2,\varepsilon$ and $\delta$. At this point, we claim that for $\varepsilon$ and $\delta$ sufficiently small, one may tune $A_1,A_2$ appropriately so as to deduce the following
\begin{align}  \label{ineq:L_NVC}
    \mathcal{L}^\varepsilon_N V_{N,\varepsilon}\le -c_1H_{N,\varepsilon}^2+C_1.
\end{align}

\noindent\textbf{Case 1:} 
 $\varepsilon \max\{|p_1|^2,\dots,|p_n|^2 \}\le K$ for certain small $K=K(A_2)$ to be chosen which only depends on $A_2$. Recall that $H_{N,\varepsilon}=\sum_{i=1}^N\sqrt{1+\varepsilon|p_i|^2}+\varepsilon\sum_{i=1}^NU(q_i)+\varepsilon\sum_{1\le i<j\le N}G(q_i-q_j)$. We have 
\begin{align*}
& CH_{N,\varepsilon}\bigg[A_2\sum_{i=1}^N \sqrt{\varepsilon}|p_i| + \sum_{1\le i<j\le N}\frac{\varepsilon}{|q_i-q_j|^{\beta_1-1}}\bigg] \\
&\le C\sqrt{K}A_2H_{N,\varepsilon}+C\sum_{i=1}^N (\varepsilon U(q_i)+\sqrt{K}) \sum_{1\le i<j\le N}\frac{\varepsilon}{|q_i-q_j|^{\beta_1-1}} + C\bigg[  \sum_{1\le i<j\le N}\frac{\varepsilon}{|q_i-q_j|^{\beta_1-1}}\bigg]^2.
\end{align*}
We apply $ab\le ra^2+\frac{1}{4r}b^2$ again to the second term on the right hand side with appropriate coefficient $r=\frac{1}{16}\frac{a_2}{a_1}$ to get
\begin{align*}
    C\sum_{i=1}^N& \varepsilon (U(q_i)+\sqrt{K}) \sum_{1\le i<j\le N}\frac{\varepsilon}{|q_i-q_j|^{\beta_1-1}}\\
    \le& \frac{1}{16}\frac{a_2}{a_1}\Big(\sum_{i=1}^N\varepsilon U(q_i)+N\sqrt{K}\Big)^2+C(a_1,a_2)\sum_{1\le i<j\le N}\frac{\varepsilon^2}{|q_i-q_j|^{2\beta_1-2}}\\
    \le &\frac{1}{8}a_2\sum_{i=1}^N\varepsilon U(q_i)\sum_{j=1}^N\varepsilon(1+|q_j|^{\lambda+1})+C(a_1,a_2)\sum_{1\le i<j\le N}\frac{\varepsilon^2}{|q_i-q_j|^{2\beta_1-2}}+C(K)\\
    \le &\frac{1}{8}a_2 H_{N,\varepsilon} \sum_{i=1}^N \varepsilon|q_i|^{\lambda+1}+\frac{N}{8}a_2 H_{N,\varepsilon}+ C\sum_{1\le i<j\le N}\frac{\varepsilon^2}{|q_i-q_j|^{2\beta_1-2}}+C(K).
\end{align*}
It follows from above estimate and~\eqref{ER:DefI12} that
\begin{align*}
    R_{12}\le &C(\varepsilon,A_2)H_{N,\varepsilon}+\lr{\frac{1}{8}a_2+\delta C} H_{N,\varepsilon} \sum_{i=1}^N \varepsilon|q_i|^{\lambda+1}+ C\sum_{1\le i<j\le N}\frac{\varepsilon^2}{|q_i-q_j|^{2\beta_1-2}} +D.
\end{align*}
Also, it is not difficult to see that
\begin{align*}
    & 3A_1H_{N,\varepsilon}^2\lr{-\sum_{i=1}^N\frac{\varepsilon|p_i|^2}{\sqrt{1+\varepsilon|p_i|^2}}+2\varepsilon Nd+2\varepsilon +\frac{C(\delta)}{A_1}\sum_{i=1}^N\sqrt{\varepsilon}|p_i|} \\
    &\qquad\le6A_1H_{N,\varepsilon}^2\varepsilon (Nd+1) +C(\delta)\sqrt{K} H_{N,\varepsilon}^2.
\end{align*}
Altogether, from~\eqref{ER:LNVNBound}, we have 
\begin{align*}
    \mathcal{L}^\varepsilon_N V_{N,\varepsilon}\le&  -\lr{\frac{3}{8}a_2- C\delta}H_{N,\varepsilon}\sum_{i=1}^N\varepsilon |q_i|^{\lambda+1}-a_4A_2 \sum_{}\frac{\varepsilon^2}{|q_i-q_j|^{2\beta_1-2}}\\
    &+C\sum_{1\le i<j\le N}\frac{\varepsilon^2}{|q_i-q_j|^{2\beta_1-2}} +C(\varepsilon,A_2) H_{N,\varepsilon}\\
    &+6A_1H_{N,\varepsilon}^2\varepsilon (Nd+1) +C(\delta)\sqrt{K} H_{N,\varepsilon}^2.
\end{align*}
By choosing $A_2$ sufficiently large and recalling the kinetic energy $\sum_{i=1}^N\sqrt{1+\varepsilon|p_i|^2}$ is bounded in this case, for certain positive constants $c$ and $C$ depending on $A_2$, we take $\delta$ sufficiently small such that $C\delta\le 1/8$ to get
\begin{align*}
    \mathcal{L}^\varepsilon_N V_{N,\varepsilon}\le &-\frac{1}{4}a_2H_{N,\varepsilon}\sum_{i=1}^N\varepsilon |q_i|^{\lambda+1}-\frac{1}{2}a_4A_2 \sum_{1\le i<j\le N}\frac{\varepsilon^2}{|q_i-q_j|^{2\beta_1-2}}\\
    & +C(\varepsilon,A_2) H_{N,\varepsilon}+6A_1H_{N,\varepsilon}^2\varepsilon (Nd+1) +C(\delta)\sqrt{K} H_{N,\varepsilon}^2+C\\
    \le& -cH_{N,\varepsilon}^2+C(\varepsilon,A_2) H_{N,\varepsilon}+6A_1H_{N,\varepsilon}^2\varepsilon (Nd+1) +C(\delta)\sqrt{K} H_{N,\varepsilon}^2+C.
\end{align*}
Now we can choose $K=K(A_2)$ sufficiently small and shrink $\varepsilon$ to zero if necessary to obtain that
\begin{align*}
    \mathcal{L}^\varepsilon_N V_{N,\varepsilon}\le-cH_{N,\varepsilon}^2+C,
\end{align*}
which proves \eqref{ineq:L_NVC} in Case 1.

\noindent\textbf{Case 2:} 
 $\varepsilon \max\{|p_1|^2,\dots,|p_n|^2 \}> K$.
 Denote the index set $S=\{i\in\{1,\dots,n\}: \varepsilon|p_{i}|^2>K\}$. For any $i_0\in S$,
\begin{align*}
    -\frac{\varepsilon|p_{i_0}|^2}{\sqrt{1+\varepsilon|p_{i_0}|^2}}\le -\frac{\varepsilon|p_{i_0}|^2}{\sqrt{\frac{1}{K}\varepsilon|p_{i_0}|^2+\varepsilon|p_{i_0}|^2}}\le -\sqrt{\frac{K}{K+1}}\sqrt{\varepsilon}|p_{i_0}|.
\end{align*}
 Therefore, by taking $A_1$ large and $\varepsilon$ small, it holds that
\begin{align*}
3A_1H_{N,\varepsilon}^2&\lr{-\sum_{i=1}^N\frac{\varepsilon|p_i|^2}{\sqrt{1+\varepsilon|p_i|^2}}+2\varepsilon Nd+2\varepsilon +\frac{C}{A_1}\sum_{i=1}^N\sqrt{\varepsilon}|p_i|}\\ 
\le& \lr{-\sum_{i\in S}\lr{\sqrt{\frac{K}{K+1}}-\frac{C}{A_1}}\sqrt{\varepsilon}|p_{i}|+\frac{C}{A_1}\sum_{i\notin S}\sqrt{\varepsilon}|p_i|+C\varepsilon}3A_1H_{N,\varepsilon}^2\\
\le &-\lr{\frac{1}{2}\sqrt{\frac{K}{K+1}}\sqrt{K}-\frac{C}{A_1}N\sqrt{K}-C\varepsilon} 3A_1 H_{N,\varepsilon}^2\\
\le &-\frac{3}{4}\frac{K}{\sqrt{K+1}}A_1H_{N,\varepsilon}^2.
\end{align*}
We can bound $R_{12}$ by
\begin{align*}
    R_{12}\le CH_{N,\varepsilon}^2+\varepsilon C A_2 H_{N,\varepsilon}+CH_{N,\varepsilon}\le C_1H_{N,\varepsilon}^2+C_2,
\end{align*}
where $C_1$ is independent of $\varepsilon, A_1, A_2$. We can further take $A_1$ large enough such that
\begin{align*}
    \mathcal{L}^\varepsilon_N V_{N,\varepsilon}\le -\frac{3}{4}\frac{K}{\sqrt{K+1}}A_1 H_{N,\varepsilon}^2+C_1H_{N,\varepsilon}^2+C_2\le -\frac{3}{8}\frac{K}{\sqrt{K+1}}A_1 H_{N,\varepsilon}^2+C_2.
\end{align*}
which verifies \eqref{ineq:L_NVC}, as claimed. 
Recalling $V_{N,\varepsilon}$ is given by~\eqref{ER:DefVN} and $\beta_1\in (1,2]$, for $A_1$ and $\kappa_N$ sufficient large, it is clear that $V_{N.\varepsilon}\ge 1$ and
\begin{align*}
    \frac{1}{2}A_1H_{N,\varepsilon}^3-C\le V_{N,\varepsilon}\le \frac{3}{2}A_1H_{N,\varepsilon}^3+C.
\end{align*}
As a consequence of~\eqref{ineq:L_NVC}, we deduce
\begin{align*}
    \mathcal{L}^\varepsilon_N V_{N,\varepsilon}\le -c V_{N,\varepsilon}^{\frac{2}{3}}+C.
\end{align*}
This establishes \eqref{ineq:L_NV_N:Relativistic} for the case $n=1$. When $n\ge 2$, similar to~\eqref{ER:L1V1n}, we apply It\^{o}'s formula to $V_{N,\varepsilon}^n$
\begin{align}\label{ER:LNVNn}
        \notag \mathcal{L}_{N}^\varepsilon V_{N,\varepsilon}^n=&n V_{N,\varepsilon}^{n-1}\mathcal{L}_{N}^\varepsilon V_{N,\varepsilon}+n(n-1)V_{N,\varepsilon}^{n-2}\sum_{i=1}^N\tr\Big(D(p_i)\nabla_{p_i} V_{N,\varepsilon}\otimes \nabla_{p_i} V_{N,\varepsilon}\Big)\\
        =&n V_{N,\varepsilon}^{n-1}\mathcal{L}_{N}^\varepsilon V_{N,\varepsilon}+n(n-1)V_{N,\varepsilon}^{n-2}\sum_{i=1}^N\langle\nabla_{p_i} V_{N,\varepsilon},D(p_i)\nabla_{p_i} V_{N,\varepsilon}\rangle.
\end{align}
By noting that $\|D(p_i)\|\le \sqrt{1+\varepsilon|p_i|^2}\le H_{N,\varepsilon}\le cV_{N,\varepsilon}^{1/3}$ and 
\begin{align*}
    |\nabla_{p_i} V_{N,\varepsilon}|^2=&\Big|(3A_1H_{N,\varepsilon}^2+\varepsilon\langle \p,\q\rangle)\frac{\varepsilon p_i}{\sqrt{1+\varepsilon|p_i|^2}}+\varepsilon H_{N,\varepsilon}p_i+A_2\varepsilon^2\sum_{j\ne i}\frac{q_i-q_j}{|q_i-q_j|^{\beta_1-1}}\Big|^2\\
    \le& 
    \varepsilon cH_{N,\varepsilon}^4+C\le \varepsilon cV^{4/3}_{N,\varepsilon}+C,
\end{align*}
we have
\begin{align*}
    \mathcal{L}_{N}^\varepsilon V_{N,\varepsilon}^n\le& -cnV_{N,\varepsilon}^{n-\frac{1}{3}}+nCV_{N,\varepsilon}^{n-1}+\varepsilon n(n-1)c V_{N,\varepsilon}^{n-\frac{1}{3}}+ n(n-1)C V_{N,\varepsilon}^{n-2}\\
    \le& -c_nV_{N,\varepsilon}^{n-\frac{1}{3}}+C_n,
\end{align*}
which completes the proof of Proposition~\ref{P:RLyapunovMulti}.
\end{proof}


\subsection{Newtonian limit} \label{sec: RLE:Newtonian_limit}
{In this subsection, we prove Theorem \ref{thm:NewtonianLimit} validating the Newtonian limit, by sending $\mathbf{c}\to \infty$ in~\eqref{E:RelativisticSys:epsilon}, which is equivalent to taking $\varepsilon=1/\mathbf{c}^2\to 0$ in \eqref{E:RelativisticSys:epsilon}. Similar to the small mass limit for the classical system \eqref{E:Classical_Eq}, owing to the presence of the nonlinearities, the proof of Theorem \ref{thm:NewtonianLimit} consists of two main steps: we first consider a truncated system of \eqref{E:RelativisticSys:epsilon} and establish the Newtonian limit toward the corresponding truncated limiting equation. This is presented in Section \ref{sec: RLE:Newtonian_limit:Lipschizt}. Then, we remove the Lipschitz constraint by exploiting suitable energy estimates, thereby ultimately concluding Theorem \ref{thm:NewtonianLimit}. All of this will be discussed in Section~\ref{sec: RLE:Newtonian_limit:Original:N>1} and Section~\ref{sec: RLE:Newtonian_limit:Original:N=1}.}
\subsubsection{Lipschitz system} \label{sec: RLE:Newtonian_limit:Lipschizt}
Recall the truncation function $\theta_R$ given by~\eqref{E:truncation function}. Since $D$ defined in~\eqref{E:DefD} is unbounded and is different from the choice of $D$ in the classical equation \eqref{E:Classical_Eq}, we consider a truncated system of \eqref{E:RelativisticSys:epsilon} as follows.
\begin{equation}\label{E:RTruncSys}
    \begin{aligned}
        \ud q_i^{\varepsilon,R}=&\frac{p^{\varepsilon,R}_i}{\sqrt{1+\varepsilon|p_i^{\varepsilon,R}|^2}}\ud t,\\
    \ud p_i^{\varepsilon,R}=&-D\lrbig{p_i^{\varepsilon,R}}\frac{p_i^{\varepsilon,R}}{\sqrt{1+\varepsilon|p_i^{\varepsilon,R}|^2}}\ud t +\div \lrbig{D\lrbig{p_i^{\varepsilon,R}}}\ud t\\
    &\qquad-\theta_R(|q_i^{\varepsilon,R}|)\nabla U(q_i^{\varepsilon,R})\ud t-\sum_{j\neq i}\theta_R(|q_i^{\varepsilon,R}-q_j^{\varepsilon,R}|^{-1})\nabla G(q_i^{\varepsilon,R}-q_j^{\varepsilon,R})\ud t\\
    &\qquad +\sqrt{2\Big[\theta_R(|p_i^{\varepsilon,R}|)\lr{D\lrbig{p_i^{\varepsilon,R}}-I}+I\Big]}\ud W_i.
    \end{aligned}
\end{equation}
Here we denote $I\in \R^{d\times d}$ to be the identity matrix. The corresponding limiting system in the Newtonian regime as $\varepsilon\to 0$ is given by
\begin{equation}\label{E:RTruncdLim}
    \begin{aligned}
        \ud q_i^R=&p_i^R \ud t,\\
    \ud p_i^R=&-p_i^R\ud t+\sqrt{2}\ud W_i-\theta_R(|q_i^R|)\nabla U(q_i^R)\ud t-\theta_R(|q_i^R-q_j^R|^{-1})\sum_{j\neq i}\nabla G\lr{q_i^R-q_j^R}\ud t.
    \end{aligned}
\end{equation}

\begin{proposition}\label{P:RTruncConv}
    Let $(\q^{\varepsilon,R},\p^{\varepsilon,R})$ and $(\q^{R},\p^{R})$ be the solutions to~\eqref{E:RTruncSys} and~\eqref{E:RTruncdLim} with the same initial conditions $(\q_0,\p_0)$. Then, for any $R>0$ and $n\ge 2$, it holds that 
    \begin{equation} \label{lim:truncated:relativistic}
        \mathbb{E} \sup_{t \in [0, T]} \Big[ \big| \mathbf{q}^{\varepsilon, R}(t) - \mathbf{q}^R(t) \big|^n + \big| \mathbf{p}^{\varepsilon, R}(t) - \mathbf{p}^R(t) \big|^n \Big] \leq \varepsilon^{{\frac{n}{2}}} C(n,R, T, \mathbf{q}_0, \mathbf{p}_0).
    \end{equation}
\end{proposition}
In order to prove Proposition~\ref{P:RTruncConv}, we will need the following auxiliary results on the uniform moment bounds on $\p^{\varepsilon,R}$ and on the diffusion matrix $D$.
\begin{lemma}\label{L:RPUniBd}
    For all $T>0$ and $n\ge 1$,
    \begin{equation} \label{ineq:RPUniBd}
        \E\sup_{t\in[0,T]}\big|\p^{\varepsilon,R}(t)\big|^n\le C(n,R, T, \mathbf{q}_0, \mathbf{p}_0).
    \end{equation}
\end{lemma}
\begin{lemma}\label{L:SpectralBound}
    Let $D(p)$ be given as~\eqref{E:DefD}. For $R\ge 1$, the eigenvalues (or say the spectral norm) of 
    \begin{align*}
        \lr{\sqrt{\Big[\theta_R(|p|)\lrbig{D(p)-I}+I\Big]}-I}^2
    \end{align*}
    are bounded by both $\frac{1}{4}\varepsilon^2R^2|p|^2$ and $\varepsilon^2R^4$. Also, the eigenvalues of 
    \begin{align*}
        \theta_R(|p|)\lrbig{D(p)-I}+I
    \end{align*}
    is bounded by $1+2\varepsilon R^2$.
\end{lemma}
For the sake of clarity, we defer the proofs of Lemmas~\ref{L:RPUniBd} and \ref{L:SpectralBound} to the end of this subsection. Assuming their results, let us conclude Proposition~\ref{P:RTruncConv}.
\begin{proof}[Proof of Proposition~\ref{P:RTruncConv}]
During this proof, we drop the superscript $R$ for notational convenience. Also, we denote
\begin{align*}
    K_i(\q)=-\theta_R(|q_i^{\varepsilon,R}|)\nabla U(q_i^{\varepsilon,R})-\sum_{j\neq i}\theta_R(|q_i^{\varepsilon,R}-q_j^{\varepsilon,R}|^{-1})\nabla G(q_i^{\varepsilon,R}-q_j^{\varepsilon,R})
\end{align*}
which is Lipschitz in the variable $\q$.
Letting $\hat{q}_i=q^\varepsilon_i-q_i$ and $\hat{p}_i=p_i^\varepsilon-p_i$, from \eqref{E:RTruncSys} and \eqref{E:RTruncdLim}, observe that $(\hat{\q},\hat{\p})$ satisfies
\begin{align*}
    \ud \hat{q}_i=&\hat{p}_i \ud t+\frac{p_i^\varepsilon}{\sqrt{1+\varepsilon|p_i^\varepsilon|^2}}\ud t-p_i^\varepsilon\ud t,\\
    \ud \hat{p}_i=&-\hat{p}_i \ud t+\div D(p_i^\varepsilon)\ud t+\lr{\sqrt{2\Big[\theta_R(|p_i^\varepsilon|)\lr{D\lr{p_i^\varepsilon}-I}+I\Big]}-\sqrt{2}I}\ud W_i+\big[K_i(\q^\varepsilon)-K_i(\q)\big]\ud t.
\end{align*}
Applying It\^{o}'s formula, we find
\begin{equation}\label{E:RTrunc_1}
    \begin{aligned}
        \frac{1}{2}\ud \lrbig{|\hat{p}_i|^2+|\hat{q}_i|^2}=&\innerbig{\hat{q}_i}{\hat{p}_i}\ud t+\inner{\hat{q}_i}{\frac{p_i^\varepsilon}{\sqrt{1+\varepsilon|p_i^\varepsilon|^2}}-p_i^\varepsilon}\ud t\\
    &-|\hat{p}_i|^2\ud t+\inner{\hat{p}_i}{\frac{\varepsilon d p_i^\varepsilon}{\sqrt{1+\varepsilon|p_i^\varepsilon|^2}}}\ud t+\innerbig{\hat{p}_i}{K_i(\q^\varepsilon)-K_i(\q)}\ud t\\
    &+\inner{\hat{p}_i}{\lr{\sqrt{2\Big[\theta_R(p_i^\varepsilon)\lr{D\lr{p_i^\varepsilon}-I}+I\Big]}-\sqrt{2}I}\ud W_i}\\
    &+\tr\lr{\lr{\sqrt{\Big[\theta_R(p_i^\varepsilon)\lr{D\lr{p_i^\varepsilon}-I}+I\Big]}-I}^2}\ud t.
    \end{aligned}
\end{equation}
We first note that
\begin{align*}
    \Big|\frac{p_i^\varepsilon}{\sqrt{1+\varepsilon|p_i|^2}}-p_i^\varepsilon\Big|\le\sqrt{\varepsilon}|p_i^\varepsilon|^2,\quad\text{and}\quad \frac{\Big|\varepsilon d p_i^\varepsilon\Big|}{\sqrt{1+\varepsilon|p_i^\varepsilon|^2}}\le \varepsilon d |p_i^\varepsilon|\le \sqrt{\varepsilon}|p_i^\varepsilon|.
\end{align*}
By the Lipschitz property of $K_i$,
\begin{align*}
    \Big|K_i(\q^\varepsilon)-K_i(\q)\Big|\le C(R)|\hat{\q}|.
\end{align*}
An application of Cauchy-Schwarz inequality shows that
\begin{align}\label{E:RTrunc_2}
    \notag\innerbig{\hat{q}_i}{\hat{p}_i}+&\inner{\hat{q}_i}{\frac{p_i^\varepsilon}{\sqrt{1+\varepsilon|p_i^\varepsilon|^2}}-p_i^\varepsilon}
    -|\hat{p}_i|^2+\inner{\hat{p}_i}{\frac{\varepsilon d p_i^\varepsilon}{\sqrt{1+\varepsilon|p_i^\varepsilon|^2}}}+\innerbig{\hat{p}_i}{K_i(\q^\varepsilon)-K_i(\q)}\\
    &\le C(R)\lr{|\hat{q}_i|^2+|\hat{p}_i|^2+|\hat{\q}|^2+{\varepsilon}|p_i^\varepsilon|^4+\varepsilon|p_i^\varepsilon|^2}-|\hat{p}_i|^2.
\end{align}
By Lemma~\ref{L:SpectralBound}, the eigenvalues of $\lr{\sqrt{\Big[\theta_R(|p_i^\varepsilon|)\lrbig{D(p_i^\varepsilon)-I}+I\Big]}-I}^2$ are bounded by $\frac{1}{4}\varepsilon^2R^2|p_i^\varepsilon|^2\le \varepsilon R^2|p_i^\varepsilon|^2$. It follows that the trace is bounded by
\begin{align}\label{E:RTrunc_3}
    \tr\lr{\sqrt{\Big[\theta_R(|p_i^\varepsilon|)\lrbig{D(p_i^\varepsilon)-I}+I\Big]}-I}^2\le d\varepsilon R^2|p_i^\varepsilon|^2.
\end{align}
Letting $\ud M_3(t)=\sum_{i=1}^N\inner{\hat{p}_i}{\lr{\sqrt{2\Big[\theta_R(p_i^\varepsilon)\lrbig{D(p_i^\varepsilon)-I}+I\Big]}-\sqrt{2}I}\ud W_i}$ be the martingale whose quadratic variation is given by
\begin{align*}
    \ud \langle M_3(t)\rangle=\sum_{i=1}^N\inner{\hat{p}_i}{\lr{\sqrt{2\Big[\theta_R(p_i^\varepsilon)\lrbig{D(p_i^\varepsilon)-I}+I\Big]}-\sqrt{2}I}^2\hat{p}_i}\ud t.
\end{align*}
Again by Lemma~\ref{L:SpectralBound}, it holds that
\begin{align*}
    \ud \langle M_3(t)\rangle\le 2\varepsilon^2R^4|\hat{\p}|^2\ud t.
\end{align*}
As a consequence, from~\eqref{E:RTrunc_1},~\eqref{E:RTrunc_2} and~\eqref{E:RTrunc_3}, we find that, for $\varepsilon$ sufficiently small,
\begin{align}\label{E:RTrunc_4}
     \notag \frac{1}{2}\ud \lrbig{|\hat{\p}|^2+|\hat{\q}|^2}\le& C(R)\lr{|\hat{\p}|^2+|\hat{\q}|^2+\varepsilon |\p^\varepsilon|^4+\varepsilon|\p^\varepsilon|^2}\ud t+\ud M_3(t)-|\hat{\p}|^2\ud t\\
     \le &C(R)\lr{|\hat{\p}|^2+|\hat{\q}|^2+\varepsilon |\p^\varepsilon|^4+\varepsilon|\p^\varepsilon|^2}\ud t+\ud M_3(t)-\varepsilon R^4|\hat{\p}|^2\ud t.
\end{align}
We invoke the exponential martingale inequality to infer
\begin{align*}
    &\P\lr{\sup_{t\in [0,T]}\Big[ M_3(t)-\varepsilon R^4|\hat{p}_i|^2\ud s\rangle\Big]>r}\\
    &\qquad\qquad\le \P\lr{\sup_{t\in [0,T]}\Big[M_3(t)-\frac{1}{2\varepsilon} \langle M_3(t)\rangle\Big]>r}\le e^{-\frac{r}{\varepsilon}}.
\end{align*}
We apply tail integral formula with respect to probability measure to obtain
\begin{align}\label{ER:M_3Moments}
    \E \sup_{t\in [0,T]}\Big[M_3(t)-\frac{1}{2\varepsilon} \langle M_3(t)\rangle\Big]^n\le \int_0^\infty n r^{n-1} e^{-\frac{r}{\varepsilon}}\ud r= \varepsilon^n \Gamma(n+1),
\end{align}
where $\Gamma(n+1)=n!$ refers to the Gamma function. Since $(\hat{\q}(0),\hat{\p}(0))=0$, we combine~\eqref{ER:M_3Moments} together with~\eqref{E:RTrunc_4} and H\"{o}lder inequality to deduce
\begin{align*}
    \E\sup_{t\in[0,T]}\big|\lrbig{\hat{\q}(t),\hat{\p}(t)}\big|^{2n}\le& C(n,R)\int_0^T\E\sup_{s\in[0,t]}|\lr{\hat{\q}(s),\hat{\p}(s)}|^{2n}\ud t +\varepsilon^{n} C(n,R)\int_0^T\E\sup_{s\in [0,t]}|\p^\varepsilon(s)|^{4n}\ud t\\
    &+\varepsilon^{n} C(n,R)\int_0^T\E\sup_{s\in [0,t]}|\p^\varepsilon(s)|^{2n}\ud t+\varepsilon^n C(n).
\end{align*}
By using the moment bounds in Lemma~\ref{L:RPUniBd} that is uniform with respect to $\varepsilon$ and applying Gr\"{o}nwall inequality, we establish limit \eqref{lim:truncated:relativistic}, thereby completing the proof.
\end{proof}
Now, we turn to the auxiliary inequalitites stated in Lemmas~\ref{L:RPUniBd} and \ref{L:SpectralBound}. In what follows, we provide the proof of the former result giving the uniformity in $\varepsilon$ for the moment bounds on the solution $\p^{\varepsilon,R}$.
\begin{proof}[Proof of Lemma~\ref{L:RPUniBd}]
Considering~\eqref{E:RTruncSys}, we apply It\^o's formula to compute
    \begin{align*}
        \frac{1}{2}\ud \lrbig{|p_i^\varepsilon|^2+|q_i^\varepsilon|^2}=&\inner{q_i^\varepsilon}{\frac{p_i^\varepsilon}{\sqrt{1+\varepsilon|p_i^\varepsilon|}}}\ud t-|p_i^\varepsilon|^2\ud t+\innerbig{p_i^\varepsilon}{\div \lrbig{D(p_i^\varepsilon)}}\ud t\\
        &-\innerbig{p_i^\varepsilon}{K_i(\q^\varepsilon)}\ud t +\inner{p_i^\varepsilon}{\sqrt{2}\sqrt{\theta_R(p_i^\varepsilon)\lrbig{D(p_i^\varepsilon)-I}+I}\ud W_i}\\
        &+\tr\lr{\theta(p_i^\varepsilon)\lrbig{D(p_i^\varepsilon)-I}+I}\ud t,
    \end{align*}
    where we recall that $K_i(\q^\varepsilon)=-\theta_R(|q_i^\varepsilon|)\nabla U(q_i^\varepsilon))-\sum_{j\neq i}\nabla G\lr{q_i^\varepsilon-q_j^\varepsilon}$ is Lipschitz. By Cauchy-Swarz inequality,
    \begin{align*}
        \inner{q_i^\varepsilon}{\frac{p_i^\varepsilon}{\sqrt{1+\varepsilon|p_i^\varepsilon|}}}\le \frac{1}{2}|q_i^\varepsilon|^2+\frac{1}{2}|p_i^\varepsilon|^2.
    \end{align*}
    Note that the divergence of $D$ defined in \eqref{E:DefD} is given by
    \begin{align*}
        \div \lrbig{D(p_i^\varepsilon)}=\frac{\varepsilon d p_i^\varepsilon}{\lrbig{1+\varepsilon|p_i^\varepsilon|^2}^{3/2}},
    \end{align*}
    whence
    \begin{align*}
        \innerbig{p_i^\varepsilon}{\div \lrbig{D(p_i^\varepsilon)}}\le \varepsilon d |p_i^\varepsilon|^2.
    \end{align*}
    Since $K_i$ is Lipschitz, we have
    \begin{align*}
        -\innerbig{p_i^\varepsilon}{K_i(\q^\varepsilon)}\le \frac{1}{2}|p_i^\varepsilon|^2+C(R)\lrbig{|\q^\varepsilon|^2+1}.
    \end{align*}
    As shown in~\eqref{E:RTrunc_3},
    \begin{align*}
        \tr\lr{\theta_R(p_i^\varepsilon)\lrbig{D(p_i^\varepsilon)-I}+I}\le d \sqrt{1+\varepsilon |p_i^\varepsilon|^2}\le d\lrbig{1+\varepsilon|p_i^\varepsilon|^2}.
    \end{align*}
    Let $\ud M_4(t)=\sum_{i=1}^N\inner{p_i^\varepsilon}{\sqrt{2}\sqrt{\theta_R(p_i^\varepsilon)\lr{D(p_i^\varepsilon)-I}+I}\ud W_i}$ be the martingale process whose the quadratic variation satisfies
    \begin{align*}
        \ud\langle M_4(t)\rangle=2\sum_{i=1}^N \langle p_i^\varepsilon, \big[\theta_R(p_i^\varepsilon)\lrbig{D(p_i^\varepsilon)-I}+I\big]p_i^\varepsilon\rangle.
    \end{align*}
    By Lemma~\ref{L:SpectralBound}, for $\varepsilon$ sufficiently small,
    \begin{align*}
        \ud \langle M_4(t)\rangle\le 2\lrbig{1+2\varepsilon R^2}|\p^\varepsilon|^2 \ud t\le 4|\p^\varepsilon|^2\ud t.
    \end{align*}
    For any $\rho>0$, we have
    \begin{align}\label{E:RTrunc_5}
        \notag \frac{\rho}{2}\ud \lrbig{|\p^\varepsilon|^2+|\q^\varepsilon|^2}\le& \rho C(R)\lrbig{|\p_i^\varepsilon|^2+|\q_i^\varepsilon|^2+1}\ud t-\rho|\p^\varepsilon|^2\ud t+\ud (\rho M_4(t))\\
        \le & \rho C(R)\lrbig{|\p_i^\varepsilon|^2+|\q_i^\varepsilon|^2+1}\ud t-\frac{1}{4\rho}\ud \la \rho M_4(t)\ra+\ud (\rho M_4(t)).
    \end{align}
    We once again invoke the exponential martingale inequality to infer
    \begin{align*}
        \P\lr{\sup_{t\in [0,T]}\Big[ M_4(t)-\frac{1}{4\rho} \langle M_4(t)\rangle\Big]\ge r}\le e^{-\frac{1}{2\rho }r}.
    \end{align*}
    Now we choose $\rho=1/3$. Taking expectation by tail integral formula yields
    \begin{align*}
        \E\exp\Big\{\sup_{t\in [0,T]}\Big[M_4(t)-\frac{3}{4}\langle M_4(t)\rangle \Big]\Big\}\le 1+\int_0^\infty e^{-\frac{3}{2}r}e^r\ud r\le 3.
    \end{align*}
    This together with~\eqref{E:RTrunc_5} produces
    \begin{align*}
        \E\exp\Big\{\sup_{t\in [0,T]}\Big[\frac{1}{6}\big(|\q^\varepsilon(t)|^2+|\p^\varepsilon(t)|^2\big) -C\int_0^t\lrbig{|\q^\varepsilon(s)|^2+|\p^\varepsilon(s)|^2+1}\ud s\Big]\Big\}\le 3\exp\Big\{\frac{1}{2}\big|\lrbig{\q_0,\p_0}\big|^2\Big\}.
    \end{align*}
    As a consequence, we have
    \begin{align*}
        \E\sup_{t\in [0,T]}|\q^\varepsilon(t),\p^\varepsilon(t)|^{2n}\le C(n,T,R)\int_0^T\E \sup_{s\in [0,t]}|\q^\varepsilon(s),\p^\varepsilon(s)|^{2n}\ud t+C(n,R,T,\q_0,\p_0).
    \end{align*}
    By virtue of Gr\"{o}nwall inequality, we obtain \eqref{ineq:RPUniBd}, as claimed.
\end{proof}
 Lastly, in this subsection, we present the proof of Lemma \ref{L:SpectralBound}, giving upper bounds on the eigenvalues of {the truncated version of $D$}. This together with Lemma \ref{L:RPUniBd} were employed to establish Proposition \ref{P:RTruncConv}.

\begin{proof}[Proof of Lemma~\ref{L:SpectralBound}]
    By spectral mapping theorem, (see e.g.~\cite[Theorem 4.10]{conway;96;a}) all eigenvalues of the matrix
    \begin{align*}
        \lr{\sqrt{\Big[\theta_R(|p|)\lrbig{D(p)-I}+I\Big]}-I}^2
    \end{align*}
    can be expressed as
    \begin{align*}
        \lambda'=\lr{\sqrt{\big[\theta_R(|p|)(\lambda-1)+1\big]}-1}^2,
    \end{align*}
    where $\lambda$ is one of the eigenvalues of $D(p)$. Recalling $D(p)$ defined in~\eqref{E:DefD}, the eigenvalues of $D(p)$ are $ \sqrt{1+\varepsilon|p|^2}$ with eigenvector $p/|p|$ ($1$-multiplicity) and $(\sqrt{1+\varepsilon|p|^2})^{-1}$ with eigenvectors ($(d-1)$-multiplicity) that are the orthonormal to $p/|p|$. Note that since $\theta_R$ satisfies~\eqref{E:truncation function}, on the one hand, when $|p|> R+1$, $\lambda'$ vanishes. On the other hand, when $|p|\le R+1$, we employ the elementary inequality $(1+x)^a-1\le ax$ for $0<a<1$ and $x\ge 0$ to infer
    \begin{align*}
        \lambda'\le \lr{\lrbig{1+\varepsilon|p|^2}^{1/4}-1}^2\le \lr{\frac{1}{4}\varepsilon|p|^2}^2.
    \end{align*}
    Since $|p|\le R+1\le 2R$, we deduce further that
    \begin{align*}
        \lambda'\le \frac{1}{4}\varepsilon^2R^2|p|^2\le \varepsilon^2R^4.
    \end{align*}
    In the same way, we can prove the eigenvalues of 
    \begin{align*}
        \theta_R(|p|)\lrbig{D(p)-I}+I
    \end{align*}
    are given by
    \begin{align*}
        \lambda''=\theta_R(|p|)(\lambda-1)+1,
    \end{align*}
    which vanishes when $|p|>R+1$ and are bounded by $\sqrt{1+\varepsilon|p|^2}\le 1+2\varepsilon R^2$ when $|p|\le R+1$.
    The proof is thus completed.
\end{proof}

\subsubsection{The Newtonian limit $\varepsilon\to 0$ when $N\ge 2$} \label{sec: RLE:Newtonian_limit:Original:N>1}

In this subsection, we consider the multi-particle case and establish Theorem \ref{thm:NewtonianLimit}, part (1). More specifically, we aim to extend the convergence result of Proposition \ref{P:RTruncConv} for Lipschitz nonlinearities to those potentials satisfying Assumptions \ref{Assumption U} and \ref{A:G}. Following the approach of \cite{duong2024trend}, this relies on a probabilistic argument making use of energy estimates on the classical Langevin equation \eqref{E:LimitingRelativisticSys} that are given next.

\begin{lemma}\label{L:RUnifEnergy}
Let $(\q,\p)$ be the solution to~\eqref{E:LimitingRelativisticSys}. Then,
    \begin{align} \label{ineq:RUnifEnergy}
    \E \sup_{t\in [0,T]}\Big[\sum_{i=1}^N U\lrbig{q_i(t)}+\sum_{1\le i<j\le N}G\lrbig{q_i(t)-q_j(t)}+\frac{1}{2}|\p(t)|^2\Big]\le C.
\end{align}
\end{lemma}

\begin{remark}
    We note that the statement of Lemma \ref{L:RUnifEnergy} is almost the same as that of~\cite[Lemma 4.3]{duong2024trend}, except for the appearance of the term $\frac{1}{2}|\p|^2$ on the left-hand side of \eqref{ineq:RUnifEnergy}. Nevertheless, the proof of \eqref{ineq:RUnifEnergy} follows the same argument of \cite[Lemma 4.3]{duong2024trend}, and thus is omitted.
\end{remark}

In order to prove Theorem \ref{thm:NewtonianLimit}, part (1), we will also need auxiliary inequalities on the potentials $U$ and $G$, stated below through Lemma \ref{L:U_G_p}, whose proof is similar to that of~\cite[Lemma A.3]{duong2024trend}. 
\begin{lemma}\label{L:U_G_p} Under Assumptions \ref{Assumption U} (U) (i) and \ref{A:G} (G) (i) (ii'), there exist positive constants $C_G$ and $c_G$ such that the followings hold:
    \begin{align} \label{ineq:U_G_p}
        C_G& \Big[ \sum_{i=1}^{N} U(q_i) + \sum_{1 \leq i < j \leq N} G(q_i - q_j)+ \frac{1}{2}|\p|^2 \Big] \notag \\
        &    \geq \sum_{i=1}^{N} |q_i| - c_G \sum_{1 \leq i < j \leq N} \log |q_i - q_j| +\sum_{i=1}^N |p_i|\\
    &\geq c_G \sum_{i=1}^{N} |q_i| +c_G\sum_{i=1}^N |p_i|+ c_G \max\{-\log |q_i - q_j| : 1 \leq i < j \leq N \}.\notag 
    \end{align}
    Furthermore,
    \begin{align}\label{E:q2-log>=0}
        \frac{1}{2}\sum_{i=1}^N|q_i|^2+c_G\sum_{1\le i<j\le N}\log |q_i-q_j|\ge 0.
    \end{align}
\end{lemma}
Now we are in a position to conclude Theorem~\ref{thm:NewtonianLimit}, part (1), cf. \eqref{lim:Newtonian:N>1}, validating the Newtonian approximation of \eqref{E:RelativisticSys:epsilon} by \eqref{E:LimitingRelativisticSys} when $N\ge 2$.
\begin{proof}[Proof of Theorem~\ref{thm:NewtonianLimit}, part (1)]
Letting $(\q^\varepsilon,\p^\varepsilon)$ and $(\q,\p)$ respectively be the solutions of \eqref{E:RelativisticSys:epsilon} and \eqref{E:LimitingRelativisticSys}, the strategy of proving \eqref{lim:Newtonian:N>1} will make use of a stopping time argument. To see this, we introduce the stopping times defined as
\begin{align*}
    \sigma^R=\inf_{t\ge 0}\Big\{\sum_{i=1}^N|q_i(t)|+\sum_{}|q_i-q_j|^{-1}+\sum_{i=1}^N|p_i|\ge R\Big\},
\end{align*}
and
\begin{align*}
    \sigma^R_\varepsilon=\inf_{t\ge 0}\Big\{\sum_{i=1}^N|q_i^\varepsilon(t)|+\sum_{}|q_i^\varepsilon-q_j^\varepsilon|^{-1}+\sum_{i=1}^N|p_i^\varepsilon|\ge R\Big\}.
\end{align*}
We observe that
\begin{align}\label{E_:RMult_1}
    &\mathbb{P} \Big( \sup_{t \in [0,T]} (|\q^{\epsilon}(t) - \q(t)| + |\p^{\epsilon}(t) - \p(t)| > \xi) \Big) \\
    &\leq \mathbb{P} \Big( \sup_{t \in [0,T]} |\q^{\epsilon}(t) - \q(t)| + |\p^{\epsilon}(t) - \p(t)| > \xi, \sigma^R \wedge \sigma^{R}_{\epsilon} \geq T \Big) + \mathbb{P} (\sigma^R \wedge \sigma^{R}_{\epsilon} < T).\notag 
\end{align}
With regard to the first term on the right-hand side of \eqref{E_:RMult_1}, in view of Proposition~\eqref{P:RTruncConv}, we apply Markov’s inequality to find
\begin{align}\label{E_:RMult_2}
    & \mathbb{P} \Big( \sup_{t \in [0,T]} \big| \q^{\varepsilon}(t) - \q(t) \big| + \big| \p^{\varepsilon}(t) - \p(t) \big| \geq \xi, \sigma^R \wedge \sigma^R_{\varepsilon} \geq T \Big)\notag  \\
    & \leq \mathbb{P} \Big( \sup_{t \in [0,T]} \big| \q^{\varepsilon, R}(t) - \q^R(t) \big| + \big| \p^{\varepsilon, R}(t) - \p^R(t) \big| > \xi \Big) \leq \frac{\varepsilon}{\xi} \cdot C(T, R).
\end{align}
In the above, we recall that $(\q^{\varepsilon, R},\p^{\varepsilon, R})$ and $(\q^{R},\p^{R})$ are respectively the solutions of the truncated equations \eqref{E:RTruncSys} and \eqref{E:RTruncdLim}. Concerning the second term on the right-hand side of~\eqref{E_:RMult_1}, we have
\begin{align*}
    &\mathbb{P} \Big( \sigma^R \land \sigma_\varepsilon^R < T \Big) \\
    &\leq \mathbb{P} \Big( \sup_{t \in [0,T]} \big| \q^{\varepsilon,R}(t) - \q^R(t) \big| + \big| \p^{\varepsilon,R}(t) - \p^R(t) \big| \leq \frac{\xi}{R}, \sigma^R \land \sigma_\varepsilon^R < T \Big) \\
    &\quad + \mathbb{P} \Big( \sup_{t \in [0,T]} \big| \q^{\varepsilon,R}(t) - \q^R(t) \big| + \big| \p^{\varepsilon,R}(t) - \p^R(t) \big| > \frac{\xi}{R} \Big)\\
    &\leq \mathbb{P}\Big( \sup_{t \in [0,T]} \big| \q^{\varepsilon, R}(t) - \q^{R}(t) \big| + \big| \p^{\varepsilon, R}(t) - \p^{R}(t) \big| \leq \frac{\xi}{R} , \sigma_\varepsilon^R < T \leq \sigma^R \Big) + \mathbb{P}(\sigma^R < T) \\
    &\quad + \mathbb{P}\Big( \sup_{t \in [0,T]} \big| \q^{\varepsilon, R}(t) - \q^{R}(t) \big| + \big| \p^{\varepsilon, R}(t) - \p^{R}(t) \Big| > \frac{\xi}{R} \Big) \\
    &= Q_1 + Q_2 + Q_3.
\end{align*}
We again estimate $I_3$ using Proposition~\ref{P:RTruncConv} together with Markov inequality.
\begin{align}\label{E_:RI3}
    Q_3 = \mathbb{P} \Big( \sup_{t \in [0, T]} \big| \q^{\varepsilon, R}(t) - \q^R(t) \big| + \big| \p^{\varepsilon, R}(t) - \p^R(t) \big| > \frac{\xi}{R} \Big) \leq \frac{\varepsilon}{\xi} \cdot C(T, R).
\end{align}
To control $Q_2$, we split $R$ into $R/3+R/3+R/3$ to find that
\begin{equation}\label{ER:sigmaR<T}
    \begin{aligned}
        &\big\{\sigma^R < T\big\}\\
    &= \Big\{ \sup_{t \in [0,T]} \sum_{i=1}^{N} |q_i(t)|+\sum_{i=1}^N |p_i| + \sum_{1 \leq i < j \leq N} |q_i(t) - q_j(t)|^{-1} \geq R \Big\} \\
    &\subseteq \Big\{ \sup_{t \in [0,T]} \sum_{i=1}^{N} |q_i(t)| \geq \frac{R}{3} \Big\}\bigcup\Big\{\sup_{t \in [0,T]}\sum_{i=1}^N |p_i|\ge \frac{R}{3}\Big\} \\
    &\hspace{1cm}\bigcup_{1\le i<j\le N} \Big\{ \sup_{t\in [0,T]}-c_G \log |q_i(t) - q_j(t)| \geq c_G \log \Big(\frac{2R}{3N^2}\Big) \Big\} .
    \end{aligned}
\end{equation}
It is clear that $\sum_{i=1}^N|p_i|\ge 0$. We invoke Lemma~\ref{L:U_G_p}, cf. \eqref{E:q2-log>=0}, to obtain for $R$ large enough
\begin{align*}
    &\Big\{ \sup_{t \in [0,T]} \sum_{i=1}^{N} |q_i(t)| \geq \frac{R}{3} \Big\}\\
    &\subseteq  \Big\{ \sup_{t \in [0,T]} \sum_{i=1}^{N} |q_i(t)|+\sum_{i=1}^N |p_i| - c_G \sum_{1 \leq i < j \leq N} \log |q_i(t) - q_j(t)| \geq \frac{R}{6}\Big) \Big\} \\
    &\subseteq \Big\{ \sup_{t \in [0,T]} \sum_{i=1}^{N} |q_i(t)|+\sum_{i=1}^N |p_i| - c_G \sum_{1 \leq i < j \leq N} \log |q_i(t) - q_j(t)| \geq c_G \log \Big(\frac{2R}{3N^2}\Big) \Big\} .
\end{align*}
Again by~\eqref{E:q2-log>=0} and the fact $\frac{1}{2}\sum_{i=1}^N|q_i|^2\ge 0$, for $R$ sufficient large, we have
\begin{align*}
    &\Big\{\sup_{t \in [0,T]}\sum_{i=1}^N |p_i|\ge \frac{R}{3}\Big\}\\
    &\subseteq \Big\{ \sup_{t \in [0,T]} \sum_{i=1}^{N} |q_i(t)|+\sum_{i=1}^N |p_i| - c_G \sum_{1 \leq i < j \leq N} \log |q_i(t) - q_j(t)| \geq c_G \log \Big(\frac{2R}{3N^2}\Big) \Big\} .
\end{align*}
Since for any pair $(i,j)$,
\begin{align*}
    c_G \sum_{i=1}^{N} |q_i| +c_G\sum_{i=1}^N |p_i|+ c_G \max\{-\log |q_k - q_l| : 1 \leq k < l \leq N \}\ge  -c_G\log |q_i - q_j| ,
\end{align*}
we apply Lemma~\ref{L:U_G_p}, cf. the second inequality of~\eqref{ineq:U_G_p} to get
\begin{align*}
    \Big\{ \sup_{t\in [0,T]}&-c_G \log |q_i(t) - q_j(t)| \geq c_G \log \Big(\frac{2R}{3N^2}\Big) \Big\}\\
    &\subseteq \Big\{ \sup_{t \in [0,T]} \sum_{i=1}^{N} |q_i(t)|+\sum_{i=1}^N |p_i| - c_G \sum_{1 \leq i < j \leq N} \log |q_i(t) - q_j(t)| \geq c_G \log \Big(\frac{2R}{3N^2}\Big) \Big\}.
\end{align*}
Now we are able to apply the first inequality of~\eqref{ineq:U_G_p} and~\eqref{ER:sigmaR<T} to get
\begin{align*}
    \big\{ \sigma^R < T \big\} \subseteq \Big\{ \sup_{t \in [0, T]} \sum_{i=1}^{N} U(q_i) + \sum_{1 \le i < j \le N} G(q_i - q_j) +\frac{1}{2}|\p|^2\ge \frac{c_G}{C_G} \log \Big( \frac{2R}{3N^2} \Big) \Big\}.
\end{align*}
In turn, we apply Lemma~\ref{L:RUnifEnergy} to obtain
\begin{align}\label{E_:RI2}
    Q_2 = \P(\sigma^R < T) \leq \frac{C_G}{c_G} \cdot \frac{C(T, \q_0, \p_0)}{\log(R/3N^2)} \leq \frac{C(T)}{\log R}.
\end{align}
Concerning $I_1$, observe that
\begin{align*}
&\Big\{ \sup_{t \in [0,T]} \big| \mathbf{q}^{\varepsilon, R}(t) - \mathbf{q}^R(t) \big| + \big| \mathbf{p}^{\varepsilon, R}(t) - \mathbf{p}^R(t) \big| \leq \frac{\xi}{R},\, \sigma_\varepsilon^R < T \leq \sigma^R \Big\} \\
&= \Big\{ \sup_{t \in [0,T]} \big| \mathbf{q}^{\varepsilon, R}(t) - \mathbf{q}(t) \big| + \big| \mathbf{p}^{\varepsilon, R}(t) - \mathbf{p}(t) \big| \leq \frac{\xi}{R},\, \sigma_\varepsilon^R < T \leq \sigma^R \Big\} \\
&\qquad \bigcap \Big\{ \sup_{t \in [0,T]} \Big( \sum_{i=1}^N \big| q_i^{\varepsilon, R}(t) \big| + \sum_{1 \leq i < j \leq N} \big| q_i^{\varepsilon, R}(t) - q_j^{\varepsilon, R}(t) \big|^{-1}+\sum_{i=1}^N |p_i^{\varepsilon, R}| \Big) \geq R \Big\}.
\end{align*}
Also, note that
\begin{align*}
    &\Big\{ \sup_{t \in [0,T]} \Big( \sum_{i=1}^N \big| q_i^{\varepsilon, R}(t) \big| + \sum_{1 \leq i < j \leq N} \big| q_i^{\varepsilon, R}(t) - q_j^{\varepsilon, R}(t) \big|^{-1}+\sum_{i=1}^N |p_i^{\varepsilon, R}| \Big) \geq R \Big\}\\
    &\quad\subseteq \Big\{\sup_{t \in [0,T]}\sum_{i=1}^N \big| q_i^{\varepsilon, R}(t) \big|\ge \frac{R}{3}\Big\}\bigcup \Big\{\sup_{t \in [0,T]}\sum_{1 \leq i < j \leq N} \big| q_i^{\varepsilon, R}(t) - q_j^{\varepsilon, R}(t) \big|^{-1}\ge \frac{R}{3}\Big\}\\
    &\hspace{1cm}\bigcup \Big\{\sup_{t \in [0,T]}\sum_{i=1}^N |p_i^{\varepsilon, R}|\ge \frac{R}{3}\Big\},
\end{align*}
and that the following implication holds for arbitrary sets $A,B,C$ and $D$, 
\begin{align*}
    \lrbig{A\cap B}\cap\lrbig{C\cup D\cup E}=&\lrbig{A\cap B\cap C}\cup \lrbig{A\cap B\cap D}\cup \lrbig{A\cap B\cap E}\\
    \subseteq&\lrbig{A\cap C}\cup\lrbig{B\cap E}\cup\lrbig{A\cap D}.
\end{align*}
We deduce
\begin{align*}
&\Big\{ \sup_{t \in [0,T]} \big| \mathbf{q}^{\varepsilon, R}(t) - \mathbf{q}^R(t) \big| + \big| \mathbf{p}^{\varepsilon, R}(t) - \mathbf{p}^R(t) \big| \leq \frac{\xi}{R},\, \sigma_\varepsilon^R < T \leq \sigma^R \Big\} \\
&\subseteq \Big\{ \sup_{t \in [0,T]} \big| \mathbf{q}^{\varepsilon, R}(t) - \mathbf{q}(t) \big| \leq \frac{\xi}{R},\, \sup_{t \in [0,T]} \sum_{i=1}^N \big| q_i^{\varepsilon, R}(t) \big| \geq \frac{R}{3} \Big\} \\
&\qquad \bigcup\{\sup_{t \in [0,T]}   \big| \mathbf{p}^{\varepsilon, R}(t) - \mathbf{p}^R(t) \big|\le \frac{\xi}{R},\sup_{t \in [0,T]}\sum_{i=1}^N |p_i^{\varepsilon, R}(t)|\ge \frac{R}{3}\}\\
&\qquad \bigcup_{} \Big\{ \sup_{t \in [0,T]} \big| \q^{\varepsilon, R}(t) - \mathbf{q}(t) \big| \leq \frac{\xi}{R},\, \sup_{t \in [0,T]} \big| q_i^{\varepsilon, R}(t) - q_j^{\varepsilon, R}(t) \big|^{-1} \geq \frac{2R}{3N^2} \Big\}\\
& \eqqcolon B_1\bigcup B_2\bigcup_{1\le i<j\le N} B_{i,j},
\end{align*}
By triangle inequality, we have
\begin{align*}
    B_1 &= \Big\{ \sup_{t \in [0,T]} \big| \q^{\varepsilon, R}(t) - \q(t) \big| \leq \frac{\xi}{R}, \sup_{t \in [0,T]} \sum_{i=1}^{N} \big| q_i^{\varepsilon, R}(t) \big| \geq \frac{R}{3} \Big\} \\
    &\subseteq \Big\{ \sup_{t \in [0,T]} \sum_{i=1}^{N} \big| q_i^{\varepsilon, R}(t) - q_i(t) \big| \leq \frac{\xi}{R \sqrt{N}}, \sup_{t \in [0,T]} \sum_{i=1}^{N} \big| q_i^{\varepsilon, R}(t) \big| \geq \frac{R}{3} \Big\} \\
    &\subseteq \Big\{ \sup_{t \in [0,T]} \sum_{i=1}^{N} \big| q_i(t) \big| \geq \frac{R}{3} - \frac{\xi}{R \sqrt{N}} \Big\}.
\end{align*}
Taking $R$ large enough such that $\frac{R}{3} - \frac{\xi}{R \sqrt{N}}\ge \sqrt{R} $, by Lemma~\ref{L:U_G_p}, we deduce
\begin{align}\label{E_:RB1}
    \notag B_1 \subseteq& \Big\{ \sup_{t \in [0, T]} \sum_{i=1}^{N} |q_i(t)| \geq \sqrt{R} \Big\} & \\
    \subseteq& \Big\{ \sup_{t \in [0, T]} \sum_{i=1}^{N} U(q_i(t)) + \sum_{1 \leq i < j \leq N} G(q_i(t) - q_j(t)) +\frac{1}{2}|\p|^2\geq \frac{c_G}{C_G} \sqrt{R} \Big\}.
\end{align}
Likewise,
\begin{align*}
    \notag B_2 &= \Big\{ \sup_{t \in [0,T]} \big| \p^{\varepsilon, R}(t) - \p(t) \big| \leq \frac{\xi}{R}, \sup_{t \in [0,T]} \sum_{i=1}^{N} \big| p_i^{\varepsilon, R}(t) \big| \geq \frac{R}{3} \Big\} \\
    \notag&\subseteq \Big\{ \sup_{t \in [0,T]} \sum_{i=1}^{N} \big| p_i^{\varepsilon, R}(t) - p_i(t) \big| \leq \frac{\xi}{R \sqrt{N}}, \sup_{t \in [0,T]} \sum_{i=1}^{N} \big| p_i^{\varepsilon, R}(t) \big| \geq \frac{R}{3} \Big\} \\
    & \subseteq \Big\{ \sup_{t \in [0,T]} \sum_{i=1}^{N} \big| p_i(t) \big| \geq \frac{R}{3} - \frac{\xi}{R \sqrt{N}} \Big\}.
\end{align*}
We once again invoke Lemma~\ref{L:U_G_p} to infer
\begin{align} \label{E_:RB2}
    B_2 \subseteq& \Big\{ \sup_{t \in [0, T]} \sum_{i=1}^{N} |p_i(t)| \geq \sqrt{R} \Big\} & \notag \\
    \subseteq& \Big\{ \sup_{t \in [0, T]} \sum_{i=1}^{N} U(q_i(t)) + \sum_{1 \leq i < j \leq N} G(q_i(t) - q_j(t)) +\frac{1}{2}|\p|^2\geq \frac{c_G}{C_G} \sqrt{R} \Big\}.
\end{align}
Turning to $B_{ij}$, using triangle inequality, we have
\begin{align*}
    \inf_{t \in [0,T]} \big| q_i(t) - q_j(t) \big| \leq 2 \sup_{t \in [0,T]} \big| \mathbf{q}^{\varepsilon,R}(t) - \mathbf{q}(t) \big| + \inf_{t \in [0,T]} \big| q_i^{\varepsilon,R}(t) - q_j^{\varepsilon,R}(t) \big|,
\end{align*}
whence
\begin{align*}
B_{ij} = & \Big\{ \sup_{t \in [0,T]} |q^{\epsilon, R}_i(t) - q(t)| \leq \frac{\xi}{R}, \inf_{t \in [0,T]} |q^{\epsilon, R}_i(t) - q^{\epsilon, R}_j(t)| \leq \frac{3N^2}{2R} \Big\} \\
\subseteq & \Big\{ \inf_{t \in [0,T]} |q_i(t) - q_j(t)| \leq \frac{2\xi + 3N^2}{2R} \Big\} \\
= & \Big\{ \sup_{t \in [0,T]} |q_i(t) - q_j(t)|^{-1} \geq \frac{2R}{2\xi + 3N^2} \Big\} \\
= & \Big\{ -\sup_{t \in [0,T]} \log |q_i(t) - q_j(t)| \geq \log 2R - \log(2\xi + 3N^2) \Big\}.
\end{align*}
From Lemma~\ref{L:U_G_p}, for $R$ large enough, we get
\begin{align}\label{E_:RBij}
    \notag B_{ij} \subseteq& \Big\{ -\sup_{t \in [0,T]} \log |q_i(t) - q_j(t)| \geq \log 2R - \log(2\xi + 3N^2) \Big\} \\
    \notag \subseteq& \Big\{ \sup_{t \in [0,T]} \sum_{i=1}^{N} |q_i(t)| - c_G| 
    \sum_{1 \leq i < j \leq N} \log |q_i(t) - q_j(t)|+\frac{1}{2}|\p|^2 \geq \frac{1}{2} c_G \log R \Big\} \\
    \subseteq& \Big\{ \sup_{t \in [0,T]} \sum_{i=1}^{N} U(q_i(t)) + 
    \sum_{1 \leq i < j \leq N} G(q_i(t) - q_j(t))+\frac{1}{2}|\p|^2 \geq \frac{c_G}{2c_G} \log R \Big\}.
\end{align}
So, from~\eqref{E_:RB1},~\eqref{E_:RB2} and~\eqref{E_:RBij}, we obtain
\begin{align}\label{E_:RI1}
    Q_1 = \mathbb{P} \Big\{ \sup_{t \in [0, T]} \big|\q^{\varepsilon, R}(t) - \q^R(t)\big| + \big|\p^{\varepsilon, R}(t) -\p^R(t)\big| \leq \frac{\xi}{R}, \, \sigma_\varepsilon^R < T \leq \sigma^R \Big\} \leq \frac{C(T, q_0, p_0)}{\log R}.
\end{align}
by virtue of Markov inequality and Lemma \ref{L:RUnifEnergy}. Now, we collect estimates~\eqref{E_:RI3},~\eqref{E_:RI2} and ~\eqref{E_:RI1} to find that
\begin{align}\label{E_:RMult_3}
    \mathbb{P}(\sigma^R \land \sigma_\varepsilon^R < T) \leq \frac{\varepsilon}{\xi} \cdot C(T, R) + \frac{C(T)}{\log R}.
\end{align}
Altogether, we combine~\eqref{E_:RMult_2} with~\eqref{E_:RMult_3} and plug into~\eqref{E_:RMult_1} to get
\begin{equation}
    \mathbb{P}\Big( \sup_{t \in [0, T]}  \lrbig{|\mathbf{q}^\varepsilon(t) - \mathbf{q}(t)| + |\mathbf{p}^\varepsilon(t) - \mathbf{p}(t)|} > \xi \Big) \leq \frac{\varepsilon}{\xi} \cdot C(T, R) + \frac{C(T)}{\log R}.
\end{equation}
By passing $R$ to infinity, we arrive at \eqref{lim:Newtonian:N>1}, thereby completing the proof of Theorem~\ref{thm:NewtonianLimit}, part (1), as claimed.
\end{proof}

\subsubsection{The Newtonian limit $\varepsilon\to 0$ when $N=1$} \label{sec: RLE:Newtonian_limit:Original:N=1}
We now turn to the single-particle case, $N=1$, and proceed to prove Theorem~\ref{thm:NewtonianLimit}, part (2). In order to do so, we will derive a suitable energy estimate for equation~\eqref{E:RelativisticSysN=1} that is independent of $\varepsilon$. Such a uniformity will allow us to extend the convergence in probability established in Section \ref{sec: RLE:Newtonian_limit:Original:N>1}
to $L^p$. More specifically, the main ingredient of the proof of Theorem~\ref{thm:NewtonianLimit}, part (2) is stated below through Proposition \ref{P:RSingUnifMomentBd}.
\begin{proposition}[Uniform moment boundedness]\label{P:RSingUnifMomentBd}
When $N=1$, let $\lrbig{q^\varepsilon(t),p^\varepsilon(t)}$ solve~\eqref{E:RelativisticSysN=1} with initial condition $(q_0,p_0)\in\X$. For all $T>0$ and $n>0$, it holds that
    \begin{equation} \label{ineq:sup.E(U+G+p):N=1}
            \sup_{\varepsilon \in (0,1]} \mathbb{E} \sup_{t \in [0,T]} 
    \Big[ U(q^{\varepsilon}(t)) + G(q^{\varepsilon}(t)) + \big| p^{\varepsilon}(t) \big|^2 \Big]^n \leq C(T, n, q_0, p_0).
    \end{equation}
\end{proposition}
In order to establish Proposition \ref{P:RSingUnifMomentBd}, we introduce the following functional
    \begin{align} \label{form:Gamma_2}
        \Gamma_3(q^\varepsilon,p^\varepsilon) = \frac{1}{2} \varepsilon (U(q^\varepsilon) + G(q^\varepsilon))^2 + (U(q^\varepsilon) + G(q^\varepsilon)) \sqrt{1 + \varepsilon |p^\varepsilon|^2} + \frac{1}{2} |p^\varepsilon|^2,
    \end{align}
and the semi-martingale process $M_5(t)$ given by
\begin{align} \label{form:M(t)}
        \ud M_5(t)\coloneqq \inner{\big(U(q^\varepsilon)+G(q^\varepsilon)\big)\frac{\varepsilon p^\varepsilon}{\sqrt{1+\varepsilon|p^\varepsilon|^2}}+p^\varepsilon}{\sqrt{2D(p^\varepsilon)}\ud W}.
    \end{align}
We note that $M_5(t)$ is involved with $\Gamma_3$ through applying It\^o's formula to $\Gamma_3$. It is therefore crucial to establish useful moment bounds on the associated quadratic variation process $\la M_5(t)\ra$. This auxiliary result is summarized in Lemma \ref{L:VariationM_t} below.

\begin{lemma}\label{L:VariationM_t}
    Let $M_5(t)$ be the martingale process given by \eqref{form:M(t)}. Then for all $T>0$, the quadratic variation process $\langle M_5(t)\rangle$ satisfies
    \begin{align*}
        \ud \langle M_5(t)\rangle\le 8\sqrt{2}(\Gamma_3^{3/2}+\Gamma_3)\ud t,
    \end{align*}
    for all $\varepsilon\le 1$.
\end{lemma}

For the sake of clarity, we will defer the proof of Lemma \ref{L:VariationM_t} to the end of this subsection. Assuming its result, let us prove Proposition \ref{P:RSingUnifMomentBd} and conclude Theorem \ref{thm:NewtonianLimit}, part (2).

\begin{proof}[Proof of Proposition \ref{P:RSingUnifMomentBd}]
    First of all, we apply It\^{o}'s formula to $\Gamma_3$ defined in \eqref{form:Gamma_2} and obtain the identity
    \begin{align*}
        \ud \Gamma_3=&-(U+G)\frac{\varepsilon|p^\varepsilon|^2}{\sqrt{1+\varepsilon|p^\varepsilon|^2}}\ud t+(U+G)\frac{\varepsilon^2 d |p^\varepsilon|^2}{1+\varepsilon|p^\varepsilon|^2}\ud t
        -|p^\varepsilon|^2\ud t+\frac{\varepsilon d |p^\varepsilon|^2}{\sqrt{1+\varepsilon|p^\varepsilon|^2}}\ud t\\
        &\qquad+\varepsilon(U+G)\lr{\frac{d-1}{1+\varepsilon|p^\varepsilon|^2}+1+\frac{\varepsilon|p^\varepsilon|^2}{1+\varepsilon|p^\varepsilon|^2}}\ud t+\lr{\frac{d-1}{\sqrt{1+\varepsilon|p^\varepsilon|^2}}+\sqrt{1+\varepsilon|p^\varepsilon|^2}}\ud t\\
        &\qquad+\ud M_5(t),
    \end{align*}
    where $M_5(t)$ is the process given by \eqref{form:M(t)}. It is not difficult to see that all of the positive drift terms are controlled by $\Gamma_3$, namely,
    \begin{align*}
            (U+G)\frac{\varepsilon^2 d |p^\varepsilon|^2}{1+\varepsilon|p^\varepsilon|^2}&\le (U+G)\varepsilon d\le 2d\Gamma_3,\\
            \varepsilon(U+G)\lr{\frac{d-1}{1+\varepsilon|p^\varepsilon|^2}+1+\frac{|p^\varepsilon|^2}{1+\varepsilon|p^\varepsilon|^2}}&=\varepsilon (d+1) \lr {U+G}\le 2(d+1)\Gamma_3,\\
            \frac{d-1}{\sqrt{1+\varepsilon|p^\varepsilon|^2}} &\le d-1,\\
            \frac{\varepsilon d |p^\varepsilon|^2}{\sqrt{1+\varepsilon|p^\varepsilon|^2}}&\le d\sqrt{\varepsilon}|p^\varepsilon|\le \lr {\Gamma_3+1}d.
    \end{align*}
   We therefore may infer positive constants $c$ and $C$ independent of $\varepsilon$ such that
    \begin{align*}
        \ud \Gamma_3\le c\Gamma_3\ud t+\ud M_5(t)+C\ud t.
    \end{align*}
    We apply It\^{o}'s formula to obtain
    \begin{align}\label{ER:DiffGamma2n}
        \ud \Gamma_3^n=&n \Gamma_3^{n-1}\ud \Gamma_3+\frac{1}{2}n(n-1)\Gamma_3^{n-2}\langle \ud \Gamma_3,\ud \Gamma_3\rangle \notag \\
        \le& n (c\Gamma_3^{n}+C)\ud t+4\sqrt{2}n(n-1)(\Gamma_3^{n-1}+\Gamma_3^{n-\frac{1}{2}})\ud t+n \Gamma_3^{n-1}\ud M_5(t),
    \end{align}
    where the last implication follows from Lemma~\ref{L:VariationM_t}. We
    integrate over $[0,t]$ and then take expectation to get
    \begin{align*}
        \E (\Gamma_3(t)^n)\le C_T+C\E\int_0^t \Gamma_3(s)^{n}+\Gamma_3(s)^{n-\frac{1}{2}}+\Gamma_3(s)^{n-1}\ud s.
    \end{align*}
    This implies by virtue of Gr\"{o}nwall's inequality that
    \begin{align*}
        \E \Gamma_3(t)^n\le C,
    \end{align*}
    whence
    \begin{align}\label{ER:IntGamma2n}
        \E\int_0^T\Gamma_3(t)^n\ud t\le C,
    \end{align}
    for certain $C=C(n,T,q_0,p_0)$ independent of $\varepsilon$.
    Now we employ once again~\eqref{ER:DiffGamma2n}, take integration over $[0,t]$, supremum over $t\in [0,T]$ then expectation to find
    \begin{align*}
        \mathbb{E}\sup_{t\in[0,T]}\Gamma_{3}(t)^{n}
        \leq& C_T
        + C\int_{0}^{T}\Gamma_{3}(t)^{n}\ud t
        + n\mathbb{E}\sup_{t\in[0,T]}\left| \int_{0}^{T}\Gamma_{3}(t)^{n-1}\ud M_5(t) \right|\\
        \leq& C_T
        + \mathbb{E}\int_{0}^{T}\big(c\Gamma_{3}(t)^{n}+C\big)\ud t
        + C\mathbb{E}\int_{0}^{T}\big(\Gamma_{3}(t)^{2n-1}+\Gamma_{3}(t)^{2n-\frac{1}{2}}\big)\ud t+ C
        \leq C.
    \end{align*}
    In the above, the first inequality follows from the Burkholder-Davis-Gundy inequality together with Lemma~\ref{L:VariationM_t} and the last inequality follows from~\eqref{ER:IntGamma2n}. By noting that $\Gamma_3\ge U+G+\frac{1}{2}|p|^2$, the proof of Proposition~\ref{P:RSingUnifMomentBd} is thus completed.
\end{proof}
    Now, we are in a position to conclude part (2) of Theorem~\ref{thm:NewtonianLimit}.
    \begin{proof}[Proof of Theorem~\ref{thm:NewtonianLimit}, part (2)]
        Similar to the proof of Theorem \ref{thm:NewtonianLimit}, part (1), we will employ a stopping time argument. Let $\tau^R$ and $\tau^R_\varepsilon$ be defined as
        \begin{align*} 
            \tau^R = \inf_{t\geq 0}\Big\{|q(t)|+|q(t)|^{-1}\geq R\Big\},
        \quad\text{and}\quad
            \tau^R_{\varepsilon} = \inf_{t\geq 0}\Big\{|q^\varepsilon(t)|+ |q^\varepsilon(t))|^{-1}\geq R\Big\}.
        \end{align*}
        where $q^\varepsilon$ and $q$ respectively are the solutions of~\eqref{E:RelativisticSysN=1} and \eqref{E:LimitingRelativisticSys:N=1}. Also, recalling the cut-off function $\theta_R$ in \eqref{E:truncation function}, we introduce the truncated systems 
        \begin{equation}\label{E:RTruncSys:N=1}
    \begin{aligned}
        \ud q^{\varepsilon,R}=&\frac{p^{\varepsilon,R}}{\sqrt{1+\varepsilon|p^{\varepsilon,R}|^2}}\ud t\\
    \ud p^{\varepsilon,R}=&-D\lrbig{p^{\varepsilon,R}}\frac{p^{\varepsilon,R}}{\sqrt{1+\varepsilon|p^{\varepsilon,R}|^2}}\ud t +\div \lrbig{D\lrbig{p^{\varepsilon,R}}}\ud t-\theta_R\big(|q^{\varepsilon,R}|\big)\nabla U\lrbig{q^{\varepsilon,R}}\ud t\\
    &\qquad-\theta_R\big(|q_i^{\varepsilon,R}|^{-1}\big)\nabla G\big(q^{\varepsilon,R}\big)\ud t+\sqrt{2\Big[\theta_R(|p^{\varepsilon,R}|)\lr{D\lrbig{p^{\varepsilon,R}}-I}+I\Big]}\ud W,
    \end{aligned}
\end{equation}
and
\begin{equation}\label{E:RTruncdLim:N=1}
    \begin{aligned}
        \ud q^R=&p^R \ud t,\\
    \ud p_i^R=&-p^R\ud t+\sqrt{2}\ud W-\theta_R\big(|q^R|\big)\nabla U\big(q^R\big)\ud t-\theta_R\big(|q^R|^{-1}\big)\nabla G\lrbig{q^R}\ud t.
    \end{aligned}
\end{equation}
By the definition of $\tau^r$ and $\tau^R_\varepsilon$, we readily have
        \begin{align*}
            \P\Big( 0\le t\le \tau^R \wedge \tau_\varepsilon^R,\, \big(q^\varepsilon(t),p^\varepsilon(t)\big)=\big(q^{\varepsilon,R}(t),p^{\varepsilon,R}(t)\big) , \, \big(q(t),p(t)\big)  =\big(q^R(t),p^R(t)\big)    \Big)=1.
        \end{align*}
        In order to estimate $\E \sup_{t\in[0,T]}\big[ |q^\varepsilon(t)-q(t)|^n+|p^\varepsilon(t)-p(t)|^n\big]$, we decompose this expectation into two terms using indicator functions, namely,
        \begin{align} \label{E_:RSing_1}
&\E \sup_{t\in[0,T]}\big[ |q^\varepsilon(t)-q(t)|^n+|p^\varepsilon(t)-p(t)|^n\big] \notag \\
&= \E \Big[\sup_{t\in[0,T]}\big[ |q^\varepsilon(t)-q(t)|^n+|p^\varepsilon(t)-p(t)|^n\big]\cdot  \mathbf{1}\{T\le \tau^R\wedge\tau^R_\varepsilon\}  \Big] \notag \\
&\qquad + \E \Big[\sup_{t\in[0,T]}\big[ |q^\varepsilon(t)-q(t)|^n+|p^\varepsilon(t)-p(t)|^n\big]\cdot  \mathbf{1}\{T> \tau^R\wedge\tau^R_\varepsilon\}  \Big] \notag\\
&\le \E \sup_{t\in[0,T]}\big[ |q^{\varepsilon,R}(t)-q^R(t)|^n+|p^{\varepsilon,R}(t)-p^R(t)|^n\big] \notag\\
&\qquad + \E \Big[\sup_{t\in[0,T]}\big[ |q^\varepsilon(t)-q(t)|^n+|p^\varepsilon(t)-p(t)|^n\big]\cdot  \mathbf{1}\{\tau^R \le T\}  \Big] \notag \\
&\qquad +\E \Big[\sup_{t\in[0,T]}\big[ |q^\varepsilon(t)-q(t)|^n+|p^\varepsilon(t)-p(t)|^n\big]\cdot  \mathbf{1}\{\tau^R_\varepsilon \le T\}  \Big] \notag \\
&=:F_1+F_2+F_3.
\end{align}
Following the proof of Proposition~\ref{P:RTruncConv} applied to the systems \eqref{E:RTruncSys:N=1} and \eqref{E:RTruncdLim:N=1}, we see that 
\begin{align}\label{E_:RSingI1}
F_1 =  \E \sup_{t\in[0,T]}\big[ |q^{\varepsilon,R}(t)-q^R(t)|^n+|p^{\varepsilon,R}(t)-p^R(t)|^n\big] \le \varepsilon^{{\frac{n}{2}}} C(T,R).
\end{align}
Concerning $F_2$, we apply Cauchy-Schwarz and Holder's inequalities to get
\begin{align}\label{E_:RSingI2}
\notag F_2 & = \E \Big[\sup_{t\in[0,T]}\big[ |q^\varepsilon(t)-q(t)|^n+|p^\varepsilon(t)-p(t)|^n\big]\cdot  \mathbf{1}\{\tau^R \le T\}  \Big]\\
\notag&\le C \Big|\E \sup_{t\in[0,T]}\big[ |q^\varepsilon(t)|^{2n}+|p^\varepsilon(t)|^{2n}+|q(t)|^{2n}+| p(t)|^{2n} \big] \Big|^{1/2}\\
&\qquad\times \Big|\P\Big( \sup_{t\in[0,T]}\big[ |q(t)|+|q(t)|^{-1}\big]\geq R  \Big)\Big|^{1/2}.
\end{align}
On the one hand, in view of Proposition~\ref{P:RSingUnifMomentBd}, it holds that
\begin{align*}
    \Big|\E \sup_{t\in[0,T]}\big[ |q^\varepsilon(t)|^{2n}+|p^\varepsilon(t)|^{2n}+|q(t)|^{2n}+| p(t)|^{2n} \big] \Big|\le C.
\end{align*}
On the other hand, in view of Lemma~\ref{L:U_G_p}, we have the following inclusions
\begin{align*}
& \Big\{ \sup_{t\in[0,T]}\big[ |q(t)|+|q(t)|^{-1}\big]\geq R  \Big\} \\
&\subseteq  \Big\{ \sup_{t\in[0,T]} |q(t)|\geq \frac{R}{2}  \Big\} \cup \Big\{ \sup_{t\in[0,T]}\big[-\log|q(t)|\big]\geq \log R-\log 2  \Big\} \\
&\subseteq  \Big\{ \sup_{t\in[0,T]} U(q(t))+G(q(t)) \ge \frac{c_G}{2C_G}\log R \Big\}.
\end{align*}
An application of Markov inequality together with Lemma~\ref{L:RUnifEnergy} shows that
\begin{align*}
\P\Big( \sup_{t\in[0,T]}\big[ |q(t)|+|q(t)|^{-1}\big]\geq R  \Big) \le \E \sup_{t\in[0,T]}\big[ U(q(t))+G(q(t))\big]\le \frac{C(T)}{\log R},
\end{align*}
whence
\begin{align*}
    F_2\le \frac{C(T)}{\sqrt{\log R}}.
\end{align*}
Turning to $F_3$, the uniform boundedness result in Proposition~\ref{P:RSingUnifMomentBd} implies that
\begin{align}\label{E_:RSingI3}
\notag F_3 & = \E \Big[\sup_{t\in[0,T]}\big[ |q^\varepsilon(t)-q(t)|^n+|p^\varepsilon(t)-p(t)|^n\big]\cdot  \mathbf{1}\{\tau^R_\varepsilon \le T\}  \Big]\\
\notag&\le C \Big|\E \sup_{t\in[0,T]}\big[ |q^\varepsilon(t)|^{2n}+|p^\varepsilon(t)|^{2n}+|q(t)|^{2n}+| p(t)|^{2n} \big] \Big|^{1/2}\\
\notag&\qquad\times \Big|\P\Big( \sup_{t\in[0,T]}\big[ |q^\varepsilon(t)|+|q^\varepsilon(t)|^{-1}\big]\geq R  \Big)\Big|^{1/2}\\
\notag& \le  C \Big|\E \sup_{t\in[0,T]}\big[ |q^\varepsilon(t)|^{2n}+|p^\varepsilon(t)|^{2n}+|q(t)|^{2n}+| p(t)|^{2n} \big] \Big|^{1/2}\\
\notag&\qquad \times \Big|\E \sup_{t\in[0,T]}\big[ U(q^\varepsilon(t))+G(q^\varepsilon(t))\big]\Big|^{1/2}\\
&\le \frac{C(T)}{\sqrt{\log R}}.
\end{align}
We collect~\eqref{E_:RSingI1},~\eqref{E_:RSingI2} and~\eqref{E_:RSingI3} and plug them into~\eqref{E_:RSing_1} to get
\begin{align*}
&\E \sup_{t\in[0,T]}\big[ |q^\varepsilon(t)-q(t)|^n+|p^\varepsilon(t)-p(t)|^n\big] \notag \\
&\le F_1+F_2+F_3\le \varepsilon^{{\frac{n}{2}}} C(T,R)+ \frac{C(T)}{\sqrt{\log R}}.
\end{align*}
{By sending $R$ to infinity and then taking $\varepsilon$ to zero}, we deduce the desired limit \eqref{lim:Newtonian:N=1:lambda>1}, thereby completing the proof of Theorem~\ref{thm:NewtonianLimit}, part (2).
    \end{proof}

Lastly, let us present the proof of Lemma \ref{L:VariationM_t}, which was a crucial ingredient of Proposition \ref{P:RSingUnifMomentBd} allowing us to conclude Theorem \ref{thm:NewtonianLimit}, part (2).

\begin{proof}[Proof of Lemma \ref{L:VariationM_t}]
        Recall $\Gamma_3$ is given by~\eqref{form:Gamma_2}. That is,
        \begin{align*}
            \Gamma_3(q^\varepsilon,p^\varepsilon) = \frac{1}{2} \varepsilon (U(q^\varepsilon) + G(q^\varepsilon))^2 + (U(q^\varepsilon) + G(q^\varepsilon)) \sqrt{1 + \varepsilon |p^\varepsilon|^2} + \frac{1}{2} |p^\varepsilon|^2.
        \end{align*}
        From~\eqref{form:M(t)}, we have
     \begin{align*}
        \ud\langle M_5(t)\rangle=&\Big|\frac{\varepsilon\lrbig{U(q^\varepsilon)+G(q^\varepsilon)}}{\sqrt{1+\varepsilon |p^\varepsilon|^2}}\sqrt{2D}p^\varepsilon+\sqrt{2D}p^\varepsilon\Big|^2\ud t\\
        \le& 4\lr{\frac{\varepsilon^2(U(q^\varepsilon)+G(q^\varepsilon))^2}{\sqrt{1+\varepsilon|p^\varepsilon|^2}}|p|^2+\sqrt{1+\varepsilon|p^\varepsilon|^2}|p^\varepsilon|^2}\ud t\\
        \le &4\lr{\varepsilon(U(q^\varepsilon)+G(q^\varepsilon))^2\sqrt{1+\varepsilon|p^\varepsilon|^2}+\sqrt{1+\varepsilon|p^\varepsilon|^2}|p^\varepsilon|^2}\ud t\\
        \le &8\sqrt{1+\varepsilon|p^\varepsilon|^2}\Gamma_3\ud t.
        \end{align*}
        So, it remains to to prove that $\sqrt{1+\varepsilon |p^\varepsilon|^2}\le \sqrt{2\Gamma_3}+1$. Indeed, this holds true for $\varepsilon\le 1$ thanks to the fact that
        \begin{align*}
            2\Gamma_3\ge 2\varepsilon\Gamma_3\ge \varepsilon|p^\varepsilon|^2.
        \end{align*}
        The proof if thus complete.
\end{proof}

\appendix

\section{Auxiliary results} \label{sec:appendix}

\begin{lemma}{\cite[Lemma A.1]{duong2024asymptotic}}\label{L:s+1times1}
For all $s\ge 0$ and any $\x=\lrbig{x_1,\dots,x_N}\in \mathcal{D}$, the following holds.
    \begin{equation}
    \sum_{i=1}^{N} \Big\langle \sum_{j \neq i} \frac{x_i - x_j}{|x_i - x_j|^{s+1}}, \sum_{l \neq i} \frac{x_i - x_l}{|x_i - x_l|} \Big\rangle \geq 2 \sum_{1 \leq i < j \leq N} \frac{1}{|x_i - x_j|^s}.
\end{equation}
\end{lemma}
\begin{lemma}{\cite[Lemma A.2]{duong2024asymptotic}}\label{L:s+1timess+1}
    For all $s\in[0,1]$ and any $\x=\lrbig{x_1,\dots,x_N}\in \mathcal{D}$, the followings hold:

    \textup{(i)} For all $s\ge 0$,
        \begin{equation}
            \sum_{i=1}^{N} \Big\langle \sum_{j \neq i} \frac{x_i - x_j}{|x_i - x_j|^{s+1}}, \sum_{\ell \neq i} \frac{x_i - x_\ell}{|x_i - x_\ell|^{s+1}} \Big\rangle \geq \frac{4}{N(N-1)^2} \sum_{1 \leq i < j \leq N} \frac{1}{|x_i - x_j|^{2s}}.
        \end{equation}

    \textup{(ii)} Furthermore, for $s\in [0,1]$
        \begin{equation}
            \sum_{i=1}^{N} \Big\langle \sum_{j \neq i} \frac{x_i - x_j}{|x_i - x_j|^{s+1}}, \sum_{\ell \neq i} \frac{x_i - x_\ell}{|x_i - x_\ell|^{s+1}} \Big\rangle \geq 2 \sum_{1 \leq i < j \leq N} \frac{1}{|x_i - x_j|^{2s}}.
        \end{equation}

\end{lemma}
\begin{lemma}{\cite[Lemma A.1]{duong2024trend}}\label{L:s+1timesGamma}
    For all $\gamma\in (0,1]$, $s\ge 0$ and $\q=\lrbig{q_1,\dots,q_N}\in \mathcal{D}$,
\begin{equation}
    \sum_{i=1}^{N} \Big\langle \sum_{ j \neq i}^{N} \frac{q_i - q_j}{|q_i - q_j|^{\gamma}}, \sum_{ k \neq i}^{N} \frac{q_i - q_k}{|q_i - q_k|^{s+1}} \Big\rangle \geq 2 \sum_{1 \leq i < j \leq N} \frac{1}{|q_i - q_j|^{s+\gamma-1}}.
\end{equation}
\end{lemma}

\section*{Acknowledgment} M. H. D is funded by an EPSRC Standard Grant EP/Y008561/1.
  \bibliographystyle{abbrv}
\bibliography{GLE-bib}
\end{document}